\documentclass[a4paper]{article}

\usepackage{amssymb}
\usepackage{amsthm}
\usepackage{latexsym}
\usepackage{amsmath}
\usepackage{color}

\usepackage{xr}

\usepackage{comment}
\excludecomment{discuss}
\excludecomment{discuss2}
\excludecomment{as}
\includecomment{arXiv}

\newif\ifcol
\colfalse
\ifcol
\newcommand{\colorr}{\color[rgb]{0.8,0,0}}
\newcommand{\colorg}{\color[rgb]{0,0.5,0}}
\newcommand{\colorb}{\color{black}}% {\color[rgb]{0,0,0.8}}
\newcommand{\colord}{\color{black}}% {\color[rgb]{0.8,0.3,0}}
\newcommand{\colorp}{\color{black}}% {\color[rgb]{0.8,0,0.8}}
\newcommand{\colorn}{\color{black}}{\color[rgb]{0.7,0.5,0}}
\newcommand{\colora}{\color{black}}{\color[rgb]{0,0,0.8}}
\newcommand{\colorc}{\color{black}}{\color[rgb]{0.5,0,0.5}}
\newcommand{\colore}{\color[rgb]{0.7,0.5,0}}

\else
\newcommand{\colorr}{\color{black}}% {{\color[rgb]{0.8,0,0}}
\newcommand{\colorg}{\color{black}}% {\color[rgb]{0,0.5,0}}
\newcommand{\colorb}{\color{black}}% {\color[rgb]{0,0,0.8}}
\newcommand{\colord}{\color{black}}% {\color[rgb]{0.8,0.3,0}}
\newcommand{\colorp}{\color{black}}
\newcommand{\colorn}{\color{black}}
\newcommand{\colora}{\color{black}}
\newcommand{\colorc}{\color{black}}
\newcommand{\colore}{\color{black}}
\fi

\newtheorem{definition}{Definition}[section]
\newtheorem{lemma}{Lemma}[section]
\newtheorem{proposition}{Proposition}[section]
\newtheorem{theorem}{Theorem}[section]
\newtheorem{remark}{Remark}[section]
\newtheorem{example}{Example}[section]

\makeatletter
\@addtoreset{equation}{section}
\makeatother

\newcommand{\EQ}[1]{\begin{equation}{#1}\end{equation}}
\newcommand{\EQQ}[1]{\begin{equation}{#1 \nonumber}\end{equation}}
\newcommand{\EQN}[1]{\begin{equation}\begin{split}{#1}\end{split}\end{equation}}
\newcommand{\EQNN}[1]{\begin{equation}\begin{split}{#1 \nonumber}\end{split}\end{equation}}
\newcommand{\MAT}[2]{\left(\begin{array}{#1} #2 \end{array} \right)} % 行列

\begin{document}

\title{Malliavin Calculus Techniques for Local Asymptotic Mixed Normality and Their Application to  Degenerate Diffusions}
\author{Masaaki Fukasawa$^*$ and Teppei Ogihara$^{**}$\\
$*$
\begin{small}Graduate School of Engineering Science, Osaka University, 1-3 Machikaneyama, Toyonaka, Osaka 560--8531, Japan\end{small}\\
$**$ 
\begin{small}Graduate School of Information Science and Technology, University of Tokyo, \end{small}\\
\begin{small}7-3-1 Hongo, Bunkyo-ku, Tokyo 113--8656, Japan \end{small}
}
\maketitle

\noindent
{\bf Abstract.}
We study sufficient conditions for a local asymptotic mixed normality
property of statistical models.
We develop a scheme with the $L^2$ regularity condition proposed by
Jeganathan [\textit{Sankhy\=a Ser. A} \textbf{44} (1982) 173--212] so that it is applicable to high-frequency observations of stochastic processes.
Moreover, by combining with Malliavin calculus techniques by Gobet [\textit{Bernoulli} \textbf{7} (2001) 899--912, 2001],
we introduce tractable sufficient conditions for smooth observations in the Malliavin sense,
which do not require Aronson-type estimates of the transition density function.
Our results, unlike those in the literature,
 can be applied even when the transition density function has zeros.
For an application, we show the local asymptotic mixed normality property
of degenerate (hypoelliptic) diffusion models under high-frequency
observations, in both complete and partial observation frameworks.
The former and the latter extend previous results
for elliptic diffusions and for integrated diffusions, respectively.
\\


\begin{discuss}

{\colorr 
＜式番号＞
\begin{itemize}
\item Theorem:\ref{L2regu-thm},{L2regu-thm};\ref{Malliavin-LAMN-thm2},{Malliavin-LAMN-thm2}; \ref{Malliavin-LAMN-thm},{Malliavin-LAMN-thm};
\ref{partial-LAMN}, {partial-LAMN}; \ref{degenerate-LAMN-thm}, {degenerate-LAMN-thm}
\item Section:\ref{L2regu-subsection}, {L2regu-subsection}; \ref{LAMN-Malliavin-subsection}, {LAMN-Malliavin-subsection};
\ref{integrated-diffusion-subsection}, {integrated-diffusion-subsection}; \ref{integrated-diffusion-LAMN-section}, {integrated-diffusion-LAMN-section};
\ref{Mmat-inverse-est-subsection}, {Mmat-inverse-est-subsection}
\item Proposition:\ref{delLogP-eq}, {delLogP-eq};\ref{delp-p-est-prop}, {delp-p-est-prop};\ref{delLogP-approx-prop},{delLogP-approx-prop}
\item Lemma:\ref{tildeK-est-lemma}, {tildeK-est-lemma}
\end{itemize}

$\lVert \cdot \rVert_{k,p}$の定義はNualart (1.37)にある．

$p$の$x_{j-1}$に関する連続性はおそらくなくてもいい。$p$のdefから可測性は言える。あとは$\sup_{x_{j-1}}$の評価があればいいか

eqnarrayは全てnonumberになっているからこのままでいい．
}
\end{discuss}

\section{Introduction}

In the study of statistical inference for parametric models, {\it
asymptotic efficiency} {\colorb plays a key role} when we consider {\colore the	} asymptotic
optimality of estimators.
This notion {\colore was} first studied for models {\colore that} satisfy {\it local asymptotic normality} (LAN);
H\'ajek \cite{haj70} showed the convolution theorem, and H\'ajek \cite{haj72} showed the minimax theorem under the LAN property.
Both theorems give different concepts of asymptotic efficiency.
%{\colorg Ibragimov and Hasminskii Chap I studies stat. model of i.i.d. Such a model satisfies LAN and eff of MLE and BE are discussed.
%Gobet (2002) shows LAN for discrete observations of ergodic diffusion with a growing observation window.
%On the other hand, when we consider a statistical model of diffusion processes with high-frequency observations on a fixed interval,
%LAN is not satisfied in general, and extended notion;}
For statistical models with {\colore the extended notion of} {\it local asymptotic mixed normality} (LAMN),
Jeganathan \cite{jeg82,jeg83} showed the convolution theorem and the minimax theorem.

Gobet \cite{gob01} showed the LAMN property for discretely observed diffusion processes on a fixed interval.
In {\colore that} model, the maximum-likelihood-type estimator proposed {\colore by} Genon-Catalot and Jacod~\cite{gen-jac93} is asymptotically efficient.
{\colore For} further results related to diffusion processes on a fixed interval, 
see Gloter and Jacod~\cite{glo-jac01a} (LAN for noisy observations of diffusion processes with deterministic diffusion coefficients),
Gloter and Gobet~\cite{glo-gob08} (LAMN for integrated diffusion processes),
Ogihara~\cite{ogi15} (LAMN for nonsynchronously observed diffusion processes),
and Ogihara~\cite{ogi18} (LAN for noisy, nonsynchronous observations of diffusion processes with deterministic diffusion coefficients).

In the model of discretely observed diffusion processes {\colore by} Gobet~\cite{gob01}, 
he initiated a scheme based on Malliavin calculus techniques to show the LAMN property.
{\colorc He introduced Malliavin calculus techniques to control {\colore the} asymptotics of log-likelihood ratios,
and his scheme is {\colore effective} for diffusion processes when the diffusion coefficient matrix is nondegenerate.}
In his scheme, it is crucial that the transition density functions of diffusion processes are estimated from below and above
by Gaussian density functions.
{\colorn Such estimates are {\colore known as} Aronson's estimate.}
Gloter and Gobet \cite{glo-gob08} {\colorb gave} {\colorn Aronson's estimate}, and consequently showed
the LAMN property for the one-dimensional integrated diffusion processes by using Gobet's scheme.
However, {\colorc the proof} of {\colorc Aronson's estimate} (Theorem 4) {\colore crucially} depends on 
the fact that the latent process is one-dimensional.

For {\colorn a diffusion model with the degenerate diffusion coefficient and} a multi-dimensional integrated diffusion model, it seems difficult
to obtain {\colorn Aronson's estimate} {\colorc in general}, and therefore we cannot apply Gobet's scheme.
On the other hand, Theorem 1 in Jeganathan~\cite{jeg82} introduced a scheme by using {\colore the} so-called $L^2$ regularity condition to show the LAMN property.
{\colore An advantage} of this scheme is that we do not need estimates for
the transition density functions.
{\colore However, the} results in \cite{jeg82} 
are not directly applicable to {\colore high-frequency} observations
that require a framework of triangular arrays.
Further, for integrated diffusions, 
following the idea in \cite{glo-gob08}, 
we need to consider a triangular array of expanding data blocks.

This paper studies {\colorn four} topics. First, we extend Theorem~1 in~\cite{jeg82} so that it can be applied to statistical {\colore models} with triangular array observations
appearing in the above diffusion models with high-frequency observations.
Second, we show that the new scheme based on the $L^2$ regularity condition can be applied under several conditions described via notions of Malliavin calculus.
The new scheme is {\colore highly} compatible with Gobet's scheme. {\colore Indeed}, the $L^2$ regularity condition is satisfied when observations are smooth in the Malliavin sense,
and the inverse of Malliavin matrix and its derivatives have moments (see (B1), (B2), and Theorem~\ref{Malliavin-LAMN-thm2}).
Moreover, if observations have {\colore a Euler--Maruyama} approximation, {\colore then the} sufficient conditions for the LAMN property is simplified (Theorem~\ref{Malliavin-LAMN-thm}).
Third, by using these schemes, we show the LAMN property for {\colorn diffusion processes with the degenerate diffusion coefficient (degenerate diffusion)
in which it is difficult to obtain Aronson-type {\colore estimates} in general.
{\colore Finally}, we deal with the LAMN property for partial observations of degenerate diffusion processes.}

{\colorn Our} new schemes can be applied to general statistical models without transition density estimates. 
In particular, they can be applied even when the transition density
function has zero points.
{\colorc The $L^2$ regularity condition is related to differences {\colore in the roots} of transition density functions,
and it is not easily applied when the transition density functions have zero points (see (\ref{L2-regu-equiv})).
However, we will see in Section~\ref{LAMN-Malliavin-subsection} that if {\colore the} observations are smooth in {\colore the} Malliavin's sense 
and the Malliavin matrix is nondegenerate, {\colore then} we can apply the $L^2$ regularity condition even when the transition density has zero points.
{\colore Consequently}, this scheme enables us to study the LAMN property for several statistical models in Wiener space.
First, this scheme {\colore allows a} simplified proof of the results in Gobet~\cite{gob01}.
Moreover, this scheme yields two interesting results.
The first one is an extension of the results in Gobet~\cite{gob01} to a wider class including degenerate diffusion processes;
we emphasis that this is achieved {\colore because} we do not rely on Aronson-type estimates.
The second one is an extension of the LAMN property for one-dimensional integrated diffusion processes 
in Gloter and Gobet~\cite{glo-gob08} to the multi-dimensional case.
We deal with the integrated diffusion process model in {\colore the} general framework of partial observations for degenerate diffusion processes.
We {\colore find} that efficient asymptotic variance {\colore is the same for an} integrated diffusion process model and for {\colore a} diffusion process model,
and {\colore that} they are exactly twice as large as the statistical model of both observations (see Remark~\ref{integrated-rem}). 
{\colore Because} our scheme does not require transition density estimates, 
{\colore we expect these ideas to be useful also} for jump-diffusion process models or L\'evy driven {\colore stochastic differential equation} models.
However, we left this {\colore for} future work.}

{\colorb Our study of integrated diffusion models is motivated by 
{\colore experimental observations of single molecules}
(see e.g. Li et al.~\cite{liEtAl10}), {\colore behind which are Langevin-type} molecular dynamics
\begin{equation*}
\ddot{Y} = b(\dot{Y},Y) + {\colore a(\dot{Y})}\dot{W}.
\end{equation*}
Here $Y$ represents the position of a molecule (or a particle) and $\dot{W}$ is {\colore white}
noise.
When $a = 0$ this reduces to the {\colore Newtonian} equation of classical dynamics.
The system can be written as an integrated diffusion
\begin{equation}\label{Langevin-eq}
\begin{split}
 & dY_t = X_t dt, \\
& dX_t = b(X_t,Y_t)dt + {\colore a(X_t)}dW_t.
 \end{split}
\end{equation}
Our LAMN property enables us to discuss optimality in estimating the
coefficient $a$ 
based on high-frequency observations of the position $Y$.

}

The rest of this paper is organized as follows. In Section~\ref{main-results-section}, we introduce {\colore our} main results, {\colore namely},
the extended scheme using the $L^2$ regularity condition, the scheme via Malliavin calculus techniques, 
and the LAMN property of {\colorc degenerate} diffusion processes.
Section~\ref{Mcal-section} contains details of Malliavin calculus techniques. We combine the extended scheme of {\colore the} $L^2$ regularity condition
with {\colore the} approaches of Gobet~\cite{gob01} and Gloter and Gobet~\cite{glo-gob08}.
%In Section~\ref{integrated-diffusion-LAMN-section}, the LAMN property
%for multi-dimensional integrated diffusion processes {\colorb is proved}.
%The new scheme of Section~\ref{Mcal-section} is applied to an augmented model which is obtained by adding observations.
%After that, we show the log-likelihood ratio of the original model is approximated by that of the augmented model.

\section{Main results}\label{main-results-section}

\subsection{The LAMN {\colore property} via the $L^2$ regularity condition}\label{L2regu-subsection}

In this subsection, we extend Theorem~1 in Jeganathan~\cite{jeg82} to statistical models of triangular array observations
so that it can be applied to high-frequency observations of stochastic processes.

Let $\{P_{\theta,n}\}_{\theta\in\Theta}$ be a family of probability
 measures defined on $({\colorc \mathfrak{R}_n},\mathcal{F}_n)$,
where $\Theta$ is an open subset of $\mathbb{R}^d$. 
We often regard a $p$-dimensional vector $v$ as a $p\times 1$ matrix.
{\colorn $I_k$ denotes the unit matrix of size $k$ for $k\in\mathbb{N}$.}

{\colore
\begin{description}
\item[{\colorc Condition} (L).] The following two conditions are satisfied for $\{P_{\theta,n}\}_{\theta\in\Theta}$.
\begin{enumerate}
\item 
There exists a sequence $\{V_n(\theta_0)\}$ of $\mathcal{F}_n$-measurable $d$-dimensional vectors and a sequence $\{T_n(\theta_0)\}$
of $\mathcal{F}_n$-measurable $d\times d$ symmetric matrices such that 
\begin{equation}\label{Tn-nd}
P_{\theta_0,n}[T_n(\theta_0) {\rm \ is \  nonnegative \ definite}]=1
\end{equation}
for any $n\in\mathbb{N}$, and
\begin{equation}\label{log-like-est2}
\log \frac{dP_{\theta_0+r_nh,n}}{dP_{\theta_0,n}}-h^{\top}V_n(\theta_0)+\frac{1}{2}h^{\top}T_n(\theta_0)h\to 0
\end{equation}
in $P_{\theta_0,n}$-probability for any $h\in\mathbb{R}^d$, where
$\{r_n\}$ is a sequence of positive definite matrices
and $\top$ denotes the transpose operator for matrices.
\item {\colore There} exists an almost surely nonnegative definite random matrix $T(\theta_0)$ such that
\begin{equation*}
\mathcal{L}(V_n(\theta_0),T_n(\theta_0)|P_{\theta_0,n})\to \mathcal{L}(T^{1/2}(\theta_0)W,T(\theta_0)),
\end{equation*}
where $W$ is a $d$-dimensional standard normal random variable independent of $T(\theta_0)$.
\end{enumerate}
\end{description}
}

The following definition of the LAMN property is Definition~1 in~\cite{jeg82}.
\begin{definition}\label{LAMN-def}
The sequence of the families $\{P_{\theta,n}\}_{\theta\in\Theta} \ (n\in\mathbb{N})$
satisfies the LAMN condition at $\theta=\theta_0\in\Theta$ if {\colore Condition (L) is satisfied, 
$P_{\theta_0,n}[T_n(\theta_0) {\rm \ is \  positive \ definite}]=1$
for any $n\in\mathbb{N}$,
and $T(\theta_0)$ is positive definite almost surely.}
\end{definition}

{\colore For proving} the LAMN property for diffusion processes using a localization technique {\colore such as} Lemma 4.1 in~\cite{gob01},
Condition~(L) is useful {\colore because} (L) for the localized model often implies (L) for the original model.
See the proofs of Theorems~\ref{degenerate-LAMN-thm} and~\ref{partial-LAMN-thm} for the details. 
\begin{remark}\label{Gamma-pd-rem}
When Condition~(L) is satisfied and $T(\theta_0)$ is positive definite almost surely, by setting
\begin{equation}
{\colore \tilde{T}_n(\theta_0)=T_n(\theta_0)1_{\{T_n(\theta_0) \ {\rm is \ p.d.}\}}+I_d1_{\{T_n(\theta_0) \ {\rm is \ not \ p.d.}\}},}
\end{equation}
the LAMN property holds with $\tilde{T}_n(\theta_0)$ and {\colore $V_n(\theta_0)$}.
\end{remark}
\begin{discuss}
{\colorr $\tilde{T}_n\approx T_n$よりLAMNの1が成り立ち，
\begin{equation*}
(W_n,\tilde{T}_n)\approx (T_n^{-1/2}1_{\{T_n(\theta_0) \ {\rm is \ p.d.}\}}T_n^{1/2}W_n,T_n)\to (T^{-1/2}T^{1/2}W,T)=(W,T)
\end{equation*}
よりOK.
}
\end{discuss}

Let $(m_n)_{n=1}^\infty$ be a sequence of positive integers.
Let $\{{\colorc \mathfrak{R}_{n,j}}\}_{j=1}^{m_n}$ be a sequence of complete, separable metric spaces,
 and let $\Theta$ be an open subset of $\mathbb{R}^d$.
Let {\colorc $\mathfrak{R}_n=\mathfrak{R}_{n,1}\times \cdots \times \mathfrak{R}_{n,m_n}$}.
We consider statistical experiments {\colorc $(\mathfrak{R}_n, \mathcal{B}(\mathfrak{R}_n), \{P_{\theta,n}\}_{\theta\in\Theta})$}.
%極限モデルを
%$(\mathcal{X},\mathcal{F},\{P_{\theta}\}_{\theta\in\Theta})$とおく．
{\colorb Let $X_j = X_{n,j} : {\colorc \mathfrak{R}_n \to \mathfrak{R}_{n,j}}$ be the natural
projection, $\bar{X}_j = \bar{X}_{n,j} = (X_1, \dots, X_j)$, and
$\mathcal{F}_j = \mathcal{F}_{n,j}=\sigma(\bar{X}_j)$ 
for $0\leq j\leq m_n$.
Suppose that there exists a $\sigma$-finite measure $\mu_j = \mu_{n,j}$
on {\colorc $\mathfrak{R}_{n,j}$}
such that $P_{\theta,n}(X_1 \in \cdot) \ll \mu_1$ and 
$P_{\theta,n}(X_j \in \cdot | \bar{X}_{j-1} = \bar{x}_{j-1})\ll \mu_j$ for 
all $\bar{x}_{j-1} \in {\colorc \mathfrak{R}_{n,1} \times \dots \times \mathfrak{R}_{n,j-1}}$,
$2\leq j\leq m_n$.
}
\begin{discuss}
{\colorr $P_{\theta,n}|_{\mathcal{F}_k}\ll \mu_{n,1}\otimes \cdots \otimes \mu_{n,k}$ for $1\leq k\leq m_n$なら
\begin{equation*}
F_k(\bar{x}_k)=\frac{dP_{\theta,n}|_{\mathcal{F}_k}}{d(\mu_{n,1}\otimes \cdots \otimes \mu_{n,k})}
\end{equation*}
として
\begin{equation*}
P_{\theta,n}(x_j\in A|\bar{x}_{j-1})=\int_A\frac{F_j(\bar{x}_j)}{F_{j-1}(\bar{x}_{j-1})}\mu_j(dx_j)
\end{equation*}
か．
}
\end{discuss}
Let $E_\theta = E_{\theta,n}$ denote the expectation with respect to
$P_{\theta,n}$, and let $p_j = p_{n,j}$ be the conditional density functions defined by
\begin{equation*}
p_1(\theta)= \frac{dP_{\theta,n}(X_1\in
 \cdot)}{d\mu_1} : {\colorc \mathfrak{R}_{n,1}} \to \mathbb{R}, 
\quad p_j(\theta)= 
\frac{dP_{\theta,n}(X_j\in \cdot | \bar{X}_{j-1})}{d\mu_{j}} : 
{\colorc \mathfrak{R}_{n,j}} \to \mathbb{R}
\end{equation*}
for $2\leq j\leq m_n$.
Then we can see that for $g : {\colorc \mathfrak{R}_{n,1} \times \dots \mathfrak{R}_{n,j}}
\to \mathbb{R}$,
\begin{equation}\label{condE-eq}
\int_{{\colorc \mathfrak{R}_{n,j}}} p_j(\theta)g(\bar{X}_{j-1}, \cdot)d\mu_{j}
=E_{\theta}[g(\bar{X}_{j-1}, X_{j})|\mathcal{F}_{j-1}].
\end{equation}

\begin{discuss}
{\colorr ∵任意のBorel functions $g,h$に対して
\begin{equation*}
E_{\theta}[\int p_j(\theta)g(x_j,\bar{x}_{j-1})d\mu_jh(\bar{x}_{j-1})]=E_{\theta}[g(x_j,\bar{x}_{j-1})h(\bar{x}_{j-1})]
\end{equation*}
だから．
}
\end{discuss}

\begin{description}
\item[Assumption (A1).] There {\colore is a} $d\times d$ positive definite matrix $r_n$
and {\colore measurable} functions 
$$ {\colore \dot{\xi}_{n,j}(\theta_0, \cdot ): \mathfrak{R}_{n,1}\times \dots \times \mathfrak{R}_{n,j} \to \mathbb{R}^d} $$
 such that for every $h\in\mathbb{R}^d$,
\begin{equation}\label{L2-regu-eq}
\sum_{j=1}^{m_n}E_{\theta_0}\bigg[\int
 {\colorb [\xi_{n,j}}(\theta_0,h)-\frac{1}{2}h^{\top}{\colorb
 r_n}{\colorb \dot{\xi}_{j}}(\theta_0)]^2d {\colorb \mu_{j}}\bigg]\to 0
\end{equation}
as $n\to\infty$, where
\begin{equation*}
{\colorb\xi_{n,j}}(\theta_0,h)=\sqrt{p_j(\theta_0+{\colorb r_nh})}-\sqrt{p_j(\theta_0)}
\end{equation*}
and
  $${\colore \dot{\xi}_{j}(\theta_0)= \dot{\xi}_{n,j}(\theta_0, \bar{X}_{j-1},\cdot ): 
    \mathfrak{R}_{n,j} \to \mathbb{R}^d.} $$
\end{description}
Condition~(A1) {\colore is} the $L^2$ regularity {\colore condition}.

For a vector $x=(x_1,\cdots, x_k)$, we denote
$\partial_x^l=(\frac{\partial^l}{\partial x_{i_1}\cdots \partial
x_{i_l}})_{i_1,\cdots,i_l=1}^k$.
If $p_j$ is smooth with respect to $\theta$ and $p_j\neq 0$, then the log-likelihood ratio is rewritten as
\begin{equation*}
\log \frac{dP_{\theta',n}}{dP_{\theta,n}}=\sum_{j=1}^{m_n}\log\frac{p_j(\theta')}{p_j(\theta)}
=\sum_{j=1}^{m_n}\int^1_0\frac{\partial_{\theta}p_j}{p_j}(t\theta'+(1-t)\theta)dt(\theta'-\theta).
\end{equation*}
To show the LAMN property, we {\colore must} identify the limit distribution of this function under $P_{\theta,n}$.
{\colore Doing so} requires estimates for density ratios with different probability measures, which are not easy {\colore to obtain} for stochastic processes in general.
Gobet~\cite{gob01} {\colorb dealt} with this problem for discretely observed diffusion processes 
by using estimates from below and above by Gaussian density functions and show the LAMN property of that model.
\begin{discuss}
{\colorr 
\begin{eqnarray}
P_{\theta,n}(A_1\times \cdots \times A_{m_n})&=&E_\theta[E[1_{A_{m_n}}|\bar{x}_{m_n-1}]1_{A_1\times \cdots A_{m_n}}]
=E_\theta[\int 1_{A_{m_n}}p_{m_n}(\theta)d\mu_{n,m_n}1_{A_1\times \cdots A_{m_n-1}}] \nonumber \\
&=&E_\theta[\int \int 1_{A_{m_n-1}\times A_{m_n}}p_{m_n-1}p_{m_n}d\mu_{n,m_n}d\mu_{n,m_n-1}1_{A_1\times \cdots \times A_{m_n-2}}]=\cdots  \nonumber \\
&=&\int \cdots \int 1_A\prod_jp_j\prod_jd\mu_j. \nonumber
\end{eqnarray}
よって
\begin{equation*}
\frac{dP_{\theta,n}}{d\otimes_j\mu_j}=\prod_jp_j.
\end{equation*}
ここでは条件付き確率が$p_j$でその積でdensityが書けるのは一般に成立．あとはtransition densityの表現で問題が起きないかどうか．
}
\end{discuss}

On the other hand, if $p_j\in C^2(\Theta)$ and $p_j(\theta_t)\neq 0$ for any $t\in [0,1]$ $\mu_j$-a.e.
for $\theta_t=\theta_0+t{\colorb r_n}h$, then by setting
${\colorb\dot{\xi}_{j} (\theta_0)}=\partial_{\theta}p_j(\theta_0)p_j(\theta_0)^{-1/2}$ we obtain
\EQN{\label{L2-regu-equiv}
&\int
 [{\colorb\xi_{n,j}}(\theta_0,h)-\frac{1}{2}h^{\top}{\colorb r_n \dot{\xi}_{j} (\theta_0)}]^2d{\colorb
 \mu_{j}} \\
&\quad =\int \bigg[h^{\top}{\colorb
 r_n}\int^1_0\frac{\partial_{\theta}p_j(\theta_t)}{2\sqrt{p_j(\theta_t)}}dt-\frac{1}{2}h^{\top}{\colorb
 r_n}\frac{\partial_{\theta}p_j(\theta_0)}{\sqrt{p_j(\theta_0)}}\bigg]^2d{\colorb
 \mu_{j}} \\
&\quad =\int \bigg[\frac{h^{\top}{\colorb
 r_n}}{2}\int^1_0\int^t_0\bigg(\frac{\partial_{\theta}^2p_j(\theta_s)}{\sqrt{p_j(\theta_s)}}-{\colore \frac{\partial_{\theta}p_j(\partial_{\theta}p_j)^\top}{2p_j^{3/2}}(\theta_s)}\bigg)dsdt{\colorb
 r_n}h\bigg]^2d{\colorb
 \mu_{j}} \\
&\quad \leq \frac{1}{4}\sup_{0\leq s\leq
 1}E_{\theta_s}\bigg[\bigg\{h^{\top}{\colorb
 r_n}\bigg(\frac{\partial_{\theta}^2p_j(\theta_s)}{p_j(\theta_s)}-{\colore \frac{\partial_{\theta}p_j(\partial_{\theta}p_j)^\top}{2p_j^2}(\theta_s)}\bigg){\colorb
 r_n}h\bigg\}^2\bigg|{\colorb\mathcal{F}_{j-1}}\bigg].
}
\begin{discuss}
{\colorr Ibragimov Hasminskiiの教科書のII.3のRemark 3.2でL2 regularityを上のような二階微分の条件に変換している．
$p$は$(U_0^{n,\theta},\cdots, U_{k_j-k_{j-1}}^{n,\theta})$のdensity. 
}
\end{discuss}
In the right-hand side of the above inequality, 
{\colore the value $\theta_s$ of the parameter is the same for the probability measure of expectation and $p_j$ in the integrand, and therefore we do not need estimates for the transition density ratios.}

Thus, a scheme with the $L^2$ regularity condition does not require estimates for {\colore the} transition density function.
This is a big advantage, and this scheme {\colore can be} applicable to multi-dimensional integrated diffusion processes
 {\colore and  degenerate} diffusion processes,
where it is difficult to obtain estimates for transition density functions.

Define
{\colorb
\begin{equation*}
\eta_j= \left(\frac{\dot{\xi}_{j}(\theta_0)}{\sqrt{p_j(\theta_0)}} 1_{\{p_j(\theta_0)
 \neq 0\}}\right) (X_j).
\end{equation*}
}
\begin{description}
\item[Assumption (A2).] 
$E_{\theta_0}[|\eta_j|^2|\mathcal{F}_{j-1}] < \infty $ and 
$E_{\theta_0}[\eta_j|\mathcal{F}_{j-1}]=0 $, $P_{\theta_0,n}$-almost
	   surely 
for every $j\geq 1$. 

\item[Assumption (A3).] For every $\epsilon >0$ and $h\in \mathbb{R}^d$,
\begin{equation*}
\sum_{j=1}^{m_n}E_{\theta_0}[|h^{\top}{\colorb r_n}\eta_j|^21_{\{|h^{\top}{\colorb r_n}\eta_j|>\epsilon\}}]\to 0.
\end{equation*}
\item[Assumption (A4).] For every $h\in\mathbb{R}^d$, there exists a constant $K>0$ such that
\begin{equation*}
\sup_{n\geq 1}\sum_{j=1}^{m_n}E_{\theta_0}[|h^{\top}{\colorb r_n}\eta_j|^2]\leq K.
\end{equation*}
\end{description} 
Let 
\begin{equation}\label{TnWn-def}
{\colorp T_n}={\colorb
r_n}\sum_{j=1}^{m_n}E_{\theta_0}[\eta_j\eta_j^{\top}|{\colorb \mathcal{F}_{j-1}}]{\colorb r_n}
\quad {\rm and} \quad {\colorp V_n}={\colorb r_n}\sum_{j=1}^{m_n}\eta_j.
\end{equation}
\begin{description}
\item[Assumption (A5).] 
There exists a random $d\times d$ symmetric matrix $T$ such that
$P[T \ {\rm is \ {\colorp n.d.}}]=1$ and
\begin{equation*}
\mathcal{L}(({\colorp V_n}, T_n)|P_{\theta_0,n}) \to  \mathcal{L}(T^{1/2}W,T),
\end{equation*}
where $W \sim N(0,I_d)\perp T$. 
{\colorp \item[Assumption (P).] $T$ in (A5) is positive definite almost surely.}
\end{description} 
\begin{discuss}
{\colorr (A2)を使っているところは今のところJeg82 Lem 1とLem7の証明．ちなみにLem1はLem 5,6でも使う．}
\end{discuss}

Conditions (A1)--(A4) correspond to (2.A.1), (2.A.2), (2.A.4), {\colore and} (2.A.5) in~\cite{jeg82}, respectively.
{\colore Condition (A5) ensures Point~2 of Condition (L)}. For the sequential observations in~\cite{jeg82},
convergence of ${\colorb r_n}\sum_{j=1}^{m_n}\eta_j$ always holds by {\colore virtue} of Hall~\cite{hal77} (see (2.3) in~\cite{jeg82}).
However, the results of~\cite{hal77} cannot be applied to the triangular array observations,
so we {\colore instead} assume (A5) for our scheme.
{\colorp To check (A5), the results in Sweeting~\cite{swe80} are useful. Moreover, it}
is {\colorb not} difficult to check (A5) for statistical models of discretely observed diffusion processes
{\colorp by using a martingale central limit theorem}.
See, for example, {\colorp Theorems~\ref{degenerate-LAMN-thm} and~\ref{partial-LAMN-thm} and their proofs}.

\begin{discuss}
{\colorr Sweetingは$\sum_j\partial_\theta(\partial_\theta p_j/p_j)\overset{P}\to \Gamma_\theta$ for any $\theta$ならだいたいOKか}
\end{discuss}

\begin{theorem}\label{L2regu-thm}
Assume $(A1)$--$(A5)$. {\colorp Then (L) holds true with $T_n$ and $V_n$ in~(\ref{TnWn-def}) for the family $\{P_{\theta,n}\}_{\theta,n}$ of probability measures.
If further (P) is satisfied, {\colore then} $\{P_{\theta,n}\}_{\theta,n}$ satisfies the LAMN condition at $\theta=\theta_0$
with {\colore $\tilde{T}_n$} in Remark~\ref{Gamma-pd-rem}.}
\end{theorem}

\begin{arXiv}
A proof is given in the appendix.
\end{arXiv}
\begin{as}
{\colorc The proof is {\colore given in} Section~\ref{L2regu-proof-section} {\colore of} the supplementary {\colore material}~\cite{fuk-ogi20s}.}
\end{as}

\begin{remark}
We assumed that ${\colorb r_n}$ is positive definite because {\colore this assumption is made} in the definition of the LAMN property in Jeganathan~\cite{jeg82} (Definition~1).
However, we can see that Theorem~\ref{L2regu-thm} holds even {\colore if} ${\colorb r_n}$ is a nondegenerate {\colore asymmetric} matrix.
In that case, {\colore even} though the assumptions of convolution theorem (Corollary 1) in~\cite{jeg82} are not satisfied,
the convolution theorem in H\'ajek~\cite{haj70} is satisfied when {\it local asymptotic normality} is satisfied
({\colore i.e.,} $T$ in (A5) is non-random) and the operator norm of ${\colorb r_n}{\colorb r_n^\top}$ converges to zero.
\end{remark}

\subsection{The LAMN {\colore property} via Malliavin calculus techniques}\label{LAMN-Malliavin-subsection}

\begin{discuss}
{\colorr adaptive推定をするわけではないから$\Theta_1\times \Theta_2$型でやるかどうかは関係ない}
\end{discuss}

Gobet~\cite{gob01,gob02} used Malliavin calculus techniques to show the LAMN property for discretely observed diffusion processes.
Gloter and Gobet~\cite{glo-gob08} developed Gobet's scheme {\colore into a} more general one and showed the LAMN property for a one-dimensional integrated diffusion process.
{\colorb These approaches require estimates for transition density
functions by Gauss density functions, {\colore thereby} hampers
multi-dimensional extension {\colore of their results}. Our alternative approach
 introduces} tractable sufficient conditions to show the LAMN property for smooth observations in the Malliavin sense,
by combining with a scheme with the $L^2$ regularity condition in Section~\ref{L2regu-subsection}.  
In particular, we can show the $L^2$ regularity condition under (B1) and (B2), 
which are related to {\colore the} smoothness of observations and estimates for the inverse of {\colore the} Malliavin matrix (Theorem~\ref{Malliavin-LAMN-thm2}).
If further, observations have {\colore a} Gaussian approximation like {\colore the Euler--Maruyama} approximation, {\colore then the} sufficient conditions for the LAMN property are simplified {\colore as} in Theorem~\ref{Malliavin-LAMN-thm}.
\begin{discuss}
{\colorr Jump Malliavinを含むJacodの教科書のMalliavin operatorを使ったアプローチは$\partial p/p$の表現は得られそうだが，
その後の証明を丸々焼き直す必要があるので今回はやめておく．
integration-by-parts-settingを使った議論も同じ．
非有界作用素の理論はB→Hの作用素をあまり扱っていないから辞めといた方がよさそう．}
\end{discuss}

\begin{discuss}
{\colorr Let $(B,H,\mu)$ be an abstract Wiener space.}
\end{discuss}

We assume that {\colorp $\Theta$ is convex and that} $m_n\to \infty$ as $n\to \infty$.
\begin{discuss}
{\colorr Proposition 3.1で$C_0^\infty$の可分性を言うためにconvexが必要か．不要かもしれないが，$\theta_0$固定のLAMNでは$\Theta$は凸で問題ない}
\end{discuss}
Let $(\epsilon_n)_{n=1}^\infty$ be a sequence of positive numbers and $(\Omega,\mathcal{F},P)$ be a probability space. 
Let $({\colorb k^n_j})_{j=0}^{m_n}$ be an increasing sequence of
nonnegative integers such that ${\colorb k^n_0}=0$.
{\colore Hereinafter, we abbreviate $k^n_j$ as simply $k_j$.}
Let $N_n=k_{m_n}$ and $X_j^{n,\theta}$ be an $\mathbb{R}^{k_j-k_{j-1}}$-valued random variable on $(\Omega,\mathcal{F},P)$ for $1\leq j\leq m_n$.
Let $P_{\theta,n}$ be the induced probability measure by $\{X_j^{n,\theta}\}_{j=1}^{m_n}$ on $(\mathbb{R}^{N_n},\mathcal{B}(\mathbb{R}^{N_n}))$ and $\mathcal{F}_{j,n}=\{A\times \mathbb{R}^{N_n-k_j}|A\in \mathcal{B}(\mathbb{R}^{k_j})\}\subset \mathcal{B}(\mathbb{R}^{N_n})$.
\begin{discuss}
{\colorr (B4)を言うために$(\Omega,\mathcal{F},P)$の存在が必要．$P_{\theta,n}$はデータだけの空間にしたいからそこで$\mathcal{L}(\tilde{F})$は定義したくない．}
\end{discuss}
For each $1\leq j\leq m_n$, we adopt the notation of Nualart~\cite{nua06}. {\colore Specifically,} let $H_j$ be a real separable Hilbert space and
$W_j=\{W_j(h),h\in H_j\}$ be an isonormal Gaussian process defined on a complete probability space $(\Omega_j,\mathcal{G}_j,Q_j)$.
We assume that $\mathcal{G}_j$ is generated by $W_j$. {\colore Even though} these
objects possibly depend on $n$, we omit the dependence in our notation.
\begin{discuss}
{\colorr 元の定義：For each $1\leq j\leq m_n$, let $H_j$ be a real separable Hilbert space, $B_j$ be a real separable Banach space,
$\mathcal{B}_j$ be the Borel $\sigma$-field on $B_j$ and $Q_j$ be a probability measure such that
$(B_j,\mathcal{B}_j,Q_j)$ be a complete probability space.
Let $W_j=\{W_j(h),h\in H_j\}$ be isonormal Gaussian process defined on $(B_j,\mathcal{B}_j,Q_j)$,
and we assume that $\mathcal{B}_j$ is generated by $W_j$. $B_j=L^p(\Omega,\mathcal{G},Q)$ということ．
}
\end{discuss}
Let $\delta_j$ be {\colore the Hitsuda--Skorokhod integral (the divergence operator), $D_j$ be the Malliavin--Shigekawa} derivative, 
and $\mathcal{S}_j=\{f(W_j(h_1),\cdots,W_j(h_k));k\in\mathbb{N},h_i\in H_j \ (i=1,\cdots,k), f\in C^\infty(\mathbb{R}^k)\}$.
For $k\in\mathbb{Z}_+$ and $p\geq 1$, $\lVert \cdot \rVert_{k,p}$ denotes the operator on $\mathcal{S}_j$ defined by
\begin{equation*}
\lVert F\rVert_{k,p}=\bigg[E_j[|F|^p]+\sum_{l=1}^kE_j[\lVert D_j^lF\rVert_{H_j^{\otimes l}}^p]\bigg]^{1/p},
\end{equation*}
where $E_j$ denotes the expectation with respect to $Q_j$. 
\begin{discuss}
{\colorr ミンコフスキーの不等式から劣加法性が成立し，ノルムになる．ソボレフ空間のノルムの定義の仕方と一致．}
\end{discuss}
Let $\mathbb{D}^{k,p}_j$ be the completion of $\mathcal{S}_j$ with respect to the distance $d(F,G):=\lVert F-G \rVert_{k,p}$.
For general properties of $W_j$, $D_j$, and $\delta_j$, {\colore see} Nualart~\cite{nua06}.
Let {\colore $F_{n,\theta,j,\bar{x}_{j-1}}$} be an $\mathbb{R}^{k_j-k_{j-1}}$-valued random variable on $(\Omega_j,\mathcal{G}_j)$
such that
$Q_jF_{n,\theta,j,\bar{x}_{j-1}}^{-1}={\colorb P}(X_j^{n,\theta}\in
\cdot |\bar{X}_{j-1}^{n,\theta}=\bar{x}_{j-1})$, where
$\bar{X}_{j-1}^{n,\theta}=\{ {\colorb X_l^{n,\theta}}\}_{l=1}^{j-1}$ {\colorb and $\bar{x}_{j-1}\in \mathbb{R}^{k_{j-1}}$}.
We assume that $F_{n,\theta,j,\bar{x}_{j-1}}$ is Fr\'echet differentiable with respect to $\theta$ on $L^p(\Omega_j)$ for any $p>1$
and denote its derivative by $\partial_\theta F_{n,\theta,j,\bar{x}_{j-1}}=(\partial_{\theta_1} F_{n,\theta,j,\bar{x}_{j-1}},\cdots, \partial_{\theta_d} F_{n,\theta,j,\bar{x}_{j-1}})^\top$.
We often omit the parameter $\bar{x}_{j-1}$ in $F_{n,\theta,j,\bar{x}_{j-1}}$ and write $F_{n,\theta,j}$.
Let $C^\infty_o$ denote the space of all $C^\infty$ functions with compact {\colore support}. 
For a matrix $A$, we denote its {\colore element $(i,j)$} by $[A]_{ij}$. 
{\colore Similarly, we denote by $[V]_l$ the $l$-th element of a vector $V$.}
Let $\bar{k}_n=\max_j(k_j-k_{j-1})$.
\begin{discuss}
{\colorr 積分観測では$m_n\approx n$}
\end{discuss}

We assume the following conditions.
\begin{description}
\item[Assumption (B1).]  $\partial_\theta^l{\colore [F_{n,\theta,j}]_i}\in \cap_{p>1}\mathbb{D}^{4-l,p}_j$ for any $n,\theta,j,i,0\leq l\leq 3$, and $\sup_{n,i,j,\bar{x}_{j-1},\theta}\lVert \partial_\theta^l{\colore [F_{n,\theta,j}]_i}\rVert_{4-l,p}<\infty$ for $p>1$.
\begin{discuss}
{\colorr $K^{-1}\in \mathbb{D}^{1,2}$等のために$F$の$p$はかなり大きくとる必要があるから任意の$p$でよい．(B3)の$p$も同じ}
\end{discuss}
\item[Assumption (B2).]  The matrix $K_j(\theta)=(\langle {\colore D_j[F_{n,\theta,j}]_k,D_j[F_{n,\theta,j}]_l}\rangle_{H_j})_{k,l}$ is invertible almost surely for any $j,\bar{x}_{j-1}$ and $\theta$, and there exists a constant $\alpha_n\geq 1$ such that
\begin{equation*}
\sup_{i,l,j,\bar{x}_{j-1},\theta}\lVert [K_j^{-1}(\theta)]_{il}\rVert_{2,8}\leq \alpha_n, \quad {\rm and} \quad \epsilon_n^2\bar{k}_n^4\sqrt{m_n}\alpha_n^2\to 0
\end{equation*}
as $n\to\infty$.
\end{description}
\begin{discuss}
{\colorr 
$\lVert \cdot \rVert_{k,p}$の$p$に関する単調性は，$q>p$の時
\begin{eqnarray}
\lVert F\rVert_{k,p}\leq \bigg[E_j[|F|^q]^{p/q}+\sum_{l=1}^kE_j[\lVert D_j^l F\rVert_{H^{\otimes l}}^q]^{p/q}\bigg]^{1/p}
\leq \lVert F\rVert_{k,q}. \nonumber
\end{eqnarray}
($x,y\geq 0$, $r\in [0, 1]$に対して，$(x+y)^r\leq x^r+y^r)$)
よりわかるので，(B2)の定義は問題なし．

$(\det K_j)^{-1}$の評価より$[K_j^{-1}]_{il}$の評価が以下では必要なのでそれを仮定する．$(\det K_j)^{-1}$評価からも得られるが余因子が絡んできて行列サイズが発散しうる状況で一般的に書くのは面倒．

＜(B2)の有界性を使った評価に関して＞\\
今までは$a,b$の有界性などから定数は初期値$\bar{x}_{j-1}$に依存しない形だったから$\sup_{\bar{x}_{j-1}}E\j[\cdot]$が評価できているので$E[\cdot]$も評価できる．
エルゴードの時もモーメント評価ができる設定位を考えればよいから(B2)は$E$評価でいいだろう．しかしエルゴードを考えると$\mathcal{S}(c_n)$を使う方が良い．
For a sequence $(c_n)_{n=1}^\infty$ of positive numbers, $\mathcal{S}(c_n)$ denotes the set of all sequences $(p_j(\bar{x}_{j-1},\theta))_{j=1}^{m_n}$ of measurable functions such that
\begin{equation*}
\sup_nc_n^{-1}E\bigg[\sum_{j=1}^{m_n}|p_j(\bar{x}_{j-1})|\bigg]<\infty.
\end{equation*}

densityの存在を言うのにマルコフ性の仮定がなくてもおそらく大丈夫か．
多分$P(\cdot |\bar{x}_{j-1})=Q_jF^{-1}$と書けることと$K$が計算できて$K^{-1}$が評価できることが実質的にマルコフ並みの良さを仮定しているか．
$x_1,\cdots,x_{j-1}$に複雑に依存していたらこんなにきれいに書けないだろう．
}
\end{discuss}

Let $\theta_0$ be the true value of parameter $\theta$. 
We will see later in Proposition~\ref{delLogP-eq} that $F_{n,\theta,j}$ admits a density $p_{j,\bar{x}_{j-1}}(x_j,\theta)$
{\colore that} satisfies $p_{j,\bar{x}_{j-1}}(x_j,\cdot) \in C^2(\Theta)$ almost {\colore everywhere} in $x_j\in \mathbb{R}^{k_j-k_{j-1}}$ under (B1) and (B2).
{\colord Let $N_j=\{x_j\in\mathbb{R}^{k_j-k_{j-1}}|\sup_{\theta\in\Theta}p_{j,\bar{x}_{j-1}}(x_j,\theta)>\inf_{\theta\in\Theta}p_{j,\bar{x}_{j-1}}(x_j,\theta)=0\}$
and $M_j=\{x_j\in\mathbb{R}^{k_j-k_{j-1}}|\inf_{\theta\in\Theta}p_{j,\bar{x}_{j-1}}(x_j,\theta)>0\}$. We further assume the following condition.
\begin{description}
\item[Assumption (N1).] For any $h\in \mathbb{R}^d$, 
\begin{equation*}
E_{\theta_0}\bigg[\sum_{j=1}^{m_n}\int_{N_j}p_{j,\bar{x}_{j-1}}(x_j,\theta_0+r_nh)dx_j\bigg]\to 0
\end{equation*}
as $n\to\infty$.
\end{description}
}

{\colord If $\sup_{\theta\in\Theta}p_j(x_j,\theta)=0$ or $\inf_{\theta\in\Theta}p_j(x_j,\theta)>0$, we have $x_j\in N_j^c$.
Condition~(N1) says that the probability of other cases is asymptotically negligible. This condition is used to validate an estimate such as (\ref{L2-regu-equiv}).
However, if $F_{n,\theta,j}$ is approximated by a Gaussian random variable and satisfies (B3) and (N2) below,
{\colore then} we can check (A1) without (N1) (see Lemma~\ref{L2-regu-lemma}).}

With these definitions, the following theorem shows that the $L^2$ regularity condition is automatically satisfied under (B1), (B2), and (N1).
{\colord Let 
\begin{equation}\label{N1-xi-def}
\dot{\xi}_j(\theta)=\frac{\partial_\theta p_j}{2\sqrt{p_j}}1_{M_j}(x_j,\theta),
\quad \eta_j=\frac{\partial_\theta p_j}{2p_j}1_{M_j}(x_j,\theta_0).
\end{equation}
}
\begin{theorem}\label{Malliavin-LAMN-thm2}
Assume (B1), (B2), (N1), (A4), and (A5) {\colord with $\dot{\xi}_j(\theta)$ and $\eta_j$ defined in (\ref{N1-xi-def})}. 
{\colorp Then (L) holds true for $\{P_{\theta,n}\}_{\theta,n}$ at $\theta=\theta_0$ with $r_n = \epsilon_nI_d$.
If further (P) is satisfied, {\colore then} $\{P_{\theta,n}\}_{\theta,n}$ satisfies the LAMN condition at $\theta=\theta_0$.}
\end{theorem}

In the following, we give tractable sufficient conditions for (A4) and (A5) when $F_{n,\theta,j}$ has a Gaussian approximation $\tilde{F}_{n,\theta,j}$.
\begin{description}
\item[Assumption (B3).]  {\colore There exist a matrix $B_{j,i,\theta}=B_{j,i,\theta,\bar{x}_{j-1},n}$ and $h_{j,l}=h_{n,\theta,j,l,\bar{x}_{j-1}}\in H_j \ (1\leq l\leq k_j-k_{j-1})$
such that $\tilde{F}_{n,\theta,j,\bar{x}_{j-1}}=(W_j(h_{j,l}))_{l=1}^{k_j-k_{j-1}}$ is Fr\'echet differentiable with respect to $\theta$ on $L^p$ space,
and $\partial_{\theta_i}\tilde{F}_{n,\theta,j}=B_{j,i,\theta}\tilde{F}_{n,\theta,j}$.
Moreover, $\partial_\theta B_{j,i,\theta}$ exists and is continuous with respect to $\theta$, and there exists a constant $C_p$ and a sequence $(\rho_n)_{n\in\mathbb{N}}$ of positive numbers such that
}
\begin{equation*}
\sup_{{\colord n,}i,j,\bar{x}_{j-1},\theta,l_1,l_2}|[{\colord \partial_\theta^l}B_{j,i,\theta}]_{l_1,l_2}|<\infty,
\end{equation*}
{\colore and
\EQQ{\lVert [F_{n,\theta,j}-\tilde{F}_{n,\theta,j}]_{i'}\rVert_{3,p} 
+\lVert \partial_{\theta_i}[F_{n,\theta,j}-\tilde{F}_{n,\theta,j}]_{i'}\rVert_{2,p} \leq C_p\rho_n}}
for {\colore $p>1$}, $1\leq j\leq m_n$, $1\leq i\leq d$, $1\leq i'\leq k_j-k_{j-1}$, $\theta\in \Theta$ and $\bar{x}_{j-1}$.
\begin{discuss}
{\colorr $X$から$W$を復元するには2.3の$a$: invertibleは必要．復元しないと条件付き期待値を外すのが難しい．

$\delta D\tilde{F}=\tilde{F}$も必要だが，この時$E[F]=E[\delta DF]=0$で，Nualart Prop 1.2.2より
$E[\lVert DF\rVert^2_H]=\sum_{n=1}^\infty n\lVert J_nF\rVert_2^2<\infty$.
\begin{equation*}
E[\lVert DF\rVert^2_H]=E[F\delta DF]=E[F^2]=\sum_{n=1}^\infty \lVert J_nF\rVert^2_2
\end{equation*}
より$J_nF=0$ for $n\geq 2$なのである$h\in H$があって$F=W(h)$. $\delta D\tilde{F}=c\tilde{F}$も考えられるが，$\tilde{K}$がdeterministicになるには$F=W(h)$が必要だからこれを仮定する．}
\end{discuss}
\end{description}
\begin{discuss}
{\colorr 条件付き期待値をはずすにはこういう評価が必要．
densityの存在はマリアバンマトリックスの非退化性とNualart Thm 2.1.2からでる．
$\tilde{K}$の各要素の評価は$K$の評価と$K-\tilde{K}$の評価から出る．
}
\end{discuss}

{\colord For a statistical model of diffusion processes, $F_{n,\theta,j}$ corresponds to {\colora normalized} discrete observations and $\tilde{F}_{n,\theta,j}$ corresponds their {\colore Euler--Maruyama} approximations.
See (\ref{F-def}) and (\ref{tildeF-def}) for an example.}

Let $\tilde{K}_j(\theta)=(\langle h_{j,l_1},h_{j,_2}\rangle_{H_j})_{l_1,l_2}$.
Then, we will see that for sufficiently large $n$, $\tilde{K}_j(\theta)$ is invertible almost surely under (B1)--(B3) and that $\alpha_n\rho_n\bar{k}_n^2\to 0$ 
in Lemma~\ref{tildeK-est-lemma} of Section~\ref{Mcal-section}.
Let
\begin{equation*}
\mathcal{L}_{j,i,\bar{x}_{j-1}}(u,\theta)=u^\top B_{j,i,\theta}^\top \tilde{K}_j^{-1}(\theta)u-{\rm tr}(B_{j,i,\theta}).
\end{equation*}

Let $\Phi_{j,i}=(B_{j,i,\theta_0}^\top\tilde{K}_j^{-1}(\theta_0)+\tilde{K}_j^{-1}(\theta_0)B_{j,i,\theta_0})/2$, and let
\EQQ{\gamma_j(\bar{x}_{j-1})=(2{\rm tr}(\Phi_{j,i}\tilde{K}_j(\theta_0)\Phi_{j,i'}\tilde{K}_j(\theta_0)))_{i,i'=1}^d.}
\begin{description}
\item[Assumption (B4).] There exist $\mathbb{R}^d$-valued random variables $\{G_j^n\}_{1\leq j\leq m_n,n,\theta}$ and a filtration $\{\mathcal{G}_j\}_{j=1}^{m_n}$ on $(\Omega,\mathcal{F},P)$ 
such that 
  $(X_j^{n,\theta_0},G_j^n)$ is $\mathcal{G}_j$-measurable, $E[G_j^n|\mathcal{G}_{j-1}]=0$, and
  \EQN{\label{Gj-def-eq}
   &Q_j((F_{n,\theta_0,j},(\mathcal{L}_{j,i,\bar{x}_{j-1}}(\tilde{F}_{n,\theta_0,j},{\colorc \theta_0}))_{i=1}^d)\in A)|_{\bar{x}_{j-1}=\bar{X}_{j-1}^{n,\theta}} \\
   &\quad =P((X_j^{n,\theta_0},G_j^n)\in A|\mathcal{G}_{j-1})
  }
for $A\in \mathcal{B}(\mathbb{R}^{k_j-k_{j-1}}\times \mathbb{R}^d)$ and sufficiently large $n$.
Moreover, 
{\colore \EQQ{\sup_n\bigg(\epsilon_n^2\sum_{j=1}^{m_n}E[|\gamma_j(\bar{X}_{j-1}^{n,\theta_0})|]\bigg)<\infty,}} 
\begin{equation}\label{rho-condition}
\alpha_n\rho_n{\colord \bar{k}_n^2}\to 0 \quad {\rm and} \quad \epsilon_n^2m_n\alpha_n^3\rho_n{\colord \bar{k}_n^6}\to 0
\end{equation} 
as $n\to \infty$, and there exist a random $d\times d$ {\colorp matrix} $\Gamma$ and a $d$-dimensional standard normal random variable $\mathcal{N}$ such that
\begin{equation}\label{B4-stable-conv}
\bigg(\epsilon_n\sum_{j=1}^{m_n}G_j^n, \epsilon_n^2\sum_{j=1}^{m_n}\gamma_j({\colorn \bar{X}_{j-1}^{n,\theta_0}})\bigg)\overset{d}\to (\Gamma^{1/2}\mathcal{N},\Gamma).
\end{equation}
\begin{discuss}
{\colorr (\ref{B4-stable-conv})式の十分条件を書こうとすると，特にorthogonal martingaleとの条件付き期待値評価を一般に書くのが大変だからやめておく}
{\colorr 元々$\sup_{i,n}(n^{-1}\sum_{j=1}^{m_n}\sup_{\bar{x}_{j-1}}E_j[\mathcal{L}_{j,i}(\tilde{F}_{n,\theta,j},\theta)^2])<\infty$を仮定していたが，
$\tilde{F}=W(h)\sim N(0,\tilde{K})$より$E_j[\mathcal{L}(\tilde{F})^2]={\rm tr}(B^2)+{\rm tr}(B^\top \tilde{K}^{-1} B\tilde{K})$.
よって
\begin{equation*}
\sup_{i,n}\frac{1}{n}\sum_{j=1}^{m_n}\sup_{\bar{x}_{j-1}}({\rm tr}(B^2)+{\rm tr}(B^\top \tilde{K}^{-1} B\tilde{K}))<\infty
\end{equation*}
を仮定すれば成り立つ．}
\end{discuss}
\item[Assumption (N2).] {\colord ${\colore [F_{n,\theta,j}]_i}\in \cap_{p>1, r\in\mathbb{N}}\ \mathbb{D}^{r,p}$ for $n,\theta,j,i,\bar{x}_{j-1}$ 
and 
  \EQQ{\sup_{\theta\in\Theta}\lVert {\colore [F_{n,\theta,j}]_i}\rVert_{r,p}<\infty} 
for any $n,i,j,\bar{x}_{j-1}$, $r\in\mathbb{N}$ and $p>1$.
{\colore Also, 
  \EQQ{\partial_{\theta_i}D_j[\tilde{F}_{n,\theta,j}]_k=\sum_{l=1}^{k_j-k_{j-1}}[B_{j,i,\theta}]_{k,l}D_j[\tilde{F}_{n,\theta,j}]_l}}
and $\sup_{\theta\in\Theta}E_j[|\det K_j^{-1}(\theta)|^p]<\infty$ for any $p>1$, $j$, $k$, $\bar{x}_{j-1}$ and $1\leq i\leq d$.}
\item[Assumption (B5).] {\colord (N1) or (N2) holds true.}
{\colorp \item[Assumption (P$'$).] $\Gamma$ in (B4) is positive definite almost surely.}
\end{description}

{\colorp
\begin{remark}
{\colore Because}
  \EQNN{x^\top \gamma_jx&=2{\rm tr}\bigg(\bigg(\sum_{i=1}^d\Phi_{j,i}x_i\bigg)\tilde{K}_j\bigg(\sum_{i'=1}^d\Phi_{j,i'}x_{i'}\bigg)\tilde{K}_j\bigg) \\
   &=2{\rm tr}\bigg(\tilde{K}_j^{1/2}\bigg(\sum_{i=1}^d\Phi_{j,i}x_i\bigg)\tilde{K}_j\bigg(\sum_{i'=1}^d\Phi_{j,i'}x_{i'}\bigg)\tilde{K}_j^{1/2}\bigg)\geq 0
  }
for $x=(x_1,\cdots,x_d)\in\mathbb{R}^d$, $\gamma_j$ is symmetric and nonnegative definite.
Hence $\Gamma$ is also symmetric and nonnegative definite almost surely under (B1)--(B3) and {\colore the fact} that
$\epsilon_n^2\sum_{j=1}^{m_n}\gamma_j({\colorn \bar{X}_{j-1}^{n,\theta_0}})\overset{d}\to \Gamma$ as $n\to\infty$.
\end{remark}
}

\begin{theorem}\label{Malliavin-LAMN-thm}
Assume {\colord (B1)--(B5)}. Then $\{P_{\theta,n}\}_{\theta,n}$ satisfies {\colorp (L)}
 {\colord  with {\colorp $T(\theta_0)$ equal to} $\Gamma$ in (B4), $r_n=\epsilon_nI_d$,
$\dot{\xi}_j(\theta)$ and $\eta_j$ are defined by (\ref{N1-xi-def}) if (N1) is satisfied, and
\begin{equation*}
\dot{\xi}_j(\theta)=\frac{\partial_\theta p_j}{2\sqrt{p_j}}1_{\{p_j\neq 0\}}(x_j,\theta),
\quad \eta_j=\frac{\partial_\theta p_j}{2p_j}1_{\{p_j\neq 0\}}(x_j,\theta_0)
\end{equation*}
if (N2) is satisfied.} {\colorp If further ($P'$) is satisfied, {\colore then} $\{P_{\theta,n}\}_{\theta,n}$ satisfies the LAMN property at $\theta=\theta_0$.}
\end{theorem}

{\colord {\colore Note} that Theorem~\ref{Malliavin-LAMN-thm} works without {\colore having to identify} zero points of the density function $p_j$, unlike {\colore in} previous studies.
This is a {\colore major} advantage because it is often not {\colore an} easy task to show {\colore either} that $p_j$ has no zero points {\colore or} that zero points are common for every $\theta$.
For example, {\colore to our knowledge,} there are no results related to zero points of transition density functions
for the statistical model of multi-dimensional integrated diffusion processes.
In the following section, we see that Theorem~\ref{Malliavin-LAMN-thm} can be applied {\colore to} this model.}

{\colorp
The following lemma is useful when we check (\ref{B4-stable-conv}) by using a martingale central limit theorem.
\begin{as}
The proof is {\colore given in} Section~\ref{Mcal-proof-section} {\colore of} the supplementary {\colore material}~\cite{fuk-ogi20s}.
\end{as}
\begin{arXiv}
{\colorc The proof is left to Section~\ref{Mcal-proof-section} in the appendix.}
\end{arXiv}

{\colore For a matrix $A$, $\lVert A\rVert_{{\rm op}}$ denotes the operator norm of $A$.}
\begin{lemma}\label{B4-suff-lemma}
Assume (B1)--(B3) and that (\ref{Gj-def-eq}) is satisfied for any $A\in \mathcal{B}(\mathbb{R}^{k_j-k_{j-1}}\times \mathbb{R}^d)$. Then,
\begin{enumerate}
\item $E[G_j^n(G_j^n)^\top|\mathcal{G}_{j-1}]=\gamma_j({\colorn \bar{X}_{j-1}^{n,\theta_0}})$ for any $1\leq j\leq m_n$, {\colore and}
\item $\epsilon_n^4\sum_{j=1}^{m_n}E[|G_j^n|^4|\mathcal{G}_{j-1}]\overset{P}\to 0$ as $n\to\infty$
if 
\begin{equation}
\epsilon_n^4m_n\alpha_n^8\bar{k}_n^8\sup_{j,i,\tilde{x}_{j-1}}
\lVert B_{j,i,\theta_0} \tilde{K}_j(\theta_0)+\tilde{K}_j(\theta_0)B_{j,i,\theta_0}^\top \rVert_{{\rm op}}^4\to 0.
\end{equation}
\end{enumerate}
\end{lemma}
}

{\colorp
\subsection{The LAMN property for degenerate diffusion models}\label{degenerate-diffusion-subsection}

{\colorn In this section, we show the LAMN property for degenerate diffusion processes
by applying the results in Sections~\ref{L2regu-subsection} and~\ref{LAMN-Malliavin-subsection}.}

Let $r\in\mathbb{N}$, and let $(\Omega,\mathcal{F},P)$ be the canonical probability space associated with {\colore an} $r$-dimensional Wiener process $W=\{W_t\}_{t\in [0,1]}$, that is,
$\Omega=C([0,1];\mathbb{R}^r)$, $P$ is the $r$-dimensional Wiener measure, $W_t(\omega)=\omega(t)$ for $\omega\in \Omega$, and $\mathcal{F}$ is the completion of {\colore the} Borel $\sigma$-field of $\Omega$ with respect to $P$.
Let $D$ be the {\colore Malliavin--Shigekawa} derivative related to the underlying Hilbert space $H=L^2([0,1];\mathbb{R}^r)$.
Let $\Theta$ be a bounded open convex set in $\mathbb{R}^d$.
We regard $\partial_xv=(\partial_{x_i}v_j)_{i,j}$ as a matrix for vectors $x$ and $v$.
$\bar{A}$ denotes the closure for a set $A$.

For $\theta\in \Theta$, let $X^\theta=(X_t^\theta)_{t\in [0,1]}$ be an $m$-dimensional diffusion {\colore process} satisfying $X_0^\theta=z_{{\rm ini}}$, and
\begin{equation}\label{degenerate-sde}
dX_t^\theta=b(X_t^\theta,\theta)dt+a(X_t^\theta,\theta)dW_t, \quad t\in [0,1],
\end{equation}
where $z_{{\rm ini}}\in \mathbb{R}^m$ and $a$ and $b$ are Borel functions.
%{\colorc $\mathcal{O}\subset \mathbb{R}^m$ is an open convex set, $z_{{\rm ini}}\in \mathcal{O}$,  
%and $a:\mathcal{O}\times \Theta\to \mathbb{R}^m\otimes \mathbb{R}^r$ 
%and $b:\mathcal{O}\times \Theta\to \mathbb{R}^m$ are Borel functions.}
We consider a statistical model with observations $(X_{j/n}^\theta)_{j=0}^n$.

Let $\kappa={\rm rank}(a(z,\theta))$. %{\colorc for $z\in\mathcal{O}$ and $\theta\in\Theta$}. 
Assume that $m/2\leq \kappa<m$ and {\colore that} $\kappa$ does not depend on $(z,\theta)$.
Then, by singular value decomposition, we can find an orthogonal matrix $U_{z,\theta}$ such that 
$U_{z,\theta}a(z,\theta)=(\tilde{a}(U_{z,\theta}z,\theta)^\top, O_{m-\kappa,r}^\top)^\top$,
where $O_{k,l}$ denotes {\colore a} $k\times l$ matrix with each element equal to zero.

{\colore First, we} assume that $U_{z,\theta}$ does not depend on $(z,\theta)$.
\begin{description}
\item[Assumption (C1).] The derivatives $\partial_z^i\partial_\theta^ja(z,\theta)$ and $\partial_z^i\partial_\theta^jb(z,\theta)$ exist on 
%${\colorc \mathcal{O}}\times \Theta$ 
$\mathbb{R}^m\times \Theta$ and can be extended to continuous functions on 
%${\colorc \mathcal{O}}\times \bar{\Theta}$ 
$\mathbb{R}^m\times \bar{\Theta}$
for $i\in\mathbb{Z}_+$ and $0\leq j\leq 3$.
Moreover, $\sup_{z,\theta}(|\partial_zb(z,\theta)|\vee |\partial_za(z,\theta)|)<\infty$, 
and there exist an orthogonal matrix $U$ and $\mathbb{R}^\kappa\otimes \mathbb{R}^r$-, $\mathbb{R}^\kappa$-, and $\mathbb{R}^{m-\kappa}$-valued 
Borel functions $\tilde{a}(z,\theta)$, $\tilde{b}(z,\theta)$, and $\check{b}(z)$, respectively, such that 
%and there exist an orthogonal matrix $U$, {\colorc convex open sets $\mathcal{O}_1\subset \mathbb{R}^\kappa$ and $\mathcal{O}_2\subset \mathbb{R}^{m-\kappa}$,
%and continuous functions $\tilde{a}:(U\mathcal{O})\times\bar{\Theta}\to \mathbb{R}^\kappa\otimes \mathbb{R}^r$, 
%$\tilde{b}:(U\mathcal{O})\times \bar{\Theta} \to \mathbb{R}^\kappa$, and $\check{b}:\mathbb{R}^m\to \mathbb{R}^{m-\kappa}$} 
%such that {\colorc $U\mathcal{O}=\mathcal{O}_1\times \mathcal{O}_2$,}
\begin{equation}
Ua(z,\theta)=
\left(\begin{array}{c} \tilde{a}(Uz,\theta) \\ O_{m-\kappa,r} \end{array}\right), 
\quad Ub(z,\theta)=
\left(\begin{array}{c} \tilde{b}(Uz,\theta) \\ \check{b}(Uz) \end{array}\right)
\end{equation}
for any 
%$z\in{\colorc \mathcal{O}}$ 
$z\in\mathbb{R}^m$ 
and $\theta\in{\colorc \bar{\Theta}}$. Further, $\tilde{a}\tilde{a}^\top(z,\theta)$ is positive definite for any 
%$(z,\theta)\in{\colorc \mathcal{O}}\times \bar{\Theta}$.
$(z,\theta)\in\mathbb{R}^m\times \bar{\Theta}$.
\end{description}
\begin{discuss}
{\colorr 一階微分の有界性は解の存在だけならsupにthetaを入れる必要はないが，$DF$の一様評価などでthetaに関する一様性が必要}
\end{discuss}

There exists a unique strong solution $(X_t^\theta)_{t\in[0,1]}$ of (\ref{degenerate-sde}) under (C1).
Let $P_{\theta,n}$ be the distribution of $(X_{k/n}^\theta)_{k=0}^n$, and let $\theta_0$ be the true value of $\theta$.
We denote $X_t=X_t^{\theta_0}$.

Under (C1), by setting $Y_t^\theta=UX_t^\theta$, we obtain
\begin{equation}\label{Y-SDE}
dY_t^\theta=
\left(\begin{array}{c} \tilde{b}(Y_t^\theta,\theta) \\ \check{b}(Y_t^\theta) \end{array}\right)dt
+\left(\begin{array}{c} \tilde{a}(Y_t^\theta,\theta) \\ O_{m-\kappa,r} \end{array}\right)dW_t. 
\end{equation}

We denote $z=(x,y)$ for $x\in\mathbb{R}^\kappa$ and $y\in\mathbb{R}^{m-\kappa}$,
{\colore $\nabla_1=(\partial_{z_1},\cdots, \partial_{z_\kappa})$,
$\nabla_2=(\partial_{z_{\kappa+1}},\cdots, \partial_{z_m})$,
}
${\rm Ker}(A)=\{x\in \mathbb{R}^l;Ax=0\}$, and by $A^+$ the Moore--Penrose inverse for a $k\times l$ matrix $A$. 

\begin{description}
\item[Assumption (C2).] The derivative $\partial_z^i\check{b}(z)$ is bounded for $i\in\mathbb{N}$ and 
\EQQ{\sup_{z\in\mathbb{R}^m}\lVert(({\colore \nabla_1}\check{b})^\top {\colore \nabla_1}\check{b})^{-1}(z)\rVert_{{\rm op}}<\infty.}
Moreover, 
  \EQN{{\rm Ker}(\tilde{a}(z,\theta))&\subset {\rm Ker}(\partial_{\theta_i}\tilde{a}(z,\theta)), \\
   {\rm Ker}(({\colore \nabla_1}\check{b})^\top(z))&\subset {\rm Ker}(({\colore \nabla_1}\check{b})^\top(z)\partial_{\theta_i}\tilde{a}\tilde{a}^+(z',\theta))
  }
for any $z,z'\in\mathbb{R}^m$, $1\leq i\leq d$, and $\theta \in\Theta$.
Furthermore, at least one of the following two conditions is satisfied;
\begin{enumerate}
\item $\check{b}$ is bounded;
\item ${\colore \nabla_2}\tilde{a}(z,\theta)=0$ and ${\colore \nabla_2}\check{b}(z)=0$ for any {\colore $z\in\mathbb{R}^m$} and $\theta\in\Theta$.
\end{enumerate} 
\end{description}
We need (C2) to satisfy $\partial_\theta F_{n,\theta,j}=B_{j,i,\theta}F_{n,\theta,j}$ in (B3). 
\begin{as}
See Section~\ref{degenerate-lamn-section} {\colore of the supplementary material~\cite{fuk-ogi20s} for the details}.
\end{as}
\begin{arXiv}
See Section~\ref{degenerate-lamn-section} in the appendix for the details.
\end{arXiv}
\begin{discuss}
{\colorr 局所化しないなら$z'=z$でもいい}
\end{discuss}

{\colore We can write $\tilde{a}^+=\tilde{a}^\top(\tilde{a}\tilde{a}^\top)^{-1}$ because $\tilde{a}\tilde{a}^\top$ is invertible.}
If $r=\kappa$ and (C1) is satisfied, {\colore then} we can easily check ${\rm Ker}(\tilde{a}(z,\theta))\subset {\rm Ker}(\partial_{\theta_i}\tilde{a}(z,\theta))$
{\colore because} $\tilde{a}(z,\theta)$ is invertible. Similarly, we can easily check 
${\rm Ker}(({\colore \nabla_1}\check{b})^\top(z))\subset {\rm Ker}(({\colore \nabla_1}\check{b})^\top(z)\partial_{\theta_i}\tilde{a}\tilde{a}^+(z',\theta))$
if $m-\kappa=\kappa$ and $({\colore \nabla_1}\check{b})^\top {\colore \nabla_1}\check{b}(z)$ is positive definite.

Let $\Psi_{t,\theta}=({\colore \nabla_1}\check{b})^\top\tilde{a}\tilde{a}^\top{\colore \nabla_1}\check{b}(UX_t,\theta)$, and let
  \EQN{\label{Gamma-def}
   \Gamma&=\bigg(\frac{1}{2}\int^1_0{\rm tr}((aa^\top)^+\partial_{\theta_i}(aa^\top)(aa^\top)^+\partial_{\theta_j}(aa^\top))(X_t,\theta_0)dt \\
   & \quad \quad +\frac{1}{2}\int^1_0{\rm tr}(\Psi_{t,\theta_0}^{-1}\partial_{\theta_i}\Psi_{t,\theta_0}\Psi_{t,\theta_0}^{-1}\partial_{\theta_j}\Psi_{t,\theta_0})dt\bigg)_{1\leq i,j\leq d}.
  }

\begin{description}
\item[Assumption (C3).] $\Gamma$ is positive definite almost surely.
\end{description}

\begin{theorem}\label{degenerate-LAMN-thm}
Assume (C1)--(C3). Then $\{P_{\theta,n}\}_{\theta,n}$ satisfies the LAMN property at $\theta=\theta_0$ with $\Gamma$.
\end{theorem}

\begin{remark}
The proof of Theorem~\ref{degenerate-LAMN-thm} in Section~\ref{degenerate-lamn-section} 
\begin{as}
{\colorc of the supplementary file~\cite{fuk-ogi20s}}
\end{as}
shows that we obtain similar results when $\kappa=m$ and $aa^\top$ is positive definite
by ignoring $\check{b}$ and $\Psi_{t,\theta}$. This approach {\colore allows} another proof of the LAMN property for nondegenerate diffusion processes by Gobet~\cite{gob01}.
\end{remark}

\begin{remark}
The first term in the right-hand side of (\ref{Gamma-def}) is equal to $\Gamma$ in Gobet~\cite{gob01} for nondegenerate diffusion processes.
It is also the same as $\Gamma$ for the statistical model with observations $\{[Y_{j/n}]_l\}_{0\leq j\leq n, 1\leq l\leq \kappa}$.
Then, the second term in the right-hand side of (\ref{Gamma-def}) corresponds to additional information obtained by 
observation $\{[Y_{j/n}]_l\}_{0\leq j\leq n, \kappa+1\leq l\leq m}$ for the degenerate process.
\end{remark}

\begin{example}\label{degenerate-ex1}
Let $\kappa\in\mathbb{N}$. Let $X_t^\theta$ and $\bar{X}_t^\theta$ be {\colore a $\kappa$-dimensional diffusion process} satisfying
\begin{equation}
dX_t^\theta=d(X_t^\theta,\bar{X}_t^\theta,\theta)dt+c(X_t^\theta,\theta)dW_t,
\quad d\bar{X}_t^\theta=X_t^\theta dt, \quad t\in [0,1],
\end{equation}
where $\theta \in \Theta\subset \mathbb{R}^d$ and $W_t$ is a $\kappa$-dimensional standard Wiener process.
We assume that $cc^\top(x,\theta)$ is positive definite and $c(x,\theta)$ and $d(z,\theta)$ are smooth functions with bounded derivatives $\partial_xc$ and $\partial_zd$.
Then, (C1) and (C2) are satisfied with $U=I_{2\kappa}$. $\Gamma$ is given by
\begin{equation}\label{Gamma-ex}
\Gamma=\left(\int^1_0{\rm tr}((cc^\top)^{-1}\partial_{\theta_i}(cc^\top)(cc^\top)^{-1}\partial_{\theta_j}(cc^\top))(X_t^{\theta_0},\theta_0)dt\right)_{1\leq i,j\leq d}.
\end{equation}
If further $\Gamma$ is positive definite almost surely, {\colore then} we obtain the LAMN property of this model by Theorem~\ref{degenerate-LAMN-thm}. 
\end{example}
\begin{discuss}
{\colorr $\Gamma$の非退化性を言うにはGloter Jacod(2001b)のidentifiability assumptionをチェックすればよいが，$\Theta$が多次元の時は
この$\Gamma$の形からいうのと変わらないか．Uchida Yoshidaのnondegeneracyは退化しているSDEでは多分使えない}
\end{discuss}

{\colorn Example~\ref{degenerate-ex1} {\colore pertains to Langevin-type} molecular dynamics (\ref{Langevin-eq}).
Here we assumed that the position $X_t^\theta$ and {\colore velocity} $\bar{X}_t^\theta$ of {\colore a} molecule are observed at discrete time points.
In {\colora Example~\ref{int-obs-ex} of} Section~\ref{integrated-diffusion-subsection}, we deal with the case where we {\colore observe only} {\colora the position} $\bar{X}_t^\theta$.

With some restriction on the diffusion coefficient $c$, we can extend Example~\ref{degenerate-ex1} 
to the case {\colore where} ${\rm dim}(X_t^\theta)>{\rm dim}(\bar{X}_t^\theta)$.
\begin{example}
Let $\kappa'\leq \kappa$. Let $X_t^\theta$ be the same as in Example~2.1, 
and let $\bar{X}_t^\theta$ be a $\kappa'$-dimensional stochastic process satisfying
$[\bar{X}_t^\theta]_i=\int^t_0[X_s^\theta]_ids$ for $1\leq i\leq \kappa'$.
Moreover, let $c(x,\theta)=f(x,\theta)A$ for some $\mathbb{R}$-valued function $f$ and matrix $A$ independent of $x$ and $\theta$.
We assume that $AA^\top$ is positive definite, $f$ is positive-valued, 
and $f(x,\theta)$ and $d(z,\theta)$ are smooth functions
with bounded derivatives $\partial_xf$ and $\partial_zd$.
Then, (C1) and (C2) are satisfied {\colore because} $({\colore \nabla_1}\check{b})^\top=(I_{\kappa'} \ O_{\kappa',\kappa -\kappa'})$
and $\partial_\theta cc^{-1}(x,\theta)=\partial_\theta f f^{-1}(x,\theta)I_\kappa$.
We have $\Psi_{t,\theta}=f^2(X_t^\theta,\theta)([AA^\top]_{ij})_{1\leq i,j\leq \kappa'}$, and hence we have
  \EQNN{[\Gamma]_{ij}&=\frac{1}{2}\int^1_0\bigg\{\frac{2\partial_{\theta_i}f}{f}\frac{2\partial_{\theta_j}f}{f}(X_t,\theta_0)\cdot \kappa
   +\frac{2\partial_{\theta_i}f}{f}\frac{2\partial_{\theta_j}f}{f}(X_t,\theta_0)\cdot \kappa'\bigg\}dt \\
   &=2(\kappa+\kappa')\int^1_0\frac{\partial_{\theta_i}f\partial_{\theta_j}f}{f^2}(X_t,\theta_0)dt.
  }

If we only observe $(X_{k/n}^\theta)_{k=0}^n$, then $\Gamma$ in Gobet~\cite{gob01} is calculated as
\begin{equation*}
\Gamma=\bigg(2\kappa\int^1_0\frac{\partial_{\theta_i}f\partial_{\theta_j}f}{f^2}(X_t,\theta_0)dt\bigg)_{1\leq i,j\leq d}.
\end{equation*}
Therefore, we conclude that $\Gamma$ for observations $X_t^\theta$ and $\bar{X}_t^\theta$ 
is $(\kappa+\kappa')/\kappa$ times as much as the one for observations $X_t^\theta$.
\end{example}

}

\begin{example}
Let $X_t^\theta=(X_t^{\theta,1},X_t^{\theta,2})$ be a two-dimensional diffusion process satisfying
\begin{equation}
\left\{
\begin{array}{ll}
dX_t^{\theta,1}=&(d(X_t^{\theta,1},X_t^{\theta,2},\theta)+e(X_t^{\theta,1},X_t^{\theta,2}))dt+c(X_t^{\theta,1},X_t^{\theta,2},\theta)dW_t, \\
dX_t^{\theta,2}=&d(X_t^{\theta,1},X_t^{\theta,2},\theta)dt+c(X_t^{\theta,1},X_t^{\theta,2},\theta)dW_t, \\
\end{array}
\right.
\end{equation}
where $\theta \in \Theta\subset \mathbb{R}^d$ and $W_t$ is a one-dimensional standard Wiener process.
That is, {\colore the} diffusion coefficients of $X_t^{\theta,1}$ and $X_t^{\theta,2}$ are the same.
We assume that $c$ is positive-valued, $\sup_{x,y}|\partial_xe(x,y)|^{-1}<\infty$, and $c(x,y,\theta)$, $d(x,y,\theta)$, and $e(x,y)$ are smooth functions 
with bounded derivatives $\partial_zc$, $\partial_zd$, and $\partial_z^ie$ for $i\in\mathbb{N}$ (z=(x,y)). 
Moreover, we assume that at least one of the following two conditions holds true:
\begin{enumerate}
\item $e$ is bounded;
\item $e(x,y)=\tilde{e}(x+y)$ and $c(x,y,\theta)=\tilde{c}(x+y,\theta)$ for some functions $\tilde{e}$ and $\tilde{c}$.
\end{enumerate}
Then (C1) and (C2) are satisfied with
\begin{equation*}
U=\frac{1}{\sqrt{2}}\left(\begin{array}{cc}
1 & 1 \\
1 & -1
\end{array}\right),
\end{equation*}
{\colore and} $\Gamma$ is given by (\ref{Gamma-ex}).
If further $\Gamma$ is positive definite almost surely, {\colore then} we obtain the LAMN property of this model by Theorem~\ref{degenerate-LAMN-thm}. 

For the statistical model with observations $(X_{j/n}^{\theta,2})_{j=0}^n$, $\Gamma$ is equal to half of the one in (\ref{Gamma-ex}).
The above result shows that the efficient asymptotic variance for estimators does not depend on $e$ and is equal to just half of the one when we observe $(X_{j/n}^{\theta,2})_{j=0}^n$.
\end{example}
}

{\colorn
\begin{example}
Let $m/2\leq \kappa<m$. Let $X_t^\theta$ be an $m$-dimensional diffusion process satisfying
\begin{equation}\label{factor-model}
dX_t^\theta=e(X_t^\theta)dt+f(X_t^\theta,\theta)AdW_t,
\end{equation}
where $W$ is a $\kappa$-dimensional standard Wiener process, $f(z,\theta)$ is an $\mathbb{R}$-valued function,
and $A$ is an $m\times\kappa$ matrix independent of $z$ and $\theta$.
Let $U^\top(\Lambda \ O_{\kappa, m-\kappa})^\top V$ be {\colora the} singular value decomposition of $A$ for a $\kappa\times \kappa$ diagonal matrix $\Lambda$,
and orthogonal matrices $U$ and $V$ of size $m$ and $\kappa$, respectively.
Then we have 
\begin{equation*}
Uf(z,\theta)A=\left(
\begin{array}{c}
\tilde{c}(Uz,\theta) \\
O_{m-\kappa,\kappa}
\end{array}
\right), \quad Ue(z)=\left(
\begin{array}{c}
\tilde{e}(Uz) \\
\check{e}(Uz)
\end{array}
\right),
\end{equation*}
where $\tilde{c}(z,\theta)=f(U^\top z,\theta)\Lambda V$, and $\tilde{e}(z)$ and $\check{e}(z)$ are suitable functions. 

We assume that ${\rm rank}(A)=\kappa$ (that is, $\Lambda$ is invertible), $f$ is positive--valued, 
$f(z,\theta)$ and $e(z)$ are smooth functions, 
and $\partial_zf$, $\partial_z^ie$, and $\lVert (({\colore \nabla_1}\check{e})^\top {\colore \nabla_1}\check{e})^{-1}\rVert_{{\rm op}}$ are bounded for $i\in\mathbb{Z}_+$.
Then we obtain $\partial_\theta \tilde{c}\tilde{c}^{-1}(z,\theta)=\partial_\theta ff^{-1}(U^\top z,\theta)I_\kappa$,
and consequently (C1) and (C2) hold.
Moreover, we have 
  \EQQ{\Psi_{t,\theta}=f^2(X_t^{\theta_0},\theta)(({\colore \nabla_1}\check{e})^\top \Lambda^2 {\colore \nabla_1} \check{e})(UX_t^{\theta_0})}
and hence
\begin{equation*}
\Gamma=\bigg(2m\int^1_0\frac{\partial_{\theta_i}f\partial_{\theta_j}f}{f^2}(X_t^{\theta_0},\theta_0)dt\bigg)_{1\leq i,j\leq d}.
\end{equation*}
If $\Gamma$ is positive definite almost surely, {\colore then} we have the LAMN property of this model.

We can regard (\ref{factor-model}) as a multi-factor model for stock prices, where each component of $W$ is regarded as a factor 
{\colore that influences the} stock prices, $A$ {\colore comprises the} contributions of each factor to each stock, and $f$ is a scalar {\colore that depends} on {\colore the} stock prices.
The above results show that we obtain the LAMN property of such a degenerate model if the number $\kappa$ of factors is in $[m/2,m)$.
\end{example}

\begin{discuss}
{\colorr
\begin{equation*}
dX_t=e(X_t)dt+Ac(X_t,\theta)dW_t
\end{equation*}
で$A$: $m\times \kappa$行列，$c$: $\kappa\times \kappa$行列，$A=U^\top \Lambda$, $c=E(\theta)C(X_t)$, $E={\rm diag}((\theta_i)_{1\leq i\leq \kappa})$とすると，
$\partial_{\theta_i}\tilde{c}\tilde{c}^+=E_i\theta_i^{-1}$と書け，
${\rm Ker}(({\colore \nabla_1}\check{e})^\top)$が$(1,0,\cdots, 0), (0,1,0,\cdots, 0), \cdots, (0,\cdots, 0,1)$のいくつかで張られていたらOKだが，
これを満たす例がわかりづらいか 
}
\end{discuss}

\subsection{The LAMN property for partial observations}\label{integrated-diffusion-subsection}

In this section, we show the LAMN property for degenerate diffusion processes with partial observations.
Gloter and Gobet~\cite{glo-gob08} showed the LAMN property for a one-dimensional integrated diffusion process.
{\colora While this model is similar to the one in Example~\ref{degenerate-ex1}, 
{\colore the} observations are only the integrated process $\bar{X}_t^{\theta_0}$ (partial observations).}
Their proof depends on Aronson's estimate, which is difficult to obtain for the multi-dimensional process.
%Though by using estimates for the transition density functions from below and above (Theorem 4 in~\cite{glo-gob08})
%of Theorem 4 uses the fact that the process is one-dimensional, and therefore it is difficult to extend to a multi-dimensional process,
We can avoid {\colore the} Aronson-type estimate by using the {\colora scheme with the $L^2$ regularity condition} 
in Sections \ref{L2regu-subsection} and \ref{LAMN-Malliavin-subsection},
and {\colore can} consequently extend their results to {\colore the} multi-dimensional case.
We can also generalize {\colore the} observed components of $X_t^{\theta_0}$ and $\bar{X}_t^{\theta_0}$, which yields 
an interesting example of a stock process and integrated volatility observations (Example~\ref{int-vol-ex}).

Let $m\in\mathbb{N}$, and let $(\Omega,\mathcal{F},P)$, $W$, $D$, $H$, and $\Theta$ be the same as {\colore in} Section~\ref{degenerate-diffusion-subsection}.
{\colora We consider a process} $Y_t^\theta=(\tilde{Y}_t^\theta,\check{Y}_t^\theta)$
{\colora {\colore that} satisfies a slight restricted version of the stochastic differential equation (\ref{Y-SDE}): 
$(\tilde{Y}_0^\theta,\check{Y}_0^\theta)=(\tilde{z}_{{\rm ini}},\check{z}_{{\rm ini}})$, and}
\EQN{d\tilde{Y}_t^\theta&=\tilde{b}(\tilde{Y}_t^\theta,\check{Y}_t^\theta,\theta)dt+\tilde{a}(\tilde{Y}_t^\theta,\theta)dW_t, \\
d\check{Y}_t^\theta&=B\tilde{Y}_t^\theta dt,}
{\colora where $B$ is an $(m-\kappa)\times \kappa$ matrix such that $BB^\top$ is positive definite.
Let ${\colorc \mathcal{Q}}:\mathbb{R}^\kappa\to \mathbb{R}^\kappa$ {\colorc be a projection}.
We assume that $({\colorc \mathcal{Q}}\tilde{Y}_{k/n}^{\theta_0})_{k=0}^n$ and $({\colorc \check{Y}_{k/n}^{\theta_0}})_{k=0}^n$ are observed.
Let {\colorc $q_1={\rm rank}(\mathcal{Q})$, $q_2=m-\kappa$}, and let $q=q_1+q_2$. We assume that {\colorc $0\leq q_1< \kappa$}.

For a $k\times l$ matrix $A$, we denote ${\rm Im}(A)=\{Ax;x\in\mathbb{R}^l\}$. }
\begin{description}
\item[Assumption (C2$'$).] 
The derivatives $\partial_x^i\partial_\theta^j\tilde{a}(x,\theta)$ and $\partial_z^i\partial_\theta^j\tilde{b}(z,\theta)$ exist on $\mathbb{R}^m\times \Theta$ 
and can be extended to continuous functions on $\mathbb{R}^m\times \bar{\Theta}$ for $i\in\mathbb{Z}_+$ and $0\leq j\leq 3$.
Moreover, $\sup_{z,\theta}(|\partial_z\tilde{b}(z,\theta)|\vee |\partial_x\tilde{a}(x,\theta)|)<\infty$,
${\rm Ker}(\tilde{a}(x,\theta))\subset {\rm Ker}(\partial_{\theta_i}\tilde{a}(x,\theta))$,
and ${\rm Ker}(B)\subset {\rm Ker}(B\partial_{\theta_i}\tilde{a}\tilde{a}^+(x,\theta))$
for any $x\in\mathbb{R}^\kappa$, $1\leq i\leq d$, and $\theta \in\Theta$.
\item[Assumption (C4).] 
\EQ{\label{Q1Q2-cond}  {\rm Ker}({\colorc B})\subset {\rm Im}({\colorc \mathcal{Q}}),}
{\colorc and}
\begin{equation}\label{Q1Q2-cond2}
{\colorc \mathcal{Q}}\partial_{\theta_i} \tilde{a}\tilde{a}^+(x,\theta)=\partial_{\theta_i} \tilde{a}\tilde{a}^+(x,\theta){\colorc \mathcal{Q}}
\end{equation}
for any $1\leq i\leq d$, $x\in\mathbb{R}^\kappa$, and $\theta\in\Theta$.
\end{description}
\begin{discuss}
{\colorr $(I-Q_2)y$に依存させることもできるが，$\tilde{b}$が$(I-Q_2)y$に依存するかどうかで
${\rm Im}(B{\colorc \mathcal{Q}})\supset ({\rm Im}(Q_2))^\perp$を仮定したり，$F$や$K$の次元が変わったりするからやめておく．}
\end{discuss}
{\colora 
By (\ref{Q1Q2-cond}), we have
  \EQ{q_1={\colore \dim {\rm Im}}({\colorc \mathcal{Q}})\geq \dim {\rm Ker}({\colorc B})=\kappa-\dim {\rm Im}({\colorc B})= \kappa-q_2,}
which implies $q\geq \kappa$.

Let {\colorc $R_1:{\rm Im}(\mathcal{Q})\to \mathbb{R}^{q_1}$} and $R_3:{\rm Im}(I_\kappa-{\colorc \mathcal{Q}})\to \mathbb{R}^{\kappa-q_1}$
be any isomorphism on vector spaces. We denote $\tilde{Q}_1=R_1{\colorc \mathcal{Q}}$, $\tilde{Q}_2={\colorc B}$, and $\tilde{Q}_3=R_3(I_\kappa-{\colorc \mathcal{Q}})$.
For a $\kappa\times \kappa$ matrix $A$, we denote}
$\Upsilon_{i,j}(A)=\tilde{Q}_iA\tilde{Q}_j^\top$ for $1\leq i,j\leq 3$,
\EQQ{\Xi_1(A)=\MAT{cc}{\Upsilon_{1,1} & \Upsilon_{1,2}/2 \\ \Upsilon_{2,1}/2 & \Upsilon_{2,2}/3}(A),
\quad \Xi_2(A)=\MAT{cc}{O_{q_1,q_1} & \Upsilon_{1,2}/2 \\ O_{q_2,q_1} & \Upsilon_{2,2}/6}(A),}
\EQQ{\Xi_3(A)=\MAT{cc}{\Upsilon_{1,1} & \Upsilon_{1,2}/2 \\ \Upsilon_{2,1}/2 & 2\Upsilon_{2,2}/3}(A),}
{\colora and for $L\in\mathbb{N}$, $L\geq 3$ and $1\leq k,l\leq 2$, 
we define an $(Lq+{\colorc (\kappa-q_1)}(k-1))\times (Lq+{\colorc (\kappa-q_1)}(l-1))$ matrix $\psi_L^{k,l}(A)$ by}
  \EQQ{\psi_L^{1,1}(A)=\MAT{ccccc}{
   \Xi_1 & \Xi_2 & O_{q,q} & \cdots & O_{q,q} \\
   \Xi_2^\top & \Xi_3 & \ddots & \ddots & \vdots \\
   O_{q,q} & \ddots & \ddots & \ddots & O_{q,q} \\
   \vdots & \ddots & \ddots & \Xi_3 & \Xi_2 \\
   O_{q,q} & \cdots & O_{q,q} & \Xi_2^\top & {\colora \Xi_3}
   }(A), 
  }
  \EQQ{
   \psi_L^{1,2}(A)=\MAT{cc}{
    & O_{(L-1)q,\kappa-q_1} \\
   \psi_L^{1,1} & \Upsilon_{1,3}/2 \\
    & \Upsilon_{2,3}/6
   }(A), \quad \psi_L^{2,1}(A)=(\psi_L^{1,2})^\top(A),
  }
  \EQQ{
   \psi_L^{2,2}(A)=\MAT{cccc}{
   \multicolumn{4}{c}{\psi_L^{1,2}} \\ 
   {\colorc O_{\kappa-q_1,(L-1)q}} & \Upsilon_{3,1}/2 & \Upsilon_{3,2}/6 & \Upsilon_{3,3}/3
   }(A).
  }
{\colora Here we ignore $\Upsilon_{i,j}$ if ${\rm rank}(\tilde{Q}_i)=0$ or ${\rm rank}(\tilde{Q}_j)=0$. Let}
  \EQNN{\mathcal{T}_{k,l,L}({\colora x})&=\Big({\rm tr}(\partial_{\theta_i}(\psi_L^{k,k}(\tilde{a}\tilde{a}^\top)^{-1})({\colora x},\theta_0)
   \psi_L^{k,l}(\tilde{a}\tilde{a}^\top)({\colora x},\theta_0) \\
   &\quad \quad \times \partial_{\theta_j}(\psi_L^{l,l}(\tilde{a}\tilde{a}^\top)^{-1})({\colora x},\theta_0)
   \psi_L^{l,k}(\tilde{a}\tilde{a}^\top)({\colora x},\theta_0))\Big)_{1\leq i,j\leq d}.
  }

{\colorc To show the LAMN property of partial observations, we consider an augmented model generated by block observations
with some observations of $(I_\kappa-\mathcal{Q})\tilde{Y}$, following the idea of Gloter and Gobet~\cite{glo-gob08}.
The matrix $\psi^{k,l}_L$ corresponds to {\colore the} covariance matrix of {\colore the} block observations.
See Sections~\ref{aug-L-subsection} and~\ref{ori-L-subsection} 
\begin{as}
{\colore of} the supplementary {\colore material}~\cite{fuk-ogi20s} 
\end{as}
for the details.}
\begin{description}
\item[Assumption (C5).] 
{\colora There} exists an $\mathbb{R}^d\otimes \mathbb{R}^d$-valued continuous function $g({\colora x})$ such that
\EQ{\label{C5-eq} {\colora L^{-1}\mathcal{T}_{k,l,L}(x)\to g(x)}}
{\colora as $L\to\infty$ uniformly in $x$ on compact sets} for $1\leq k,l\leq 2$.
\end{description}
\begin{discuss}
{\colorr (C5)は局所化した後に$L^{-1}\mathcal{T}_{k,l,L}(x)\to g(x)$の一様収束が成り立ち，
$\sum_j\gamma_j\overset{P}\to \Gamma$, $\sup_nE[n^{-1}\sum_j|\gamma_j|]<\infty$等に使う．}
\end{discuss}

Let
\EQ{\label{Gamma'-def} \Gamma'={\colora \frac{1}{2}\int^1_0g(\tilde{Y}_t)dt}.}

\begin{description}
\item[Assumption (C6).] $\Gamma'$ is positive definite almost surely.
\end{description}

Let $P_{\theta,n}$ be the distribution of partial observations $({\colorc \mathcal{Q}}\tilde{Y}_{k/n}^{\theta})_{k=0}^n$ and $({\colorc \check{Y}_{k/n}^{\theta}})_{k=0}^n$.
\begin{theorem}\label{partial-LAMN-thm}
Assume (C2$'$) {\colore and} (C4)--(C6). Then $\{P_{\theta,n}\}_{\theta,n}$ satisfies the LAMN property at $\theta=\theta_0$ with $\Gamma'$.
\end{theorem}

\begin{example}[Integral observations]\label{int-obs-ex}
Let $X_t^\theta$ and $\bar{X}_t^\theta$ be the same as {\colore in} Example~\ref{degenerate-ex1}.
We consider a statistical model with {\colora observations $(\bar{X}_{k/n}^{\theta_0})_{k=0}^n$}.
In this case, we have $m=2\kappa$, $B=I_\kappa$, ${\colorc \mathcal{Q}=O_{\kappa,\kappa}}$.
We assume that $cc^\top(x,\theta)$ is positive definite and {\colore that} $c(x,\theta)$ and $d(z,\theta)$ are smooth functions 
with bounded derivatives $\partial_xc$ and $\partial_zd$.
{\colore As in} Example~\ref{degenerate-ex1}, we have (C2$'$). Moreover, we can check (C4).
\begin{discuss}
{\colorr $\Theta$を$\theta_0$の周りで取り直せば$\bar{\Theta}$上のsmoothへの拡張はOK.}
\end{discuss}

We can see that $\psi_L^{2,2}(A)=V_L\otimes A$,
where {\colora $\otimes$ denotes} the Kronecker {\colora product}, and $V_L$ is {\colore an} $(L+1)\times (L+1)$ matrix satisfying 
\EQ{\label{VL-def} [V_L]_{ij}=(2/3)1_{\{i=j\}}+(1/6)1_{\{|i-j|=1\}}-(1/3)1_{\{i=j \ {\rm and} \ i\in\{1,L+1\}\}}.}

{\colora {\colore Because} we obtain similar equations for $\psi_L^{1,1}(A)$ and $\psi_L^{1,2}(A)$,
\begin{as}
together with Lemma~\ref{block-mat-lem} {\colorc in the supplementary {\colore material}~\cite{fuk-ogi20s}},
\end{as} 
\begin{arXiv}
together with Lemma~\ref{block-mat-lem},
\end{arXiv}
we have (\ref{C5-eq}) for
\EQ{\label{int-g-eq} g(x)=({\rm tr}((cc^\top)^{-1}\partial_{\theta_i}(cc^\top)(cc^\top)^{-1}\partial_{\theta_j}(cc^\top))(x,\theta_0))_{1\leq i,j\leq d}.}
\begin{discuss}
{\colorr 適当な$V_L^{1,2}$に対し，$\psi_L^{1,2}(A)=V_L^{1,2}\otimes A$. よって混合積性質から(\ref{C5-eq})が成立．
${\rm tr}(A\otimes B)={\rm tr}(A){\rm tr}(B)$. $cc^\top$は$x$に関してsmoothだからC5の収束は広義一様収束．}
\end{discuss}
Therefore, we have the LAMN property of this model if $\Gamma'$ in (\ref{Gamma'-def}) is positive definite almost surely.

This result is an extension of Gloter and Gobet~\cite{glo-gob08} to multi-dimensional processes.
Moreover, the result can be applied to the Langevin-type molecular dynamics in (\ref{Langevin-eq}) with {\colore positional} observations.
}
\end{example}

\begin{remark}\label{integrated-rem}
If we observe $(X_{k/n}^{\theta_0})_{k=0}^n$ instead of $(\bar{X}_{k/n}^{\theta_0})_{k=0}^n$, then Gobet~\cite{gob01} shows the LAMN property for {\colora this} model
with $\Gamma$ the same as (\ref{Gamma'-def}) {\colora and (\ref{int-g-eq})}. On the other hand, if we observe both $(X_{k/n}^{\theta_0})_{k=0}^n$ and $(\bar{X}_{k/n}^{\theta_0})_{k=0}^n$,
{\colore then} Example~\ref{degenerate-ex1} shows the LAMN property with $\Gamma$ twice {\colore that} in (\ref{Gamma'-def}).
Therefore, we can say that the efficient asymptotic variance with observations $(X_{k/n}^{\theta_0})_{k=0}^n$ and $(\bar{X}_{k/n}^{\theta_0})_{k=0}^n$ 
is half of {\colore that} with observations $(X_{k/n}^{\theta_0})_{k=0}^n$ or {\colore half of that} with $(\bar{X}_{k/n}^{\theta_0})_{k=0}^n$.
\end{remark}

\begin{example}[Observations of {\colora a} stock process and {\colore integrated} volatility] \label{int-vol-ex}
{\colora Let $W$ be a} two-dimensional standard Wiener process, {\colore and let} 
$c$ {\colora be an} $\mathbb{R}^2\otimes \mathbb{R}^2$-valued function {\colora with $c^j$ for $j\in\{1,2\}$.
Let $X_t={\colorc (X_t^i)_{i=1}^3}$ be a four-dimensional process satisfying}
\EQN{
dX_t^1&=d^1({\colorc X_t},\theta)dt+c^1(X_t^1,X_t^2,\theta)dW_t, \\
dX_t^2&=d^2({\colorc X_t},\theta)dt+c^2(X_t^1,X_t^2,\theta)dW_t, \\
d{\colorc X_t^3}&=X_t^2dt.
}
{\colora We assume {\colore that we observe} $((X_{k/n}^1,{\colorc X_{k/n}^3}))_{k=0}^n$. In this case, we have {\colorc $m=3$}, $\kappa=2$, $r=2$, 
{\colorc \EQ{{\colorc \mathcal{Q}}=\MAT{cc}{1 & 0 \\ 0 & 0}, \quad B=(0 \ 1).}}
}

We assume that $cc^\top(x,\theta)$ is positive definite {\colora for each $(x,\theta)$,} and $c(x,\theta)$, $d^1(z,\theta)$, and $d^2(z,\theta)$ are smooth functions 
with bounded derivatives $\partial_xc$, $\partial_zd^1$, and $\partial_zd^2$.
%Similarly to {\colora Example~\ref{int-obs-ex}}, we have (C2$'$). 
{\colorc We can check (\ref{Q1Q2-cond}).}

{\colora We consider the following two cases.
\begin{enumerate}
\item The case {\colore where} $c(x_1,x_2,\theta)=f(x_1,x_2,\theta)A$ for a matrix A and a positive--valued function $f(x_1,x_2,\theta)$:

We have that $AA^\top$ is positive definite and $\partial_{\theta_i}\tilde{a}\tilde{a}^+=\partial_{\theta_i}cc^{-1}=\partial_{\theta_i}ff^{-1}I_2$.
Then {\colorc (C2$'$) and (C4) are satisfied}.
Moreover, we obtain 
  \EQ{\psi_L^{k,l}(\partial_{\theta_i}(\tilde{a}\tilde{a}^\top))=\frac{2\partial_{\theta_i}f}{f}\psi_L^{k,l}(\tilde{a}\tilde{a}^\top).}
\begin{as}
Together with Lemma~\ref{block-mat-lem} {\colorc in the supplementary {\colore material}~\cite{fuk-ogi20s}},
\end{as} 
\begin{arXiv}
Together with Lemma~\ref{block-mat-lem}, 
\end{arXiv}
we have (C5) with 
  \EQQ{g(x_1,x_2)=\bigg(\frac{8\partial_{\theta_i}f\partial_{\theta_j}f}{f^2}(x_1,x_2,\theta_0)\bigg)_{i,j}.}
Therefore, we have the LAMN property if $\Gamma'$ in (\ref{Gamma'-def}) is positive definite almost surely.

\item The case {\colore where} $c(x_1,x_2,\theta)$ is a diagonal matrix for any $(x_1,x_2,\theta)$:\\
{\colore Because} $\partial_{\theta_i}\tilde{a}\tilde{a}^+$ also becomes a diagonal matrix, {\colorc (C2$'$) and (C4) are satisfied}.
Moreover, we have $\Upsilon_{1,2}=0$ and $\Upsilon_{1,3}=0$. Then, by rearranging {\colore the} rows and columns of $\psi_L^{2,2}$
by using an orthogonal matrix $\mathcal{V}_L$ of size $2L+1$, we have
\EQ{\mathcal{V}_L\psi_L^{2,2}(M)\mathcal{V}_L^\top=\MAT{cc}{[M]_{11}I_L & O_{L,L} \\ O_{L,L} & [M]_{22}V_L}}
for any diagonal matrix $M$ of size $2$, where $V_L$ {\colore is} defined in (\ref{VL-def}).
Together with Lemma~\ref{block-mat-lem} 
\begin{as}
{\colorc in the supplementary {\colore material}~\cite{fuk-ogi20s}}
\end{as}
and similar equations for $\psi_L^{1,2}$ and $\psi_L^{1,1}$, we have (\ref{C5-eq}) with
\EQ{g(x_1,x_2)=\bigg(\bigg(\frac{4\partial_{\theta_i}[c]_{11}\partial_{\theta_j}[c]_{11}}{[c]_{11}^2}
+\frac{4\partial_{\theta_i}[c]_{22}\partial_{\theta_j}[c]_{22}}{[c]_{22}^2}\bigg)(x_1,x_2,\theta_0)\bigg)_{1\leq i,j\leq d}.}
Then we have the LAMN property if $\Gamma'$ in (\ref{Gamma'-def}) is positive definite almost surely.
\end{enumerate}
}
\end{example}
\begin{remark}
{\colorc 
If $c^1(x_1,x_2,\theta)=\sqrt{x_2}v$ for a unit vector $v$ in Example~\ref{int-vol-ex}, {\colore then} we have $X_t^3=\langle X^1\rangle_t$, and therefore
$X^1$ represents a stock process for a stochastic volatility model and $X^3$ is the integrated volatility process.
If we observe daily stock {\colore prices} and realized volatility calculated from high-frequency data, {\colore then}
we can regard it as an approximation of the integrated volatility process.
{\colore Even though} $c^1=\sqrt{x_2}$ does not {\colore satisfy} our assumptions {\colore because} $\partial_xc$ is not bounded,
we can approximate this function by setting $c^1$ as a positive-valued smooth function satisfying $c^1(x_1,x_2,\theta)=\sqrt{x_2}$ on $\{|x_2|\geq \epsilon\}$
for small $\epsilon>0$.}
\end{remark}
}

\section{Malliavin calculus and the $L^2$ regularity condition}\label{Mcal-section}

In this section, we {\colore show} how to check (A1)--(A5) in Section~\ref{L2regu-subsection} under (B1)--({\colorp B5}).
The equations for density derivatives in Proposition~\ref{delLogP-eq} are crucial for the proof.
From these equations, we obtain Proposition~\ref{delp-p-est-prop} and Lemma~\ref{root-p-smooth-lemma}, which are necessary for checking (A1).

For a matrix $A$, {\colore $|A|$ denotes the Frobenius norm}, $|A|=\sqrt{\sum_{ij}|[A]_{ij}|^2}$.
Let 
\begin{equation*}
L^\theta(V)={\colore \sum_{k,k'}[K_j^{-1}(\theta)]_{k,k'} D_j[F_{n,\theta,j}]_{k'}[V]_k}
\end{equation*}
for a vector {\colore $V\in \mathbb{R}^{k_j-k_{j-1}}$}. 

The following proposition is essentially from Proposition~4.1 in~\cite{gob01} and Theorem~5 in~\cite{glo-gob08}.
To check (A1), we {\colorc need} an equation for $\partial_\theta^2p_j$.
\begin{as}
The proof is {\colore given in} Section~\ref{Mcal-proof-section} {\colore of} the supplementary {\colore material}~\cite{fuk-ogi20s}.
\end{as}
\begin{arXiv}
{\colorc The proof is left to Section~\ref{Mcal-proof-section} in the appendix.}
\end{arXiv}
\begin{proposition}\label{delLogP-eq}
Assume (B1) and (B2). Then $F_{n,\theta,j}$ admits a density denoted by $p_{j,\bar{x}_{j-1}}(x_j,\theta)$.
Moreover, $p_{j,\bar{x}_{j-1}}(x_j,\cdot ) \in C^2(\Theta)$, 
\begin{equation}\label{log-density-p-eq1}
\partial_\theta p_{j,\bar{x}_{j-1}}(x_j,\theta)=p_{j,\bar{x}_{j-1}}(x_j,\theta)E_j\big[\delta_j(L^\theta(\partial_\theta F_{n,\theta,j}))\big|F_{n,\theta,j}=x_j\big],
\end{equation}
and 
  \EQN{\label{log-density-p-eq2}
   \partial_\theta^2 p_{j,\bar{x}_{j-1}}(x_j,\theta)&=p_{j,\bar{x}_{j-1}}(x_j,\theta)E_j\big[\delta_j(L^\theta(\partial_\theta^2 F_{n,\theta,j}))
   +\delta_j(L^\theta(\mathfrak{A}_j))\big|F_{n,\theta,j}=x_j\big]
  }%now
almost everywhere in $x_j\in \mathbb{R}^{k_j-k_{j-1}}$, 
{\colorc where $\mathfrak{A}_j=(\delta_j(L^\theta(\partial_\theta F_{n,\theta,j}\partial_\theta {\colore [F_{n,\theta,j}]_k})))_k$}.
\end{proposition}
\begin{discuss}
{\colorr この二階微分の表現は非同期LAMNのLemma 3.6とも合っている}
\end{discuss}

\begin{as}
The proof of the following proposition is {\colore given in} Section~\ref{Mcal-proof-section} {\colore of} the supplementary {\colore material}~\cite{fuk-ogi20s}.
\end{as}
\begin{arXiv}
{\colorc The proof of the following proposition is left to Section~\ref{Mcal-proof-section} in the appendix.}
\end{arXiv}

\begin{proposition}\label{delp-p-est-prop}
Assume (B1) and (B2). Then
\begin{eqnarray}
\sup_{i,j,\bar{x}_{j-1},\theta}E_j[|\partial_{\theta_i}
 p_{j,\bar{x}_{j-1}}/p_{j,\bar{x}_{j-1}}|^41_{\{p_{j,\bar{x}_{j-1}}\neq
 0\}}(F_{n,\theta,j})]^{1/4}&\leq &C\alpha_n{\colord \bar{k}_n^2}, \label{delp-p-est1} \\ 
\quad
 \sup_{i,l,j,\bar{x}_{j-1},\theta}E_j[|\partial_{\theta_i}\partial_{\theta_l}
 p_{j,\bar{x}_{j-1}}/p_{j,\bar{x}_{j-1}}|^21_{\{p_{j,\bar{x}_{j-1}}\neq
 0\}}(F_{n,\theta,j})]^{1/2}&\leq &C\alpha_n^2{\colord \bar{k}_n^4}. \label{delp-p-est2}
\end{eqnarray}
\end{proposition}

Let $\theta_h=\theta_0+\epsilon_nh$ for $h\in \mathbb{R}^d$,
  \EQNN{\mathcal{E}^1_j(x_j,\theta)&=\mathcal{E}_j^1(x_j,\theta,\bar{x}_{j-1})=(E_j[\delta_j(L^\theta(\partial_{\theta_i}F_{n,\theta,j,\bar{x}_{j-1}}))|F_{n,\theta,j,\bar{x}_{j-1}}=x_j])_{i=1}^d, \\
   \mathcal{E}^2_j(x_j,\theta)&=(E_j[\delta_j(L^\theta(\partial_{\theta_i}\partial_{\theta_l}F_{n,\theta,j}))
   +\delta_j(L^\theta({\colorc \mathfrak{A}_j}))|F_{n,\theta,j}=x_j])_{i,l=1}^d. 
  }
{\colord We set the conditional expectations equal to zero when $p_j(x_j,\theta)=0$.
\begin{discuss}
{\colorr そのようにできるということは任意のボレル集合$A$に対して
\begin{equation*}
E[E[\delta|X^\theta=\bar{x}_j]|_{y=X^\theta}1_A(X^\theta)]=E[\delta 1_A(X^\theta)]
\end{equation*}
を満たせばよいが，左辺は$\int E[\delta|X^\theta=\bar{x}_j]1_A(\bar{x}_j)p_{X^\theta}(\bar{x}_j)d\bar{x}_j$
に等しいから$p_{X^\theta}(\bar{x}_j)=0$の時は任意に選べる．
}
\end{discuss}
Then $\mathcal{E}_j^1(x_j,\theta)$ and $\mathcal{E}_j^2(x_j,\theta)$ are measurable with respect to $\theta$ almost everywhere in $x_j$
because 
\EQQ{[\mathcal{E}_j^1(x_j,\theta)]_i=(\partial_{\theta_i}p_j/p_j)1_{\{p_j\neq 0\}}
\quad {\rm and} \quad [\mathcal{E}_j^2(x_j,\theta)]_{il}=(\partial_{\theta_i}\partial_{\theta_l}p_j/p_j)1_{\{p_j\neq 0\}}.}}
\begin{discuss}
{\colorr $i$と$l$を交換しても$[\mathcal{E}^2_j]_{il}$の値が変わらないことはProp 3.1での構成の仕方を見ればわかる．}
\end{discuss}

\noindent
{\bf Proof of Theorem~\ref{Malliavin-LAMN-thm2}.}

We check (A1)--(A3) in Theorem~\ref{L2regu-thm} by setting 
  \EQQ{p_j(\theta)=p_j(x_j,\theta)=p_{j,\bar{x}_{j-1}}(x_j,\theta).}

For sufficiently large $n$, we have $\{\theta_{th}\}_{t\in[0,1]}\subset \Theta$,
  \EQNN{&E_{\theta_0}\bigg[\sum_{j=1}^{m_n}\int_{N_j}[\sqrt{p_j}(x_j,\theta_h)-\sqrt{p_j}(x_j,\theta_0)]^2dx_j\bigg] \\
   &\quad \leq 2E_{\theta_0}\bigg[\sum_{j=1}^{m_n}\int_{N_j}(p_j(x_j,\theta_h)+p_j(x_j,\theta_0))dx_j\bigg] \to 0  
  }
as $n\to\infty$ by (N1), and
\begin{eqnarray}
&&\int_{{\colord N_j^c}}\bigg\{\sqrt{p_j}(x_j,\theta_h)-\sqrt{p_j}(x_j,\theta_0)-\frac{\epsilon_n}{2}h\cdot  {\colorb \dot{\xi}_j(\theta_0)}\bigg\}^2dx_j \nonumber \\
&&\quad =\int_{{\colord M_j}}\bigg\{\int^1_0\epsilon_n{\colord \frac{\partial_\theta p_j}{2\sqrt{p_j}}}(x_j,\theta_{th})dt\cdot h-\epsilon_n{\colord \frac{\partial_\theta p_j}{2\sqrt{p_j}}}(x_j,\theta_0)\cdot h\bigg\}^2dx_j \nonumber \\
&&\quad
 =\int_{{\colord M_j}}\bigg\{\int^1_0\int^t_0\epsilon_n^2h^\top{\colord \partial_\theta \bigg(\frac{\partial_\theta p_j}{2\sqrt{p_j}}\bigg)}(x_j,{\colord \theta_{sh}})hdsdt\bigg\}^2dx_j \nonumber \\
&&\quad\leq \epsilon_n^4|h|^4\int^1_0E\bigg[{\colord \bigg\lVert\frac{\partial_\theta^2p_j}{2p_j}-\frac{\partial_\theta p_j\partial_\theta p_j^\top}{4p_j^2}\bigg\rVert_{{\rm op}}^21_{\{p_j\neq 0\}}}(F_{n,\theta_{sh},j},\theta_{sh})\bigg]ds. \nonumber 
\end{eqnarray}
Together with Proposition~\ref{delp-p-est-prop}, we have
\begin{equation}
\sum_{j=1}^{m_n}E\bigg[\int_{{\colord N_j^c}}\bigg\{\sqrt{p_j}(x_j,\theta_h)-\sqrt{p_j}(x_j,\theta_0)-\frac{\epsilon_n}{2}h\cdot  {\colorb \dot{\xi}_j(\theta_0)}\bigg\}^2dx_j\bigg] \to 0,
\end{equation}
which implies (A1).

Moreover, we have (A2) {\colore because}
\begin{equation*}
E_{\theta_0}\bigg[{\colord \frac{\partial_\theta p_j}{p_j}}1_{\{p_j\neq 0\}}(x_j,\theta_0)\bigg|{\colord \mathcal{F}_{j-1}}\bigg]=\int \partial_\theta p_j(x_j,\theta_0)dx_j=0,
\end{equation*}
{\colord where $\mathcal{F}_{j-1}$ is the one in Section~\ref{L2regu-subsection}.}

Further, Proposition~\ref{delp-p-est-prop} yields (A3).

\qed

In the following, we prove Theorem~\ref{Malliavin-LAMN-thm}.
To show (A5), we replace $\mathcal{E}_j^1(F_{n,\theta,j},\theta)$ by $(\mathcal{L}_{j,i,\bar{x}_{j-1}}(\tilde{F}_{n,\theta,j}))_{i=1}^d$,
and then we apply (B4). For that purpose, we first estimate {\colore the} difference between $K_j$ and $\tilde{K}_j$.
\begin{lemma}\label{tildeK-est-lemma}
Assume (B1)--(B3) and that $\alpha_n\rho_n\bar{k}_n^2\to 0$. Then, for any $1\leq j\leq m_n$ and $p>1$, $\tilde{K}_j(\theta)$ is an invertible matrix almost surely and satisfies
\begin{equation}
\sup_{i,l,j,\bar{x}_{j-1},\theta}\lVert [K_j(\theta)-\tilde{K}_j(\theta)]_{il}\rVert_{2,p} \leq C_p\rho_n,
\quad
{\colord \sup_{j,\bar{x}_{j-1},\theta}\lVert \tilde{K}_j^{-1}(\theta)\rVert_{{\rm op}}}\leq
C\alpha_n{\colorb \bar{k}_n}
\end{equation}
for sufficiently large $n$.
\end{lemma}

\begin{as}
The proof is {\colore given in} Section~\ref{Mcal-proof-section} {\colore of} the supplementary {\colore material}~\cite{fuk-ogi20s}.
\end{as}
\begin{arXiv}
{\colorc The proof is left to Section~\ref{Mcal-proof-section} in the appendix.}
\end{arXiv}

\begin{proposition}\label{delLogP-approx-prop}
Assume {\colore (B1)--(B3)} and that $\alpha_n\rho_n\bar{k}_n^2\to 0$ as $n\to \infty$. Then 
%\begin{equation}\label{delLogP-approx-prop-eq1}
%E_j[\mathcal{L}_{j,i,\bar{x}_{j-1}}(\tilde{F}_{n,\theta,j},\theta)]=0,
%\end{equation}
{\colore there exists a positive constant $C$} such that
\begin{equation}\label{delLogP-approx-prop-eq2}
\sup_{i,j,\bar{x}_{j-1},\theta}E_j\bigg[\bigg|\frac{\partial_{\theta_i} p_{j,\bar{x}_{j-1}}}{p_{j,\bar{x}_{j-1}}}1_{\{p_j\neq 0\}}(F_{n,\theta,j},\theta)-\mathcal{L}_{j,i,\bar{x}_{j-1}}(\tilde{F}_{n,\theta,j},\theta)\bigg|^2\bigg]^{1/2}
\leq C\alpha_n^2\rho_n{\colord \bar{k}_n^4}
\end{equation}
for sufficiently large $n$.
\end{proposition}

\begin{proof}
{\colore For $V\in (\mathbb{D}_j^{1,p})^{k_j-k_{j-1}}$, 
we regard $D_jV=(D_j[V]_l)_l$ as a vector of size $k_j-k_{j-1}$.}
Let ${\bf L}_{j,i}^\theta=\partial_{\theta_i} \tilde{F}_{n,\theta,j}^\top \tilde{K}_j^{-1}(\theta)D_j\tilde{F}_{n,\theta,j}$.
{\colore First, we} show that
\begin{equation}\label{tildeL-approx}
\sup_{i,j,\bar{x}_{j-1},\theta}\lVert L^\theta(\partial_{\theta_i} F_{n,\theta,j})-{\bf L}_{j,i}^\theta\rVert_{\mathbb{D}^{1,p}(H_j)}\leq C_p\alpha_n^2\rho_n{\colord \bar{k}_n^4}.
\end{equation}

Condition~(B3) and Lemma~\ref{tildeK-est-lemma} yield estimates for 
\EQQ{(\partial_{\theta_i} F_{n,\theta,j}-\partial_{\theta_i} \tilde{F}_{n,\theta,j})^\top K_j^{-1}D_jF_{n,\theta,j}
\quad {\rm and}\quad \partial_{\theta_i} \tilde{F}_{n,\theta,j}^\top\tilde{K}_j^{-1}(D_jF_{n,\theta,j}-D_j\tilde{F}_{n,\theta,j}).}
{\colore Because} $K_j^{-1}-\tilde{K}_j^{-1}=\tilde{K}_j^{-1}(\tilde{K}_j-K_j)K_j^{-1}$, 
we also obtain an estimate for $\partial_{\theta_i} \tilde{F}_{n,\theta,j}^\top(K_j^{-1}-\tilde{K}_j^{-1})D_jF_{n,\theta,j}$.
Then we have (\ref{tildeL-approx}).
\begin{discuss}
{\colorr $|\partial_\theta F_{n,\theta,j}|\lVert \tilde{K}_j^{-1}\rVert_{{\rm op}}|(\tilde{K}_j-K_j)K_j^{-1}D_jF_{n,\theta,j}|\leq C\bar{k}_n^{1/2}\alpha_n\bar{k}_n\cdot \bar{k}_n^{1/2}\rho_n\bar{k}_n^2\alpha_n$.}
\end{discuss}

Moreover, Proposition~1.3.3 in Nualart~\cite{nua06} and (B3) yield
  \EQN{\label{deltaTildeL-est}
   \delta_j({\bf L}_{j,i}^\theta)&=\partial_{\theta_i} \tilde{F}_{n,\theta,j}^\top\tilde{K}_j^{-1}\delta_j(D_j\tilde{F}_{n,\theta,j})
   -{\rm tr}(\tilde{K}_j^{-1}\langle D_j\partial_{\theta_i} \tilde{F}_{n,\theta,j},D_j\tilde{F}_{n,\theta,j}\rangle_{H_j}) \\
   &=\tilde{F}_{n,\theta,j}^\top B_{j,i,\theta}^\top\tilde{K}_j^{-1}\tilde{F}_{n,\theta,j}
   -{\rm tr}(\tilde{K}_j^{-1}B_{j,i,\theta}\tilde{K}_j)=\mathcal{L}_{j,i,\bar{x}_{j-1}}(\tilde{F}_{n,\theta,j},\theta).
  }
\begin{discuss}
{\colorr \begin{equation*}
\sum_{k,n}\langle D^kB^{m,n}\tilde{F}^n,D^k\tilde{F}^l\rangle=\sum_nB^{mn}K_{n,l}=(BK)_{ml}.
\end{equation*}
}
\end{discuss}

Together with Proposition~\ref{delLogP-eq}, we have
  \EQNN{&\frac{\partial_{\theta_i} p_j}{p_j}1_{\{p_j\neq 0\}}(F_{n,\theta,j})-\mathcal{L}_{j,i,\bar{x}_{j-1}}(\tilde{F}_{n,\theta,j}, \theta) \\
   &\quad =E_j[\delta_j(L^\theta(\partial_{\theta_i} F_{n,\theta,j})-{\bf L}_{j,i}^\theta)|F_{n,\theta,j}] +E_j[r_i||F_{n,\theta,j}]-r_i,
  }
where $r_i=\mathcal{L}_{j,i,\bar{x}_{j-1}}(\tilde{F}_{n,\theta,j},\theta)-\mathcal{L}_{j,i,\bar{x}_{j-1}}(F_{n,\theta,j},\theta)$.
\begin{discuss}
{\colorr Nualart Prop 1.5.4より$\lVert \delta(u)\rVert_{0,p}\leq C_p\lVert u\rVert_{\mathbb{D}^{1,p}(H)}$で証明を見ると$C_p$は$j$によらずにとれる．$E[|r_i|^p]^{1/p}$の評価で$\bar{k}_n^3$がでる．}
\end{discuss}
Then we obtain (\ref{delLogP-approx-prop-eq2}) by (B3), (\ref{tildeL-approx}), {\colord and {\colore the fact} that
\begin{equation}\label{r-est}
E_j[|r_i|^p]^{1/p}\leq C_p\alpha_n\rho_n\bar{k}_n^3
\end{equation}
for any $p\geq 1$.
}

%Moreover, Proposition~\ref{delLogP-eq} and (\ref{deltaTildeL-est}) yield
%  \EQNN{&E_j\bigg[\frac{\partial_{\theta_i} p_j}{p_j}1_{\{p_j\neq 0\}}(F_{n,\theta,j},\theta)-\mathcal{L}_{j,i,\bar{x}_{j-1}}(\tilde{F}_{n,\theta,j},\theta)\bigg] \\
%   &\quad =E_j[\delta_j(L^\theta(\partial_{\theta_i} F_{n,\theta,j})-{\bf L}_{j,i}^\theta)]=0.
%  }
\end{proof}

{\colord
\begin{lemma}\label{root-p-smooth-lemma}
Assume (B1), (B2), and (N2). Then, for any $n\in\mathbb{N}$, $1\leq j\leq m_n$, and $h\in\mathbb{R}^d$ satisfying $\{\theta_{th}\}_{t\in [0,1]}\subset \Theta$, 
the function $\sqrt{p_{j,\bar{x}_{j-1}}} (x_j,\theta_{th})$ is absolutely continuous on $t\in[0,1]$
almost everywhere in $x_j$.
\end{lemma}

\begin{as}
The proof is {\colore given in} Section~\ref{Mcal-proof-section} {\colore of} the supplementary {\colore material}~\cite{fuk-ogi20s}.
\end{as}
\begin{arXiv}
{\colorc The proof is left to Section~\ref{Mcal-proof-section} in the appendix.}
\end{arXiv}

\begin{lemma}\label{L2-regu-lemma}
Assume (B1)--(B3), (B5), and (\ref{rho-condition}). Then (A1) holds true.
\end{lemma}

\begin{proof}
If (N1) is satisfied, {\colore then} the proof of Theorem~\ref{Malliavin-LAMN-thm2} implies (A1). {\colore Thus,} we may assume (N2).
We fix $h\in\mathbb{R}^d$ and consider a sufficiently large $n$ so that $\{\theta_{th}\}_{t\in [0,1]}\subset \Theta$.
Thanks to Lemma~\ref{root-p-smooth-lemma}, $\partial_t\sqrt{p_{j,t}}$ exists almost everywhere in $t\in [0,1]$ and $\sqrt{p_{j,1}}-\sqrt{p_{j,0}}=\int^1_0\partial_t\sqrt{p_{j,t}}dt$
almost everywhere in $x_j\in\mathbb{R}^{k_j-k_{j-1}}$.
Moreover, we can see that $\partial_t\sqrt{p_{j,t}}=\partial_tp_{j,t}/(2\sqrt{p_{j,t}})$ when $p_{j,t}\neq 0$ by Proposition~\ref{delLogP-eq}.

For $t\in (0,1)$ such that $\partial_t\sqrt{p_{j,t}}$ exists and $p_{j,t}=0$, we have
\begin{equation*}
\liminf_{s\searrow t}\frac{\sqrt{p_{j,s}}}{s-t}\geq 0 \quad {\rm and} \quad \limsup_{s\nearrow t}\frac{\sqrt{p_{j,s}}}{s-t}\leq 0,
\end{equation*}
which imply $\partial_t\sqrt{p_{j,t}}=0$. Therefore, we obtain	
\begin{equation*}
\sqrt{p_{j,1}}-\sqrt{p_{j,0}}=\int^1_0\frac{\partial_tp_{j,t}}{2\sqrt{p_{j,t}}}1_{\{p_{j,t}\neq 0\}}dt.
\end{equation*}
Then we have
  \EQN{\label{l2-regu-eq1}
   &\sum_{j=1}^{m_n}\int \bigg(\sqrt{p_{j,1}}(x_j)-\sqrt{p_{j,0}}(x_j)-\frac{\partial_\theta p_j\cdot \epsilon_nh}{2\sqrt{p_j}}1_{\{p_j\neq 0\}}(x_j,\theta_0)\bigg)^2dx_j \\
   &\quad =\sum_{j=1}^{m_n}\int \bigg(\int^1_0\frac{\partial_\theta p_j\cdot \epsilon_nh}{2\sqrt{p_j}}1_{\{p_j\neq 0\}}(x_j,\theta_{th})dt-\frac{\partial_\theta p_j\cdot \epsilon_nh}{2\sqrt{p_j}}1_{\{p_j\neq 0\}}(x_j,\theta_0)\bigg)^2dx_j \\
   &\quad \leq \epsilon_n^2|h|^2\sum_{j=1}^{m_n}\int^1_0\int \bigg|\frac{\partial_\theta p_j}{2\sqrt{p_j}}1_{\{p_j\neq 0\}}(x_j,\theta_{th})-\frac{\partial_\theta p_j}{2\sqrt{p_j}}1_{\{p_j\neq 0\}}(x_j,\theta_0)\bigg|^2dx_jdt. 
  }

Let 
\begin{equation*}
\mathfrak{G}_{j,t}=\frac{\partial_\theta p_j}{2\sqrt{p_j}}1_{\{p_j\neq 0\}}(x_j,\theta_{th})-\frac{\sqrt{p_j}}{2}(\mathcal{L}_{j,i,\bar{x}_{j-1}}(x_j,\theta_{th}))_{i=1}^d,
\end{equation*}
then (\ref{r-est}), Proposition~\ref{delLogP-approx-prop}, and (\ref{rho-condition}) yield
  \EQN{\label{l2-regu-eq2}
   &\epsilon_n^2\sum_{j=1}^{m_n}\int^1_0\int |\mathfrak{G}_{j,t}-\mathfrak{G}_{j,0}|^2dx_jdt
   \leq 2\epsilon_n^2\sup_t \bigg(\sum_{j=1}^{m_n}\int (|\mathfrak{G}_{j,t}|^2+|\mathfrak{G}_{j,0}|^2)dx_j\bigg) \\
   &\quad \leq \epsilon_n^2\sup_t \bigg(\sum_{j=1}^{m_n}E_j\bigg[\bigg|\frac{\partial_\theta p_j}{p_j}1_{\{p_j\neq 0\}}
   -(\mathcal{L}_{j,i,\bar{x}_{j-1}})_{i=1}^d\bigg|^2(F_{n,\theta_{th},j},\theta_{th})\bigg]\bigg) \\
   &\quad \leq C\epsilon_n^2m_n(\alpha_n^4\rho_n^2\bar{k}_n^8+\alpha_n^2\rho_n^2\bar{k}_n^6) \to 0
  }
for any $\bar{x}_{j-1}$.

Moreover, {\colore because} the function $t\mapsto \sqrt{p_{j,t}}\mathcal{L}_{j,i,\bar{x}_{j-1}}(x_j,\theta_{th})$ is absolutely continuous, we have
  \EQN{\label{l2-regu-eq3}
   &\epsilon_n^2\sum_{j=1}^{m_n}\int^1_0\int \bigg(\frac{\sqrt{p_j}}{2}\mathcal{L}_{j,i}(x_j,\theta_{th})-\frac{\sqrt{p_j}}{2}\mathcal{L}_{j,i}(x_j,\theta_0)\bigg)^2dx_jdt \\
   &\quad =\epsilon_n^2\sum_{j=1}^{m_n}\int^1_0\int \bigg(\int^t_0\bigg(\frac{\sqrt{p_j}}{2}\partial_\theta\mathcal{L}_{j,i}+\frac{\partial_\theta p_j}{4\sqrt{p_j}}\mathcal{L}_{j,i}\bigg)(x_j,\theta_{sh})\cdot \epsilon_nhds\bigg)^2dx_jdt \\
   &\quad \leq C\epsilon_n^4m_n\sup_tE_j\bigg[\bigg|\frac{1}{2}\partial_\theta \mathcal{L}_{j,i}+\frac{\partial_\theta p_j}{4p_j}\mathcal{L}_{j,i}\bigg|^2(F_{n,\theta_{th},j},\theta_{th})\bigg]
  }
for any $1\leq i\leq d$. 

{\colore Because} 
 $\partial_\theta \tilde{K}_j^{-1}=-\tilde{K}_j^{-1}\partial_\theta \tilde{K}_j\tilde{K}_j^{-1}$ 
and 
\begin{equation*}
\partial_{\theta_i}\tilde{K}_j
  =\langle \partial_{\theta_i}D_j\tilde{F}_{n,\theta,j},D_j\tilde{F}_{n,\theta,j}\rangle + \langle D_j\tilde{F}_{n,\theta,j},\partial_{\theta_i}D_j\tilde{F}_{n,\theta,j}\rangle
  =B_{j,i,\theta}\tilde{K}_j+\tilde{K}_jB_{j,i,\theta}^\top,
\end{equation*}
\begin{discuss}
{\colorr $F_\theta,G_\theta$がFrechet微分可能なら内積の連続性より
\begin{eqnarray}
h^{-1}(\langle F_{\theta+h},G_{\theta+h}\rangle-\langle F_\theta,G_\theta\rangle)
&=&h^{-1}(\langle F_{\theta+h}-F_\theta,G_{\theta+h}\rangle+\langle F_\theta,G_{\theta+h}-G_\theta\rangle) \nonumber \\
&\to&\langle \partial_\theta F_\theta,G_\theta\rangle+\langle F_\theta,\partial_\theta G_\theta\rangle. \nonumber
\end{eqnarray}
}
\end{discuss}
Lemma~\ref{tildeK-est-lemma} yields
\EQN{\label{del-L-est}
&E_j[|\partial_{\theta_l}\mathcal{L}_{j,i,\bar{x}_{j-1}}(F_{n,\theta,j},\theta)|^2] \\
&\quad \leq CE_j[|F_{n,\theta,j}^\top\partial_{\theta_l}B_{j,i,\theta}^\top \tilde{K}_j^{-1}F_{n,\theta,j}|^2] \\
&\quad \quad +CE_j[|F_{n,\theta,j}^\top B_{j,i,\theta}^\top (B_{j,l,\theta}^\top\tilde{K}_j^{-1}+\tilde{K}_j^{-1}B_{j,l,\theta})F_{n,\theta,j}|^2] \\
&\quad \leq C\alpha_n^2\bar{k}_n^8
}
for any $\theta\in\Theta$.
Moreover, (\ref{r-est}) yields 
  \EQN{\label{L-F-est}
   &E_j[|\mathcal{L}_{j,i,\bar{x}_{j-1}}(F_{n,\theta,j})|^4] \\
   &\quad \leq CE_j[|\mathcal{L}_{j,i,\bar{x}_{j-1}}(F_{n,\theta,j})-\mathcal{L}_{j,i,\bar{x}_{j-1}}(\tilde{F}_{n,\theta,j})|^4]
   +CE_j[|\mathcal{L}_{j,i,\bar{x}_{j-1}}(\tilde{F}_{n,\theta,j})|^4] \\
   &\quad \leq C(\alpha_n\rho_n\bar{k}_n^3)^4+CE_j[|\partial_{\theta_i}\tilde{F}_{n,\theta,j}^\top \tilde{K}_j^{-1}\tilde{F}_{n,\theta,j}|^4]
   \leq C(\alpha_n\rho_n\bar{k}_n^3+\alpha_n\bar{k}_n^2)^4
  }
for any $\theta\in\Theta$.

(\ref{l2-regu-eq3})--(\ref{L-F-est}) and Proposition~\ref{delp-p-est-prop} yield
  \EQNN{&\epsilon_n^2\sum_{j=1}^{m_n}\int^1_0\int \bigg(\frac{\sqrt{p_j}}{2}\mathcal{L}_{j,i,\bar{x}_{j-1}}(x_j,\theta_{th})-\frac{\sqrt{p_j}}{2}\mathcal{L}_{j,i,\bar{x}_{j-1}}(x_j,\theta_0)\bigg)^2dx_jdt  \\
   &\quad \leq C\epsilon_n^4m_n(\alpha_n^2\bar{k}_n^8+\alpha_n^2\bar{k}_n^4(\alpha_n^2\rho_n^2\bar{k}_n^6+\alpha_n^2\bar{k}_n^4)).
  }
The right-hand side converges to zero by (B2) and (\ref{rho-condition}).
Together with (\ref{l2-regu-eq1}) and (\ref{l2-regu-eq2}), we obtain the conclusion.
\end{proof}
}

\noindent
{\bf Proof of Theorem~\ref{Malliavin-LAMN-thm}.}

{\colord Thanks to {\colorp Remark~\ref{Gamma-pd-rem},} Lemma~\ref{L2-regu-lemma}, and the proof of Theorem~\ref{Malliavin-LAMN-thm2},} 
it is sufficient to check (A4) and (A5) {\colorp under (B1)--(B5)}.
Let $X_j=X_j^{n,\theta_0}$, $\bar{X}_{j-1}=(X_1,\cdots, X_{j-1})$, and
\begin{equation*}
\mathcal{H}_j=E[\mathcal{E}_j^1(\mathcal{E}_j^1)^\top(X_j,\theta_0,\bar{X}_{j-1})|{\colord \sigma(\bar{X}_{j-1})}].
\end{equation*}
Then it suffices to show that 
\begin{equation}\label{A4-suff}
\sup_m\bigg(\epsilon_n^2\sum_{j=1}^{m_n}E[|\mathcal{H}_j|]\bigg)<\infty
\end{equation}
and
\begin{equation}\label{A5-suff}
\bigg(\epsilon_n\sum_{j=1}^{m_n}\mathcal{E}_j^1(X_j,\theta_0,\bar{X}_{j-1}),\epsilon_n^2\sum_{j=1}^{m_n}\mathcal{H}_j\bigg)
\overset{d}\to (\Gamma^{1/2}\mathcal{N},\Gamma).
\end{equation}

For sufficiently large $n$, (\ref{Gj-def-eq}) and Proposition~\ref{delLogP-approx-prop} yield
  \EQNN{&E[|\mathcal{E}_j^1(X_j,\theta_0,\bar{X}_{j-1})-G_j^n|^2|\mathcal{G}_{j-1}] \\
   &\quad =E_j[|\mathcal{E}_j^1(F_{n,\theta_0,j},\theta_0,\bar{x}_{j-1})-(\mathcal{L}_{j,i,\bar{x}_{j-1}}(\tilde{F}_{n,\theta_0,j},\theta_0))_{i=1}^d|^2]|_{\bar{x}_{j-1}=\bar{X}_{j-1}} \\
   &\quad \leq C\alpha_n^4\rho_n^2\bar{k}_n^8. 
  }
\begin{discuss}
{\colorr (\ref{Gj-def-eq})より任意のボレル関数$f$に対して
\begin{equation*}
E_j[f(\bar{x}_{j-1},F_{n,\theta,j,\bar{x}_{j-1}},\mathcal{L}_{j,i,\bar{x}_{j-1}}(\tilde{F}_{n,\theta,j,\bar{x}_{j-1}}))]|_{\bar{x}_{j-1}=X_{j-1}^{n,\theta_0}}
=E[f(X_{j-1}^{n,\theta_0},X_j^{n,\theta_0},G_j^n)|\mathcal{G}_{j-1}].
\end{equation*}
(∵$f(x,y,z)=1_A(x)1_B(y,z)$の時はOK. $\pi-\lambda$定理と極限操作より一般の場合もよい)

}
\end{discuss}
Together with (\ref{rho-condition}) {\colord and (B4)}, we obtain
  \EQN{\label{del-log-p-Gj-diff-est}
   &E\bigg[\bigg|\epsilon_n\sum_{j=1}^{m_n}\mathcal{E}_j^1(X_j,\theta_0,\bar{X}_{j-1})-\epsilon_n\sum_{j=1}^{m_n}G_j^n\bigg|^2\bigg] \\
   &\quad \leq \epsilon_n^2\sum_{j=1}^{m_n}E\bigg[E\bigg[\bigg|\mathcal{E}_j^1(X_j,\theta_0,\bar{X}_{j-1})-G_j^n\bigg|^2\bigg|\mathcal{G}_{j-1}\bigg]\bigg] \\
   &\quad =O(\epsilon_n^2m_n\alpha_n^4\rho_n^2\bar{k}_n^8) \to 0 
  }
as $n\to \infty$.
\begin{discuss}
{\colorr 
\begin{equation*}
E\bigg[\frac{\partial_\theta p_{j,X_{j-1}^{n,\theta_0}}}{p_{j,X_{j-1}^{n,\theta_0}}}(X_j^{n,\theta_0},\theta_0)\bigg|\mathcal{G}_{j-1}\bigg]
=E_j\bigg[\frac{\partial_\theta p_{j,\bar{x}_{j-1}}}{p_{j,\bar{x}_{j-1}}}(F_{n,\theta_0,j},\theta_0)\bigg]\bigg|_{\bar{x}_{j-1}=X_{j-1}^{n,\theta_0}}=0
\end{equation*}
に注意．
}
\end{discuss}

Let $\mathfrak{F}_j=({\colord\mathcal{L}_{j,i,\bar{x}_{j-1}}}(\tilde{F}_{n,\theta_0,j}))_{i=1}^d$, then we have
\begin{equation}\label{gamma-j-eq}
\gamma_j(\bar{x}_{j-1})=E_j[(\mathcal{L}_{j,i,\bar{x}_{j-1}}(\tilde{F}_{n,\theta_0,j})\mathcal{L}_{j,l,\bar{x}_{j-1}}(\tilde{F}_{n,\theta_0,j}))_{il}],
\end{equation}
{\colorb and} $\sup_{\bar{x}_{j-1}}E_j[|\mathfrak{F}_j|^2]^{1/2}=O(\alpha_n
\bar{k}_n^2+{\colord \alpha_n\rho_n\bar{k}_n^3})$ {\colord by (\ref{L-F-est})}. 
Together with (\ref{rho-condition}) {\colord and Propositions~\ref{delp-p-est-prop} and~\ref{delLogP-approx-prop}}, we have
  \EQN{\label{gamma-diff}
   &\sup_{\bar{x}_{j-1}}|E_j[\mathcal{E}_j^1(\mathcal{E}_j^1)^\top(F_{n,\theta_0,j},\theta_0,\bar{x}_{j-1})]-\gamma_j(\bar{x}_{j-1})| \\
   &\quad \leq C\sup_{\bar{x}_{j-1}}\bigg(E_j[\mathcal{E}_j^1|^2]^{\frac{1}{2}}E_j[|\mathcal{E}_j^1-\mathfrak{F}_j|^2]^{\frac{1}{2}}
   +E_j[|\mathfrak{F}_j|^2]^{\frac{1}{2}}E_j[|\mathcal{E}_j^1-\mathfrak{F}_j|^2]^{\frac{1}{2}}\bigg) \\
   &\quad =O({\colord \alpha_n\bar{k}_n^2\cdot \alpha_n^2\rho_n \bar{k}_n^4+(\alpha_n\bar{k}_n^2+\alpha_n\rho_n\bar{k}_n^3)\alpha_n^2\rho_n\bar{k}_n^4)})
   =o(\epsilon_n^{-2}m_n^{-1}).
  }

Then, (\ref{del-log-p-Gj-diff-est}), (\ref{gamma-diff}), and (B4) yield
\begin{equation*}
\sup_n\bigg(\epsilon_n^2\sum_{j=1}^{m_n}E[|\mathcal{H}_j|]\bigg)
=\sup_n\bigg(\epsilon_n^2\sum_{j=1}^{m_n}E[|\gamma(\bar{X}_{j-1})|]+O(1)\bigg)<\infty
\end{equation*}
and
  \EQNN{
   &\bigg(\epsilon_n\sum_{j=1}^{m_n}\mathcal{E}_j^1(X_j,\theta_0,\bar{X}_{j-1}),\epsilon_n^2\sum_{j=1}^{m_n}\mathcal{H}_j\bigg) \\
   &\quad=\bigg(\epsilon_n\sum_{j=1}^{m_n}G_j^n, \epsilon_n^2\sum_{j=1}^{m_n}\gamma_j({\colorn \bar{X}_{j-1}})\bigg)+o_p(1)
   \overset{d}\to(\Gamma^{1/2}\mathcal{N},\Gamma). 
  }
\qed

\begin{small}
\bibliographystyle{abbrv}
\bibliography{referenceLibrary_mathsci,referenceLibrary_other}
\end{small}

\appendix

\section{Proof of Theorem~\ref{L2regu-thm}}\label{L2regu-proof-section}

We use a similar approach to {\colore that for} Theorem~1 in Jeganathan~\cite{jeg82}.
{\colore First, we} show that
\begin{equation}\label{Xni-conv}
\sum_{i=1}^{m_n}(h^{\top}{\colorb r_n}{\colorb \eta_i\eta_i^{\top}}{\colorb r_n}h-h^{\top}T_nh)\to0
\end{equation}
in $P_{\theta_0,n}$-probability, which corresponds to {\colore Lemma~1} in~\cite{jeg82}.
\begin{discuss}
{\colorr $T_n$の収束先がdeterministicでないのでMcLeish1974の証明と少しだけ変わる．}
\end{discuss}

Let $X_{n,i}=|h^{\top}{\colorb r_n}{\colorb \eta_i}|$. We denote $P_{\theta_0,n}$ by $P$, $E_{\theta_0}$ by $E$, and {\colore $E_{\theta_0}[\cdot |\mathcal{F}_{j-1}]$} by $E_{(j)}$.
By (A4), for any {\colorb $\eta >0$} there exists $K>0$ such that 
\begin{equation}\label{addina}
 \sup_nP[\sum_{i=1}^{m_n}E_{(i)}[X_{n,i}^2]>K]<\eta/4.
\end{equation}
Let {\colorb $\delta >0$ and then take
$\epsilon >0 $ with  $16\epsilon^2(K+\epsilon^2)/\delta^2<\eta$}.
Let $W_{n,i}=X_{n,i}1_{\{|X_{n,i}|\leq \epsilon,\sum_{j=1}^iE_{(j)}[X_{n,j}^2]\leq K\}}$.
For sufficiently large $n$, we obtain
\begin{equation*}
P[X_{n,i}\neq W_{n,i} \ {\rm for \ some} \ i]\leq \sum_{i=1}^{m_n}P[|X_{n,i}|>\epsilon]+P[\sum_{i=1}^{m_n}E_{(i)}[X_{n,i}^2]>K]<\eta/2.
\end{equation*}
Moreover, (A3) yields
\EQNN{&\sum_{i=1}^{m_n}E_{(i)}[X_{n,i}^2-W_{n,i}^2] \\
&\quad \leq \sum_{i=1}^{m_n}E_{(i)}[X_{n,i}^21_{\{|X_{n,i}|>\epsilon\}}]+\sum_{i=1}^{m_n}E_{(i)}[X_{n,i}^2]1_{\{\sum_{j=1}^{m_n}E_{(j)}[X_{n,j}^2]>K\}} \\
&\quad \leq \delta/2+\sum_{i=1}^{m_n}E_{(i)}[X_{n,i}^2]1_{\{\sum_{j=1}^{m_n}E_{(j)}[X_{n,j}^2]>K\}},
}
\begin{discuss}
{\colorr $\sum_{j=1}^iE_{(j)}[X_{n,j}^2]$は$\mathcal{F}_{i-1}$可測だから
\begin{eqnarray}
\sum_iE_{(i)}[X_{n,i}^21_{\{|X_{n,i}|\leq \epsilon,\sum_{j=1}^iE_{(j)}[X_{n,j}^2]>K\}}
&\leq &\sum_iE_{(i)}[X_{n,i}^2]1_{\{\sum_{j=1}^iE_{(j)}[X_{n,j}^2]>K\}} \nonumber \\
&\leq &\sum_iE_{(i)}[X_{n,i}^2]1_{\{\sum_{j=1}^{m_n}E_{(j)}[X_{n,j}^2]>K\}}. \nonumber
\end{eqnarray}}
\end{discuss}
and  {\colorb so by (\ref{addina}) {\colore we have}
\begin{equation*}
 P[\sum_{i=1}^{m_n}E_{(i)}[X_{n,i}^2-W_{n,i}^2]>\delta/2] < \eta/4.
\end{equation*}
Further,}
  \EQNN{E[(\sum_{i=1}^{m_n}(W_{n,i}^2-E_{(i)}[W_{n,i}^2]))^2]&=\sum_{i=1}^{m_n}E[W_{n,i}^4-E_{(i)}[W_{n,i}^2]^2] \\
   &\leq \epsilon^2E[\sum_{i=1}^{m_n}E_{(i)}[W_{n,i}^2]] \leq \epsilon^2(K+\epsilon^2).
  }
\begin{discuss}
{\colorr 
\begin{equation*}
E[\sum_{i=1}^{m_n}E_{(i)}[W_{n,i}^2]]\leq E[\sum_iE_{(i)}[|X_{n,i}|^21_{\{|X_{n,i}|\leq \epsilon\}}]1_{\{\sum_{j=1}^{i-1}E_{(j)}[X_{n,j}^2]\leq K\}}].
\end{equation*}
$\sum_{j=1}^{i-1}E_{(j)}[X_{n,j}^2]\leq K$となる最大の$i$に対して，右辺の被積分関数は$\sum_{i'=1}^iE_{(i')}[X_{n,i'}^21_{\{|X_{n,i'}|\leq \epsilon\}}]\leq K+\epsilon^2$.
}
\end{discuss}

Therefore, we obtain
  \EQNN{&P[|\sum_{i=1}^{m_n}(X_{n,i}^2-E_{(i)}[X_{n,i}^2])|>\delta] \\
   &\quad \leq P[X_{n,i}\neq W_{n,i} \ {\rm for \ some} \ i]+P[\sum_{i=1}^{m_n}E_{(i)}[X_{n,i}^2-W_{n,i}^2]>\delta/2] \\
   &\qquad +P[|\sum_{i=1}^{m_n}(W_{n,i}^2-E_{(i)}[W_{n,i}^2])|>\delta/2] \\
   &\quad \leq 3\eta/4+4/\delta^2E[|\sum_{i=1}^{m_n}(W_{n,i}^2-E_{(i)}[W_{n,i}^2])|^2]<\eta.
  }
{\colore Because} $\eta,\delta>0$ are arbitrary, (\ref{Xni-conv}) holds true.

(A5) corresponds to {\colore Lemma~2} in~\cite{jeg82}.
\begin{discuss}
{\colorr Lem 2はHall 1977を適用する条件が合わなくなる.}
\end{discuss}
Moreover, by setting $\dot{\eta}_{nj}(\theta_0,h)=(p_j(\theta_0+{\colorb r_n}h)^{1/2}p_j(\theta_0)^{-1/2}-1)1_{\{p_j(\theta_0)\neq 0\}}$,
we obtain the following similarly to Lemmas~5--7 in~\cite{jeg82}:
\begin{discuss}
{\colorr ${\colorb \dot{\xi}_j}\to \dot{\xi}_j, \eta_j\to \eta_j, \dot{\eta}_{nj}\to \eta_j$. 
Lem 5--7を示すにはLem 3,4が必要．
Lemma 3は(A1)条件と$|c^2-d^2|$の不等式を使っているだけだからＯＫ．
Lemma 4も(A1),(A4)条件と$|c^2-d^2|$の不等式とLemma 3を使っているだけ．
(\ref{jeg-eq3})はLemma 5のstatementを$T\to T_n$としている．
Lemma 4より
\begin{equation*}
\sum_{j=1}^{m_n}E\bigg[\bigg|\dot{\eta}_{nj}^2(\theta_0,h)-\frac{1}{4}h^{\top}{\colorb r_n}\eta_j\eta_j^{\top}{\colorb r_n}h\bigg|\bigg]\to 0
\end{equation*}
なのでチェビシェフの不等式と(\ref{Xni-conv})と合わせて(\ref{jeg-eq3})が得られる．
（Lemma 5の証明の途中で実質的に(\ref{jeg-eq3})を示している）（Lem 5は後の証明でTightnessを利用しているだけだから(\ref{jeg-eq3})の形で問題ない．）

(\ref{jeg-eq2})はLemma 6の証明からすぐ得られる．(\ref{jeg-eq1})は(\ref{jeg-eq2})と(\ref{jeg-eq3})と$\sum_j|\dot{\eta}_{nj}|^3\leq \max_j|\dot{\eta}_{nj}|\times \sum_j|\dot{\eta}_{nj}|^2$と，
(A5)より$\{T_n\}_n$が$P_{\theta_0,n}$-tightであることを使えばよい．
(\ref{jeg-eq4})は(\ref{condE-eq})とLem 4,3に対応する結果を使えば
\begin{eqnarray}
2\sum_{j=1}^{m_n}E_{\theta_0}[\dot{\eta}_{nj}(\theta_0,h)|\bar{x}_{j-1}]
&=&2\sum_{j=1}^{m_n}\bigg\{\int \sqrt{p_j(\theta_0+{\colorb r_n}h)p_j(\theta_0)}d\mu_j-1\bigg\} \nonumber \\
&=&-\sum_{j=1}^{m_n}\int \xi_{n,j}^2(\theta_0,h)d\mu_j \nonumber \\
&=&-\frac{1}{4}\sum_{j=1}^{m_n}\int (1-Z_j)|h^{\top}{\colorb r_n}{\colorb \dot{\xi}_j}|^2d\mu_j+o_p(1) \nonumber \\
&=&-\frac{1}{4}\sum_{j=1}^{m_n}E_{\theta_0}[|h^{\top}{\colorb r_n}\eta_j|^2|\bar{x}_{j-1}]+o_p(1) \nonumber \\
&=&-\frac{1}{4}h^{\top}T_nh+o_p(1). \nonumber
\end{eqnarray}
よって(A2)より，$Y_j=2\dot{\eta}_{nj}(\theta_0,h)-h^{\top}{\colorb r_n}\eta_j$に対して
\begin{equation*}
\sum_{j=1}^{m_n}(Y_j-E_{\theta_0}[Y_j|\bar{x}_{j-1}]) \to 0
\end{equation*}
in $P_{\theta_0,n}$ probabilityをしめせばよい．
\begin{equation*}
E_{\theta_0}\bigg[\bigg(\sum_{j=1}^{m_n}(Y_j-E_{\theta_0}[Y_j|\bar{x}_{j-1}])\bigg)^2\bigg]\leq \sum_{j=1}^{m_n}E_{\theta_0}[Y_j^2]
\end{equation*}
なので，Lem 4の証明の最後の(2.13)よりＯＫ．

}
\end{discuss}
\begin{eqnarray}
\bigg|\sum_{j=1}^{m_n}\dot{\eta}_{nj}^2(\theta_0,h)-\frac{1}{4}h^{\top}T_nh\bigg|\to 0, \label{jeg-eq3} \\
\max_{1\leq j\leq m_n}|\dot{\eta}_{nj}(\theta_0,h)|\to 0, \label{jeg-eq2} \\ 
\sum_{j=1}^{m_n}|\dot{\eta}_{nj}(\theta_0,h)|^3\to 0, \label{jeg-eq1} \\
\bigg|2\sum_{j=1}^{m_n}\dot{\eta}_{nj}(\theta_0,h)-h^{\top}{\colorb r_n}\sum_{j=1}^{m_n}\eta_j+\frac{1}{4}h^{\top}T_nh\bigg|\to 0, \label{jeg-eq4}
\end{eqnarray}
in $P_{\theta_0,n}$-probability.

For any $h\in\mathbb{R}^d$, (\ref{jeg-eq2}) and Taylor's formula yield
  \EQNN{\log\frac{dP_{\theta_0+{\colorb r_n}h,n}}{dP_{\theta_0,n}}
   &=2\sum_{j=1}^{m_n}\log(1+\dot{\eta}_{nj}(\theta_0,h)) \\
   &=2\sum_{j=1}^{m_n}\dot{\eta}_{nj}(\theta_0,h)-\sum_{j=1}^{m_n}\dot{\eta}_{nj}^2(\theta_0,h)
   +\sum_{j=1}^{m_n}\alpha_{nj}|\dot{\eta}_{nj}(\theta_0,h)|^3 
  }
with probability {\colore tending} to one, where $|\alpha_{nj}|\leq 1$.
\begin{discuss}
{\colorr 一に近づく確率で見ればいいから$p_j\neq 0$上だけ考えればいい（$p_j$のどれかが零なら結合分布のdensityが零だから全部非零で考えればいい）．
さらに$\dot{\eta}$の絶対値が$1$未満で考えればよくてその上では$p_j(\theta_0+{\colorb r_n}h)\neq 0$.
$\mu$に関するdensityは$p_j$の積で書かれ，尤度比の絶対連続部分は$p_j$積の比にindicatorがついたものになるからＯＫ．
\begin{equation*}
\frac{1}{3!}\partial_x^3\log(1+x)=\frac{1}{3(1+x)^3}
\end{equation*}
だから，$\dot{\eta}_{nj}|\leq 1-3^{-1/3}$なら$|\alpha_{nj}|\leq 1$ととれる．
}
\end{discuss}

Together with (\ref{jeg-eq1}), we have
\begin{equation*}
\bigg|\log\frac{dP_{\theta_0+{\colorb r_n}h,n}}{dP_{\theta_0,n}}
-2\sum_{j=1}^{m_n}\dot{\eta}_{nj}(\theta_0,h)+\sum_{j=1}^{m_n}\dot{\eta}_{nj}^2(\theta_0,h)\bigg|\to 0
\end{equation*}
in $P_{\theta,n}$-probability. Therefore, (\ref{jeg-eq3}),(\ref{jeg-eq4}), and (A5) yield 
{\colorp Condition~(L).

If further (P) is satisfied, {\colore then} the LAMN property holds by Remark~\ref{Gamma-pd-rem}.}

\qed

\section{Degenerate diffusion models}\label{degenerate-lamn-section}

In this section, we prove Theorem~\ref{degenerate-LAMN-thm}.

We set $X_j^{n,\theta}$ in Section~\ref{LAMN-Malliavin-subsection} as
\begin{equation*}
X_j^{n,\theta}=\left(
\begin{array}{cc}
\sqrt{n} & 0 \\
0 & n\sqrt{n}
\end{array}
\right)\left(Y_{j/n}^\theta-Y_{(j-1)/n}^\theta-
\left(\begin{array}{c} 0 \\ \check{b}(Y_{(j-1)/n}^\theta)/n  \end{array}\right)\right).
\end{equation*}
\begin{discuss}
{\colorr $(Y_{j/n})_{j=0}^n$と$(Y_{j/n}-Y_{(j-1)/n})_{j=0}^n$は同一視される．後者のdensityを$p_j$と書くと，
\begin{equation*}
\log\frac{dP_{\theta',n}}{dP_{\theta,n}}=\sum_j\log\frac{p_j(\theta')}{p_j(\theta)}
\end{equation*}
で$Y_{(j-1)/n}$: givenでの$Y_{j/n}-Y_{(j-1)/n}-(0^\top \ \check{b}(Y_{(j-1)/n})^\top)^\top$のdensityは定数項を引いただけなので
$p_j(x,y+\check{b}(Y_{(j-1)/n}))$と書け，対応する確率測度を$P'_{\theta,n}$として，
\begin{equation*}
\log\frac{dP_{\theta',n}}{dP_{\theta,n}}=\log\frac{dP'_{\theta',n}}{dP'_{\theta,n}}
\end{equation*}
がわかる
}
\end{discuss}
Let $\mathcal{W}=(\mathcal{W}_t)_{{\colorc t\geq 0}}$ be an $r$-dimensional standard Wiener process on a canonical probability space, and let
$\mathcal{X}_t^{n,\theta}=\mathcal{X}_t^{n,\theta,z_0}$ and $\mathcal{Y}_t^{n,\theta}=\mathcal{Y}_t^{n,\theta,z_0}$ be $\kappa$- and $(m-\kappa)$-dimensional diffusion processes, respectively, satisfying
$(\mathcal{X}_0^{n,\theta}, \mathcal{Y}_0^{n,\theta})=0$, and
\begin{eqnarray}
\left\{
\begin{array}{ll}
d\mathcal{X}_t^{n,\theta}&=\tilde{b}_n(t,\mathcal{X}_t^{n,\theta},\mathcal{Y}_t^{n,\theta},\theta)dt+\tilde{a}_n(t,\mathcal{X}_t^{n,\theta},\mathcal{Y}_t^{n,\theta},\theta)d\mathcal{W}_t, \\
d\mathcal{Y}_t^{n,\theta}&=\check{b}_n(t,\mathcal{X}_t^{n,\theta},\mathcal{Y}_t^{n,\theta})dt,
\end{array}
\right. \nonumber
\end{eqnarray}
where $z_0=(x_0,y_0)\in\mathbb{R}^m$ and 
\begin{eqnarray}
\tilde{a}_n(t,x,y,\theta)&=&\tilde{a}(x_0+n^{-1/2}x,y_0+n^{-3/2}(y+t\sqrt{n}\check{b}(z_0)),\theta), \nonumber \\
\tilde{b}_n(t,x,y,\theta)&=&n^{-1/2}\tilde{b}(x_0+n^{-1/2}x,y_0+n^{-3/2}(y+t\sqrt{n}\check{b}(z_0)),\theta), \nonumber \\
\check{b}_n(t,x,y)&=&\sqrt{n}(\check{b}(x_0+n^{-1/2}x,y_0+n^{-3/2}(y+t\sqrt{n}\check{b}(z_0)))-\check{b}(z_0)). \nonumber
\end{eqnarray}
\begin{discuss}
{\colorr $t$を入れないと$\mathcal{Y}_0$がゼロにできないからモーメント評価が難しくなるか}
\end{discuss}
Then, $F_{n,\theta,j}$ in Section~\ref{LAMN-Malliavin-subsection} is given by 
\begin{equation}\label{F-def}
F_{n,\theta}=((\mathcal{X}_1^{n,\theta})^\top,(\mathcal{Y}_1^{n,\theta})^\top)^\top
\end{equation}
{\colorc with $m_n=n$ and $k_j=jm$ {\colore because}
  \EQN{\label{mcx-mcy-dist} &\bigg\{\MAT{c}{\mathcal{X}_t^{n,\theta} \\ \mathcal{Y}_t^{n,\theta}}\bigg\}_{t\geq 0} \\
   &\quad \overset{d}=\bigg\{\MAT{cc}{\sqrt{n} & 0 \\ 0 & n\sqrt{n}}
   \bigg(Y_{\frac{t+j-1}{n}}^\theta-Y_{\frac{j-1}{n}}^\theta-
   \MAT{c}{ 0 \\ t\check{b}(Y_{(j-1)/n}^\theta)/n}\bigg)\bigg\}_{t\geq 0}\bigg|_{Y_{(j-1)/n}^\theta=z_0}.
  }
}
\begin{discuss}
{\colorr $t'=nt-(j-1)$と変数変換すると
\begin{eqnarray}
X_j^{n,\theta}&=&\left(
\begin{array}{cc}
\sqrt{n} & 0 \\
0 & n\sqrt{n}
\end{array}
\right)\bigg\{\int_{(j-1)/n}^{j/n}\left(\begin{array}{c} 
\tilde{b}(Y_t^\theta,\theta) \\ \check{b}(Y_t^\theta)-\check{b}(Y_{(j-1)/n}^\theta) 
\end{array}\right)dt
+\int_{(j-1)/n}^{j/n}
\left(\begin{array}{c} \tilde{a}(Y_t^\theta,\theta) \\ 0 \end{array}\right)dW_t \bigg\} \nonumber \\
&=&\int_0^1\left(\begin{array}{c} 
n^{-1/2}\tilde{b}(Y_{t'/n+(j-1)/n}^\theta,\theta) \\
\sqrt{n}(\check{b}(Y_{t'/n+(j-1)/n}^\theta)-\check{b}(Y_{(j-1)/n}^\theta)) \end{array}\right)dt'
+\int_0^1\left(\begin{array}{c} \tilde{a}(Y_{t'/n+(j-1)/n}^\theta,\theta) \\ 0 \end{array}\right)dW_{t'} \nonumber \\
&\sim&\int_0^1\left(\begin{array}{c} 
n^{-1/2}\tilde{b}(x_0+n^{-1/2}\mathcal{X}_{t'}^{n,\theta},y_0+n^{-3/2}(\mathcal{Y}_{t'}^{n,\theta}+t'\sqrt{n}\check{b}(z_0)),\theta) \\
\sqrt{n}(\check{b}(x_0+n^{-1/2}\mathcal{X}_{t'}^{n,\theta},y_0+n^{-3/2}(\mathcal{Y}_{t'}^{n,\theta}+t'\sqrt{n}\check{b}(z_0)))-\check{b}(z_0)) \end{array}\right)dt' \nonumber \\
&&+\int_0^1\left(\begin{array}{c} \tilde{a}(x_0+n^{-1/2}\mathcal{X}_{t'}^{n,\theta},y_0+n^{-3/2}(\mathcal{Y}_{t'}^{n,\theta}+t'\sqrt{n}\check{b}(z_0)),\theta) \\ 0 \end{array}\right)dW_{t'}. \nonumber
\end{eqnarray}
なぜならば
\begin{equation*}
\int_{(j-1)/n}^{j/n}a_sdW_s=\lim_{m\to \infty}\sum_{k=1}^ma_{\frac{j-1}{n}+\frac{k-1}{nm}}(W_{\frac{j-1}{n}+\frac{k}{nm}}-W_{\frac{j-1}{n}+\frac{k-1}{nm}})
\sim \frac{1}{\sqrt{n}}(W_{k/m}-W_{(k-1)/m})\sim \int^1_0a_{(j-1)/n+t'/n}d\mathcal{W}_{t'},
\end{equation*}
}
\end{discuss}

We first assume the following condition, which is stronger than (C1).
\begin{description}
\item[Assumption (C1$'$).] (C1) is satisfied, $\partial_z^i\partial_\theta^j \tilde{a}$ and $\partial_z^i\partial_\theta^j \tilde{b}$
are bounded for $i\in\mathbb{Z}_+$ and $0\leq j\leq 3$, and 
  \EQQ{\sup_{z,\theta}\lVert (\tilde{a}{\colora \tilde{a}^\top})^{-1}(z,\theta)\rVert_{{\rm op}}<\infty.}
\end{description}
{\colore First, we} show Condition~(L) under (C1$'$) and (C2) by using Theorem~\ref{Malliavin-LAMN-thm}, and then {\colore we} weaken the assumptions to (C1) and (C2)
by the localization technique similar to Lemma~4.1 of Gobet~\cite{gob01}. \\

Define
\begin{equation}\label{tildeF-def}
\tilde{F}_{n,\theta}=\left(\begin{array}{c}
\tilde{a}(z_0,\theta)\mathcal{W}_1 \\
({\colore \nabla_1}\check{b})^\top(z_0)\tilde{a}(z_0,\theta)\int^1_0\mathcal{W}_tdt
\end{array}\right), 
\end{equation}
and
\begin{equation*}
\tilde{X}_j^{n,\theta}=\left(\begin{array}{c}
\sqrt{n}\tilde{a}(Y_{(j-1)/n}^\theta,\theta)(W_{j/n}-W_{(j-1)/n}) \\
n^{3/2}({\colore \nabla_1}\check{b})^\top(Y_{(j-1)/n}^\theta)\tilde{a}(Y_{(j-1)/n}^\theta,\theta)\int^{j/n}_{(j-1)/n}(W_t-W_{(j-1)/n})dt
\end{array}\right),
\end{equation*}
then we have 
\begin{equation*}
\partial_\theta \tilde{F}_{n,\theta}=\left(
\begin{array}{c}
\partial_\theta \tilde{a}(z_0,\theta)\mathcal{W}_1 \\
({\colore \nabla_1} \check{b})^\top(z_0)\partial_\theta\tilde{a}(z_0,\theta)\int^1_0\mathcal{W}_tdt
\end{array}
\right),
\end{equation*}
and $\tilde{K}_j$ in Section~\ref{LAMN-Malliavin-subsection} can be calculated as
\begin{discuss}
{\colorr $\tilde{X}_j^{n,\theta}$と$\tilde{F}_j$の対応は$t$を変数変換すれば$n^{3/2}\int_{(j-1)/n}^{j/n}(W_t-W_{(j-1)/n})dt\sim \int^1_0W_tdt$などとなるのでOK.}
\end{discuss}
\begin{equation*}
\tilde{K}(\theta)=\left(
\begin{array}{cc}
\tilde{a}\tilde{a}^\top(z_0,\theta) & (1/2)\tilde{a}\tilde{a}^\top{\colore \nabla_1}\check{b}(z_0,\theta) \\
(1/2)({\colore \nabla_1}\check{b})^\top\tilde{a}\tilde{a}^\top(z_0,\theta) & (1/3)({\colore \nabla_1}\check{b})^\top\tilde{a}\tilde{a}^\top{\colore \nabla_1}\check{b}(z_0,\theta)
\end{array}
\right).
\end{equation*}
\begin{discuss}
{\colorr $E[(W_{j/n}-W_{(j-1)/n})\int^{j/n}_{(j-1)/n}(W_t-W_{(j-1)/n})dt]=\int^{j/n}_{(j-1)/n}(t-(j-1)/n)dt=(2n^2)^{-1}$.}
\end{discuss}

We denote by $V^\perp$ the orthogonal complement of a subspace $V$ on a vector space.
Let $\tilde{B}_{i,\theta}=\partial_{\theta_i}\tilde{a}\tilde{a}^+(z_0,\theta)$ 
and ${\colora \mathcal{C}_{i,\theta}}=({\colore \nabla_1} \check{b})^\top \partial_{\theta_i}\tilde{a}\tilde{a}^+{\colore \nabla_1}\check{b}(({\colore \nabla_1}\check{b})^\top {\colore \nabla_1} \check{b})^{-1}(z_0,\theta)$.
Then {\colore because} $\tilde{a}^+\tilde{a}$ is a projection to $({\rm Ker}(\tilde{a}))^\perp$
\begin{discuss}
{\colorr 吉田Thm 2.4 and ${\rm Im}(\tilde{a}^\top)={\rm Ker}(\tilde{a})^\perp$}
\end{discuss}
and $({\rm Ker}(\partial_{\theta_i}\tilde{a}))^\perp\subset ({\rm Ker}(\tilde{a}))^\perp$ by (C2), we have
\begin{equation}\label{B3-check}
\tilde{B}_{i,\theta}\tilde{a}(z_0,\theta)=\partial_{\theta_i}\tilde{a}(z_0,\theta)
\end{equation}
for $1\leq i\leq d$ and $\theta\in\Theta$. Moreover, because ${\colore \nabla_1}\check{b}(({\colore \nabla_1}\check{b})^\top{\colore \nabla_1} \check{b})^{-1}({\colore \nabla_1}\check{b})^\top$
is a projection to $({\rm Ker}(({\colore \nabla_1}\check{b})^\top))^\perp$, (C2) yields
\begin{equation}\label{B3-check2}
{\colora \mathcal{C}_{i,\theta}}({\colore \nabla_1}\check{b})^\top\tilde{a}(z_0,\theta)=({\colore \nabla_1}\check{b})^\top \partial_{\theta_i}\tilde{a}\tilde{a}^+\tilde{a}(z_0,\theta)
=({\colore \nabla_1}\check{b})^\top \partial_{\theta_i}\tilde{a}(z_0,\theta)
\end{equation}
for $1\leq i\leq d$ and $\theta\in\Theta$. Then, by setting 
\begin{equation}\label{B-def}
B_{i,\theta}=\left(
\begin{array}{cc}
\tilde{B}_{i,\theta} & O_{\kappa, m-\kappa} \\
O_{m-\kappa,\kappa} & {\colora \mathcal{C}_{i,\theta}}
\end{array}
\right),
\end{equation}
(\ref{B3-check}) and (\ref{B3-check2}) yield $\partial_\theta \tilde{F}_{n,\theta}=B_{i,\theta}\tilde{F}_{n,\theta}$.\\

{\colore For an $\mathbb{R}^N$-valued random variable $V=(V_l)_{l=1}^N\in (\mathbb{D}_j^{1,p})^N$, we regard $D_tV$ as an $r\times N$ matrix.}
\begin{discuss}
{\colorr ${\rm Ker}(\tilde{a})\not\subset{\rm Ker}(\partial_{\theta_i}\tilde{a})$ならある$x$があって$\tilde{a}x=0$かつ$\partial_{\theta_i}\tilde{a}x\neq 0$.
よって任意の行列$B$に対して$\partial_{\theta_i}\tilde{a}x\neq B\tilde{a}x$となり，(B3)が成り立たない．
$\partial_{\theta_i}\tilde{a}=B\tilde{a}$となる$B$が存在する時，
$\partial_{\theta_i}\tilde{a}\tilde{a}^+=B\tilde{a}\tilde{a}^+=B$と書ける．
また，(B3)を満たすためには$({\colore \nabla_1}\check{b})^\top\partial_{\theta_i}\tilde{a}\tilde{a}^+=\mathcal{C}_i({\colore \nabla_1}\check{b})^\top$となる$C_i$も必要だが，
このような$\mathcal{C}_i$は
\begin{equation*}
\mathcal{C}_i=\mathcal{C}_i({\colore \nabla_1} \check{b})^\top {\colore \nabla_1} \check{b}(({\colore \nabla_1} \check{b})^\top {\colore \nabla_1} \check{b})^{-1}
=({\colore \nabla_1} \check{b})^\top \partial_{\theta_i} \tilde{a}\tilde{a}^+{\colore \nabla_1} \check{b}(({\colore \nabla_1} \check{b})^\top {\colore \nabla_1} \check{b})^{-1}
\end{equation*}
と書ける．$\tilde{a}^+=\tilde{a}^\top(\tilde{a}\tilde{a}^\top)^{-1}$
}
\end{discuss}

\begin{lemma}\label{degenerate-L-lemma}
Assume (C1$'$) and (C2). Then Condition~(L) is satisfied for $\{P_{\theta,n}\}_{\theta,n}$.
\end{lemma}

\begin{proof}
Thanks to Theorem~\ref{Malliavin-LAMN-thm}, it is sufficient to check (B1)--(B5).
Let $\tilde{a}_{n,t}=\tilde{a}_n(t,\mathcal{X}_t^{n,\theta},\mathcal{Y}_t^{n,\theta},\theta)$,
$\tilde{b}_{n,t}=\tilde{b}_n(t,\mathcal{X}_t^{n,\theta},\mathcal{Y}_t^{n,\theta},\theta)$,
and $\check{b}_{n,t}=\check{b}_n(t,\mathcal{X}_t^{n,\theta},\mathcal{Y}_t^{n,\theta})$. 

{\colore First, we} check (B1). (C1$'$) implies $\sup_{n,\theta}E[\sup_t|\mathcal{X}_t^{n,\theta}|^p]^{1/p}\leq C_p$.
Moreover, we obtain
  \EQ{\label{check-b-eq} \check{b}_{n,s}=\int^1_0\bigg(({\colore \nabla_1}\check{b})^\top\mathcal{X}_s^{n,\theta}+\frac{1}{n}({\colore \nabla_2}\check{b})^\top{\colorc \dot{\mathcal{Y}}_s^{n,\theta}}\bigg)
   \bigg(x_0+u\frac{\mathcal{X}_s^{n,\theta}}{\sqrt{n}},y_0+u\frac{{\colorc \dot{\mathcal{Y}}_s^{n,\theta}}}{n^{3/2}}\bigg)du,
  }
{\colorc where $\dot{\mathcal{Y}}_s^{n,\theta}=\mathcal{Y}_s^{n,\theta}+s\sqrt{n}\check{b}(z_0)$.}
Therefore, if ${\colore \nabla_2}\check{b}\equiv 0$, {\colore then}
\begin{equation}
\sup_{n,\theta}E[\sup_t|\mathcal{Y}_t^{n,\theta}|^p]^{1/p}\leq \sup_{n,\theta,t}E[|\check{b}_{n,t}|^p]^{1/p}\leq C_p.
\end{equation}
If $\check{b}$ is bounded, {\colore then}
\begin{equation*}
E[|\check{b}_{n,t}|^p]\leq C_p+C_pE[|\mathcal{Y}_t^{n,\theta}|^p]\leq C_p+C_p\int^t_0E[|\check{b}_{n,s}|^p]ds,
\end{equation*}
and hence Gronwall's inequality yields
\begin{equation*}
\sup_{n,\theta}E[\sup_t|\mathcal{Y}_t^{n,\theta}|^p]^{1/p}\leq \sup_{n,\theta,t}E[|\check{b}_{n,t}|^p]^{1/p}<\infty.
\end{equation*}
Then we have $\sup_{n,\theta}\lVert F_{n,\theta}\rVert_{0,p}<\infty$ under (C1$'$) and (C2).

Let $\mathcal{Z}_t^{n,\theta}=((\mathcal{X}_t^{n,\theta})^\top,(\mathcal{Y}_t^{n,\theta})^\top)^\top$.
Theorem~2.2.1 in~\cite{nua06} yields
\begin{eqnarray}
D_t\mathcal{Z}_r^{n,\theta}&=&(\tilde{a}_{n,t}^\top \quad O_{r,m-\kappa})
+\int^r_t(D_t\mathcal{Z}_s^{n,\theta}\partial_z\tilde{b}_{n,s} \quad D_t\mathcal{Z}_s^{n,\theta}\partial_z\check{b}_{n,s})ds \nonumber \\
&&+\bigg(\bigg(\int^r_t\sum_{k,l}[D_t\mathcal{Z}_s^{n,\theta}]_{ik}\partial_{z_k}[\tilde{a}_{n,s}]_{jl}d\mathcal{W}_s^l\bigg)_{ij} \quad O_{r,m-\kappa}\bigg) \nonumber
\end{eqnarray}
for $r\geq t$.
Then (C1$'$) and Gronwall's inequality yield $\sup_{n,\theta}E[\lVert DF_{n,\theta}\rVert_H^p]^{1/p}<\infty$.
By using Lemma~2.2.2 in~\cite{nua06}, we similarly obtain $\sup_{n,\theta}\lVert F_{n,\theta}\rVert_{k,p}<\infty$ for any $k\in\mathbb{Z}_+$ and $p\geq 1$.

Theorem~39 in Chapter V of Protter~\cite{pro90} yields
  \EQNN{\MAT{c}{\partial_\theta \mathcal{X}_t^{n,\theta} \\ \partial_\theta \mathcal{Y}_t^{n,\theta}}
   &=\int^t_0\MAT{c}{\partial_\theta \tilde{b}_{n,s}+(\partial_z\tilde{b}_{n,s})^\top\partial_\theta \mathcal{Z}_s^{n,\theta} \\
   (\partial_z\check{b}_{n,s})^\top \partial_\theta \mathcal{Z}_s^{n,\theta}}ds \\
   &\quad +\int^t_0\MAT{c}{\partial_\theta \tilde{a}_{n,s}+\sum_k\partial_\theta [\mathcal{Z}_s^{n,\theta}]_k\partial_{z_k}\tilde{a}_{n,s} \\ 0}d\mathcal{W}_s.
  }
Together with Theorem~2.2.1 and Lemma~2.2.2 in~\cite{nua06}, we have $\sup_{n,\theta}\lVert \partial_\theta F_{n,\theta}\rVert_{3,p}<\infty$.
Similarly, we obtain $\sup_{n,\theta}\lVert \partial_\theta^lF_{n,\theta}\rVert_{4-l,p}<\infty$ for any $p\geq 1$ and $0\leq l\leq 3$, which implies (B1).

{\colore Next, we show (B2).} Under (C1$'$) and (C2), we have
\begin{equation*}
\sup_{n,t,z,\theta}\left(\lVert \partial_z\tilde{b}_n(t,z,\theta)\rVert_{{\rm op}}\vee \lVert \partial_t^i\partial_z^j\check{b}_n(t,z)\rVert_{{\rm op}} 
\vee \max_{1\leq i\leq m}\lVert \partial_{z_i}^l\tilde{a}_n(t,z,\theta)\rVert_{{\rm op}}\right)<\infty
\end{equation*}
for $i\in\{0,1\}$, $j\in\{1,2,3\}$, and $l\in\{0,1\}$, and 
\begin{equation*}
\sup_{n,t,z,\theta}\left(\lVert ((\partial_x\check{b}_n)^\top\partial_x\check{b}_n)^{-1}(t,z)\rVert_{{\rm op}}\vee \lVert (\tilde{a}_n\tilde{a}_n^\top)^{-1}(t,z,\theta)\rVert_{{\rm op}}\right)<\infty.
\end{equation*}
Together with Proposition~\ref{one-interval-nondeg-lemma}, we obtain (B2), 
where $\epsilon_n=1/\sqrt{n}$, $m_n=n$, $\bar{k}_n=m$, and $\alpha_n$ is a constant independent of $n$.
Furthermore, by setting (\ref{B-def}), we have (N2) and $\partial_\theta \tilde{F}_{n,\theta}=B_{i,\theta}\tilde{F}_{n,\theta}$.

To verify (B3) with $\rho_n=1/\sqrt{n}$, we only need to check $\lVert \partial_\theta^lF_{n,\theta}-\partial_\theta^l\tilde{F}_{n,\theta}\rVert_{3-l,p}\leq C_p/\sqrt{n}$
for $l\in \{0,1\}$ and $p>1$. {\colore First, we} have
  \EQN{\label{F-tildeF-eq} F_{n,\theta}-\tilde{F}_{n,\theta}&=\int^1_0\MAT{c}{
   \tilde{b}_{n,s} \\ \check{b}_{n,s}-({\colore \nabla_1}\check{b})^\top(z_0)\tilde{a}(z_0,\theta)\mathcal{W}_s}ds \\
   &\quad +\int^1_0\MAT{c}{\tilde{a}_{n,s}-\tilde{a}(z_0,\theta) \\ 0} d\mathcal{W}_s.
  } 
{\colore Because}
\begin{equation}
\tilde{a}_{n,s}-\tilde{a}(z_0,\theta)
=\frac{1}{\sqrt{n}}\int^1_0\bigg({\colore \nabla_1}\tilde{a}\mathcal{X}_s^{n,\theta}+\frac{1}{n}{\colore \nabla_2}\tilde{a}{\colorc \dot{\mathcal{Y}}_s^{n,\theta}}\bigg)
\bigg(x_0+u\frac{\mathcal{X}_s^{n,\theta}}{\sqrt{n}},y_0+u\frac{{\colorc \dot{\mathcal{Y}}_s^{n,\theta}}}{n^{3/2}}\bigg)du, \nonumber
\end{equation} 
we have
\begin{equation}\label{X-tildeA-eq}
\sup_{n,\theta}E\left[\sup_t|\mathcal{X}_t^{n,\theta}-\tilde{a}(z_0,\theta)\mathcal{W}_t|^p\right]^{1/p}\leq \frac{C_p}{\sqrt{n}}.
\end{equation}
Together with (\ref{check-b-eq}), we have
\begin{equation}\label{check-b-eq2}
\sup_{n,\theta}E\left[\sup_t|\check{b}_{n,t}-({\colore \nabla_1}\check{b})^\top(z_0)\tilde{a}(z_0,\theta)\mathcal{W}_t|^p\right]^{1/p}\leq \frac{C_p}{\sqrt{n}}.
\end{equation}
(\ref{F-tildeF-eq})--(\ref{check-b-eq2}) yield $\lVert F_{n,\theta}-\tilde{F}_{n,\theta}\rVert_{0,p}\leq C_p/\sqrt{n}$.

Similarly to above, Theorem~39 in Chapter V of Protter~\cite{pro90} and Theorem~2.2.1 and Lemma~2.2.2 in~\cite{nua06} yield
$\lVert \partial_\theta^lF_{n,\theta}-\partial_\theta^l\tilde{F}_{n,\theta}\rVert_{3-l,p}\leq C_p/\sqrt{n}$ for $p\geq 1$ and $l\in\{0,1\}$,
and consequently (B3) holds.

\begin{discuss}
{\colorr
\begin{eqnarray}
D_tF_{n,\theta}-D_t\tilde{F}_{n,\theta}&=&\left(\begin{array}{c}
\tilde{a}_{n,t}^\top-\tilde{a}^\top (z_0,\theta) \\
0
\end{array}\right) +\int^1_t\left(\begin{array}{c}
D_t\mathcal{Z}_s^{n,\theta}\partial_z\tilde{b}_{n,s} \\
D_t\mathcal{Z}_s^{n,\theta}\partial_z\check{b}_{n,s}-\tilde{a}^\top{\colore \nabla_1}\check{b}(z_0,\theta)
\end{array}\right)ds \nonumber \\
&&+ \left(\begin{array}{c}
(\int^1_t\sum_{k,l}[D_t\mathcal{Z}_s^{n,\theta}]_{ik}\partial_{z_k}[\tilde{a}_{n,s}]_{jl}d\mathcal{W}_s^l)_{ij} \\
0
\end{array}\right). \nonumber
\end{eqnarray}
$\partial_{z_k}\tilde{a}_{n,s}$, $\tilde{a}_{n,t}^\top-\tilde{a}_0^\top$, $D_t\mathcal{Z}_s^{n,\theta}-\tilde{a}^\top$などから$1/\sqrt{n}$がでる．
\begin{eqnarray}
\partial_\theta F_{n,\theta}-\partial_\theta \tilde{F}_{n,\theta}&=&\int^1_0\left(\begin{array}{c}
\partial_\theta \tilde{b}_{n,s}+(\partial_z\tilde{b}_{n,s})^\top \partial_\theta \mathcal{Z}_s^{n,\theta} \\
(\partial_z \check{b}_{n,s})^\top \partial_\theta \mathcal{Z}_s^{n,\theta} -({\colore \nabla_1}\check{b})^\top(z_0)\partial_\theta \tilde{a}(z_0,\theta)\mathcal{W}_s
\end{array}\right)ds \nonumber \\
&&+\int^1_0\left(\begin{array}{c}
\partial_\theta \tilde{a}_{n,s}-\partial_\theta \tilde{a}(z_0,\theta)+\sum_k\partial_\theta [\mathcal{Z}_s^{n,\theta}]_k\partial_{z_k}\tilde{a}_{n,s} \\
0
\end{array}\right)d\mathcal{W}_s. \nonumber
\end{eqnarray}
}
\end{discuss}

Finally, we show (B4). We denote $\check{\varphi}(A)=({\colore \nabla_1}\check{b})^\top (z_0)A{\colore \nabla_1}\check{b}(z_0)$ for a $\kappa\times \kappa$ matrix $A$.
Let $S=(1/12)\check{\varphi}(\tilde{a}\tilde{a}^\top)(z_0,\theta)$.
Then for $\epsilon=\inf_{z,\theta}(\lVert (\tilde{a}\tilde{a}^\top)^{-1}(z,\theta)\rVert_{{\rm op}}^{-1})$,
$\epsilon'=\inf_z(\lVert (({\colore \nabla_1}\check{b})^\top{\colore \nabla_1}\check{b})^{-1}(z)\rVert_{{\rm op}}^{-1})$, and $y\in \mathbb{R}^{m-\kappa}$, we obtain
\begin{equation*}
y^\top \check{\varphi}(\tilde{a}\tilde{a}^\top)y
\geq \epsilon y^\top ({\colore \nabla_1}\check{b})^\top{\colore \nabla_1}\check{b}y \geq \epsilon \epsilon'|y|^2,
\end{equation*}
which implies that $S$ is positive definite.
Together with (0.8.5.6) in Horn and Johnson~\cite{hor-joh13}, we have
\begin{equation*}
\tilde{K}^{-1}(\theta)=\left(
\begin{array}{cc}
(\tilde{a}\tilde{a}^\top)^{-1}+(1/4){\colore \nabla_1}\check{b} S^{-1}({\colore \nabla_1}\check{b})^\top & -(1/2){\colore \nabla_1}\check{b}S^{-1} \\
-(1/2)S^{-1}({\colore \nabla_1} \check{b})^\top & S^{-1}
\end{array}
\right)(z_0,\theta).
\end{equation*}

Let $\tilde{K}_0=\tilde{K}(\theta_0)$ and $\mathcal{G}_t=\sigma(W_s;s\leq t)$. We can set $G_j^n=(G_{j,i}^n)_{1\leq i\leq d}$ and $\gamma_j$ in (B4) by
\begin{equation*}
G_{j,i}^n=\frac{1}{2}(\tilde{X}_j^{n,\theta_0})^\top (B_{i,\theta_0}^\top \tilde{K}_0^{-1}+\tilde{K}_0^{-1}B_{i,\theta_0})\tilde{X}_j^{n,\theta_0}-{\rm tr}(B_{i,\theta_0})
\end{equation*}
and 
\begin{equation*}
[\gamma(z_0)]_{ii'}=\frac{1}{2}{\rm tr}(\tilde{K}_0^{-1}(B_{i,\theta_0}\tilde{K}_0+\tilde{K}_0B_{i,\theta_0}^\top)
\tilde{K}^{-1}_0(B_{i',\theta_0}\tilde{K}_0+\tilde{K}_0B_{i',\theta_0}^\top)).
\end{equation*}

Repeated use of (\ref{B3-check}) and (\ref{B3-check2}) yields
\begin{equation}
B_{i,\theta_0}\tilde{K}_0+\tilde{K}_0B_{i,\theta_0}^\top
=\left(\begin{array}{cc}
\partial_{\theta_i}(\tilde{a}\tilde{a}^\top) & (1/2)\partial_{\theta_i}(\tilde{a}\tilde{a}^\top){\colore \nabla_1}\check{b} \\
(1/2)({\colore \nabla_1}\check{b})^\top\partial_{\theta_i}(\tilde{a}\tilde{a}^\top) & (1/3)\check{\varphi}(\partial_{\theta_i}(\tilde{a}\tilde{a}^\top))
\end{array}\right)(z_0,\theta_0),
\end{equation}
\begin{discuss}
{\colorr 
\begin{equation*}
B_{i,\theta_0}\tilde{K}_0=\left(\begin{array}{cc}
\partial_{\theta_i}\tilde{a}\tilde{a}^\top & (1/2)\partial_{\theta_i}\tilde{a}\tilde{a}^\top{\colore \nabla_1}\check{b} \\
(1/2)({\colore \nabla_1}\check{b})^\top \partial_{\theta_i}\tilde{a}\tilde{a}^\top & (1/3)\check{\varphi}(\partial_{\theta_i}\tilde{a}\tilde{a}^\top)
\end{array}\right),
\end{equation*}
\begin{equation*}
\tilde{K}_0B_{i,\theta_0}^\top=\left(\begin{array}{cc}
\tilde{a}\partial_{\theta_i}\tilde{a}^\top & (1/2)\tilde{a}\partial_{\theta_i}\tilde{a}^\top{\colore \nabla_1}\check{b} \\
(1/2)({\colore \nabla_1}\check{b})^\top \tilde{a}\partial_{\theta_i}\tilde{a}^\top & (1/3)\check{\varphi}(\tilde{a}\partial_{\theta_i}\tilde{a}^\top)
\end{array}\right).
\end{equation*}
}
\end{discuss}
and hence
  \EQN{
   &\tilde{K}_0^{-1}(B_{i,\theta_0}\tilde{K}_0+\tilde{K}_0B_{i,\theta_0}^\top) \\
   &\quad =\MAT{cc}{
    (\tilde{a}\tilde{a}^\top)^{-1}\partial_{\theta_i}(\tilde{a}\tilde{a}^\top) & (1/2)(\tilde{a}\tilde{a}^\top)^{-1}\partial_{\theta_i}(\tilde{a}\tilde{a}^\top){\colore \nabla_1}\check{b}-{\colorc \mathcal{S}_i} \\
    O_{m-\kappa,\kappa} & (1/12)S^{-1}\check{\varphi}(\partial_{\theta_i}(\tilde{a}\tilde{a}^\top))
   }(z_0,\theta_0),
  }
{\colorc where $\mathcal{S}_i=(1/24){\colore \nabla_1}\check{b}S^{-1}\check{\varphi}(\partial_{\theta_i}(\tilde{a}\tilde{a}^\top))(z_0,\theta_0)$.}
Together with {\colore the} equations
  \EQQ{U\partial_\theta (aa^\top) U^\top=\MAT{cc}{
    \partial_\theta(\tilde{a}\tilde{a}^\top) & O_{\kappa,m-\kappa} \\
    O_{m-\kappa,\kappa} & O_{m-\kappa,m-\kappa}
   }
  } 
and
  \EQQ{U(aa^\top)^+U^\top=(Uaa^\top U^\top)^+=\MAT{cc}{
    (\tilde{a}\tilde{a}^\top)^{-1} & O_{\kappa,m-\kappa} \\
    O_{m-\kappa,\kappa} & O_{m-\kappa,m-\kappa}
   }, 
  }
\begin{discuss}
{\colorr
\begin{equation*}
Uaa^\top U^\top=\left(
\begin{array}{cc}
\tilde{a}\tilde{a}^\top & 0 \\
0 & 0
\end{array}
\right), 
\quad U\partial_\theta a=\left(
\begin{array}{c}
\partial_\theta \tilde{a} \\
0
\end{array}
\right)
\end{equation*}
$a^+U^\top=(Ua)^+=(\tilde{a}^+ \ 0)$となることが，右辺がムーア・ペンローズの４条件を満たすことからわかる．}
\end{discuss}
we have
  \EQNN{[\gamma(z_0)]_{ii'}&={\rm tr}((aa^\top)^+\partial_{\theta_i}(aa^\top)(aa^\top)^+\partial_{\theta_{i'}}(aa^\top))(U^\top z_0,\theta_0)/2 \\
   &\quad +{\rm tr}(\check{\varphi}(\tilde{a}\tilde{a}^\top)^{-1}\check{\varphi}(\partial_{\theta_i}(\tilde{a}\tilde{a}^\top))
   \check{\varphi}(\tilde{a}\tilde{a}^\top)^{-1}\check{\varphi}(\partial_{\theta_{i'}}(\tilde{a}\tilde{a}^\top)))(z_0,\theta_0)/2.
  }

Therefore, we obtain
\begin{equation}\label{degenerate-Gamma-conv}
\frac{1}{n}\sum_{j=1}^n\gamma(UX_{(j-1)/n})\overset{P}\to \Gamma
\end{equation}
as $n\to \infty$.
Moreover, Lemma~\ref{B4-suff-lemma} {\colore yields}
\begin{equation}\label{degenerate-B4-check}
\frac{1}{n}\sum_{j=1}^{[nt]}E[G_j^n(G_j^n)^\top|\mathcal{G}_{(j-1)/n}]\overset{P}\to \Gamma_t,
\quad \frac{1}{n^2}\sum_{j=1}^nE[|G_j^n|^4|\mathcal{G}_{(j-1)/n}]\overset{P}\to 0
\end{equation}
for $t\in (0,1]$, where $\Gamma_t$ is defined by replacing {\colore the} interval of integration in the definition of $\Gamma$ with $[0,t]$.
Furthermore, it is easy to see that
  \EQNN{\frac{1}{\sqrt{n}}\sum_{j=1}^{[nt]}E[G_j^n(W_{j/n}-W_{(j-1)/n})^\top|\mathcal{G}_{(j-1)/n}]&=0,\\
   \frac{1}{\sqrt{n}}\sum_{j=1}^{[nt]}E[G_j^n(N_{j/n}-N_{(j-1)/n})|\mathcal{G}_{(j-1)/n}]&=0
  }
for any $t\in [0,1]$ and any bounded $(\mathcal{G}_t)_{t\in[0,1]}$-martingale $N=(N_t)_{t\in[0,1]}$ orthogonal to $W$.
Together with Theorem~3.2 in Jacod~\cite{jac97}, we have
$\sum_{j=1}^nG_j^n/\sqrt{n}\to \Gamma^{1/2}\mathcal{N}$
stably as $n\to\infty$, which implies (B4).

\end{proof}

\noindent
{\bf Proof of Theorem~\ref{degenerate-LAMN-thm}.}

For $q>0$, let $\phi_q^1:\mathbb{R}^\kappa\to \mathbb{R}^\kappa$ and $\phi_q^2:\mathbb{R}^{m-\kappa}\to \mathbb{R}^{m-\kappa}$ be $C^\infty$ functions with compact {\colore support}
satisfying $\phi_q^1(x)=x$ on $\{|x|\leq q\}$ and $\phi_q^2(y)=y$ on $\{|y|\leq q\}$.
Let $\phi_q(z)=(\phi_q^1(x),\phi_q^2(y))$, and let
\begin{equation*}
a_q(z,\theta)=U^\top {\colorc \MAT{c}{\tilde{a}(\phi_q(Uz),\theta) \\ 0}}, \quad b_q(z,\theta)=U^\top \left(\begin{array}{c}
\tilde{b}(\phi_q(Uz),\theta) \\
\check{b}(Uz)
\end{array}\right).
\end{equation*}
Let $P_{\theta,q,n}$ be the corresponding probability measure, 
and ${\colorp V_{q,n}}, T_{q,n}, W_q, T_q$ {\colore correspond} to ${\colorp V_n},T_n,W,T$, respectively.
Then (C1$'$) and (C2) are satisfied for the statistical model of $P_{\theta,q,n}$,
and hence this model satisfies Condition~(L) by Lemma~\ref{degenerate-L-lemma}.
\begin{discuss}
{\colorr $\tilde{a}_q(z,\theta)=\tilde{a}(\phi_q(z),\theta)$から$\tilde{a}_q\tilde{a}_q^\top$は非退化で${\colore \nabla_2}\tilde{a}\equiv 0$なら${\colore \nabla_2}\tilde{a}_q\equiv 0$．
$\tilde{b}_q=\tilde{b}(U\phi_q(U^\top z),\theta)$などから他の条件のOK．}
\end{discuss}
Moreover, we have
\begin{equation*}
\log\frac{dP_{\theta_0+h/\sqrt{n},q,n}}{dP_{\theta_0,q,n}}
-(h^\top {\colorp V_{q,n}}+\frac{1}{2}h^\top {\colorp T_{q,n}}h)\to 0
\end{equation*}
in $P_{\theta_0,q,n}$-probability, and $\mathcal{L}((T_{q,n},{\colorp V_{q,n}})|P_{\theta_0,q,n})\to \mathcal{L}((T_q,{\colorp T_q^{1/2}}W_q))$.
By letting $q\to \infty$, Proposition~4.3.2 in Le~Cam~\cite{lec86} yields {\colorp Condition~(L)} for $\{P_{\theta,n}\}_{\theta,n}$.
{\colorp Then Remark~\ref{Gamma-pd-rem} leads to the conclusion.}

\qed

\begin{discuss}
{\colorr
$Ub=(\tilde{b}^\top,\check{b}^\top)^\top$より，$(m-\kappa)\times m$行列$P$と$m\times m$行列$Q$を$[P]_{ij}=\delta_{i,j+\kappa}$, $[Q]_{ij}=\delta_{ij}1_{\{i\leq \kappa\}}$と定めると，
\begin{equation*}
U(\partial_z b)^\top=\left(\begin{array}{cc}
({\colore \nabla_1} \tilde{b})^\top & ({\colore \nabla_2}\tilde{b})^\top \\
({\colore \nabla_1} \check{b})^\top & ({\colore \nabla_2}\check{b})^\top
\end{array}\right), \quad
PU(\partial_zb)^\top Q=({\colore \nabla_1}\check{b}^\top \ 0).
\end{equation*}
}
\end{discuss}

\section{Partial observation models}\label{integrated-diffusion-LAMN-section}

In this section, we prove Theorem~\ref{partial-LAMN-thm} by using the results in Section~\ref{LAMN-Malliavin-subsection}.
It is difficult to apply the scheme in Section~\ref{LAMN-Malliavin-subsection} directly 
{\colore because} the conditional distribution $P(X_j^{n,\theta}\in \cdot |{\colore \bar{X}_{j-1}^{n,\theta}=\bar{x}_{j-1}})$ is complicated
{\colora when some components are hidden}.
Therefore, we follow the idea of~\cite{glo-gob08}, that is, we first show {\colorp Condition~(L)} of an augmented model 
obtained by adding some observations of {\colorn $(I_\kappa-{\colorc \mathcal{Q}})\tilde{Y}_t$} to the {\colorn partial} observations.
Then we show the LAMN property of the original model by approximating the log-likelihood ratios of the augmented model 
using functionals of the original observations.

\begin{discuss}
{\colorr ＜$\check{b}$が一般の場合＞
\EQQ{\check{Y}_{k/n}-\check{Y}_{(k-1)/n}=\int_{(k-1)/n}^{k/n}\check{b}(Y_t^\theta)dt
\approx \check{b}(Y_{k/n})/n+{\colore \nabla_1}\check{b}/n(\tilde{Y}_{k/n}-\tilde{Y}_{(k-1)/n})}
を使って観測されない変数$\tilde{Y}_{k/n}$を複製する必要があるが，差分形の和をとる必要があり誤差項が
積みあがる上，${\colore \nabla_1}\check{b}(z_0)$の$z_0$を変えていかないといけないので一般化逆行列をとってもうまくいかない

＜ムーアペンローズの一般化逆行列の性質＞\\
$(AA^+)^\top=AA^+$, $(A^+A)^\top =A^+A$, $A$が正方行列なら$A^+A$は${\rm Im} A^\top$への射影
}
\end{discuss}

\subsection{An augmented model}

We divide the whole observation interval $[0,1]$ into several blocks
and show {\colorp Condition~(L)} for an augmented model that is obtained by adding an observation of {\colorn $(I_\kappa-{\colorc \mathcal{Q}})\tilde{Y}_t$} for each block.
Let $({\colorc e_n})_{n=1}^\infty$ be a sequence of positive integers {\colore that} diverges to infinity very slowly
({\colore the} precise diverge rate of ${\colorc e_n}$ is specified in {\colorn Lemma~\ref{partial-obs-L-lemma}}).
Let $L_n=\lfloor (n-1)/{\colorc e_n}\rfloor$ {\colora and $t_{j,k}=(k+j{\colorc e_n})/n$}. We consider an augmented model 
{\colora generated by observation blocks} 
  \EQ{\label{block-obs} {\colorn \{\{{\colorc \mathcal{Q}}\tilde{Y}_{{\colora t_{j,k}}},{\colorc \check{Y}_{{\colora t_{j,k}}}}\}_{k=1}^{{\colorc e_n}},(I_\kappa-{\colorc \mathcal{Q}})\tilde{Y}_{{\colora t_{j+1,0}}}\}}}
{\colora for $0\leq j\leq L_n-1$ and 
  \EQ{\label{block-obs2} \{{\colorc \mathcal{Q}}\tilde{Y}_{t_{L_n,k}},{\colorc \tilde{Y}_{t_{L_n,k}}}\}_{k=1}^{n-{\colorc e_n}L_n}.}
}
\begin{discuss}
{\colorr ${\colorc e_n}L_n=n$となると困るから$n-1$. $(n-1)/{\colorc e_n}-1<L_n\leq (n-1)/{\colorc e_n}$より${\colorc e_n}L_n<n\leq {\colorc e_n}(L_n+1)$.}
\end{discuss}

{\colora
{\colore Because} $I_\kappa-{\colorc \mathcal{Q}}$ and {\colorc $B^+B$} are projections 
to $({\rm Im}({\colorc \mathcal{Q}}))^\perp$ and $({\rm Ker}({\colorc B}))^\perp$, respectively, {\colorc (C4) implies}
%\EQQ{({\rm Ker}({\colorc B}))^\perp\supset ({\rm Im}({\colorc \mathcal{Q}}))^\perp,}
\begin{discuss}
{\colorr ${\colorc \mathcal{Q}}$は射影だから$I-{\colorc \mathcal{Q}}$も射影}
\end{discuss}
\EQ{\label{Q-eq} {\colorc \tilde{Q}_3B^+}B=\tilde{Q}_3.}

Then, we can approximate
\EQQ{n{\colorc \tilde{Q}_3B^+}(\check{Y}_{k/n}-\check{Y}_{(k-1)/n})=n{\colorc \tilde{Q}_3B^+}B\int_{(k-1)/n}^{k/n}\tilde{Y}_t dt 
\approx \tilde{Q}_3\tilde{Y}_{k/n}.}

Therefore,} {\colorn we set $X_j^{n,\theta}$ in Section~\ref{LAMN-Malliavin-subsection} as 
  \EQNN{X_j^{n,\theta}=&\Big\{\big\{\sqrt{n}\tilde{Q}_1\Delta_{n,j,1}\tilde{Y},{\colorc n^{3/2}}(\Delta_{n,j,1}\check{Y}-B\tilde{Y}_{t_{j,0}}/n)\big\}, \\
   &~~ \big\{\sqrt{n}\tilde{Q}_1\Delta_{n,j,k}\tilde{Y},{\colorc n^{3/2}}\Delta_{n,j,k}^2\check{Y}\big\}_{k=2}^{{\colorc e_n}}, 
   \sqrt{n}(\tilde{Q}_3{\colorc \tilde{Y}_{t_{j+1,0}}}-n{\colorc \tilde{Q}_3B^+}\Delta_{n,j,{\colorc e_n}}\check{Y})\Big\} }
{\colora for $0\leq j\leq L_n-1$, and 
  \EQNN{X_{L_n}^{n,\theta}&=\Big\{\big\{\sqrt{n}\tilde{Q}_1\Delta_{n,L_n,1}\tilde{Y}, {\colorc n^{3/2}}(\Delta_{n,L_n,1}\check{Y}-B\tilde{Y}_{t_{L_n,0}}/n)\big\}, \\
   &\quad\quad  \big\{\sqrt{n}\tilde{Q}_1\Delta_{n,L_n,k}\tilde{Y},{\colorc n^{3/2}}\Delta_{n,L_n,k}^2\check{Y}\big\}_{2\leq k\leq n-{\colorc e_n}L_n}\Big\},}
which are obtained as a linear transformation of block observations (\ref{block-obs}) and (\ref{block-obs2}).}
Here we denote 
  $\Delta_{n,j,l} V =V_{t_{j,l}}-V_{t_{j,l-1}}$, 
  $\Delta_{n,j,l'}^2 V =\Delta_{n,j,l'}V-\Delta_{n,j,l'-1}V$
for $l\geq 1$, $l'\geq 2$ and a stochastic process $V=(V_t)_{t\in [0,1]}$.
%{\colora If $q_1=\kappa$, we define $X_j^{n,\theta}$ without the last elements for $0\leq j\leq L_n-1$.}

\begin{discuss}
{\colorr Augは$\tilde{Q}_3\tilde{Y}_{t_{j,0}}$を付け足したもの．線形変換で$X_j^{n,\theta}$に変わる．$\tilde{Q}_3\tilde{Y}_1$は付け加えない}
\end{discuss}
\begin{discuss}
{\colorr 尤度関数の不変性はgeneralNoteを参照}
\end{discuss}

{\colora {\colorc Thanks to (\ref{mcx-mcy-dist}),} {\colore the} corresponding $F_{n,\theta,j}$ and $\tilde{F}_{n,\theta,j}$ are defined by}
  \EQQ{F_{n,\theta}=\big\{\tilde{Q}_1\Delta_k{\colorc \mathcal{X}^{n,\theta}},{\colorc \Delta_k^2\mathcal{Y}^{n,\theta}}\big\}_{k=1}^{{\colora {\colorc e_n}}}
   \cup \big\{\tilde{Q}_3{\colorc \mathcal{X}^{n,\theta}_{{\colorc e_n}}}-{\colorc \tilde{Q}_3B^+}({\colorc \mathcal{Y}^{n,\theta}_{{\colorc e_n}}}-{\colorc \mathcal{Y}^{n,\theta}_{{\colorc e_n}-1}})\big\}}
and 
  \EQNN{\tilde{F}_{n,\theta}
   &=\bigg\{\tilde{Q}_1\tilde{a}({\colorc x_0},\theta)\Delta_k\mathcal{W},\tilde{Q}_2\tilde{a}({\colorc x_0},\theta)\int^1_0(\mathcal{W}_{t+k-1}-\mathcal{W}_{(t+k-2)\vee 0})dt\bigg\}_{k=1}^{{\colorc e_n}} \\
   &\quad \quad \bigcup\bigg\{\tilde{Q}_3\tilde{a}({\colorc x_0},\theta)\int^1_0(\mathcal{W}_{{\colorc e_n}}-\mathcal{W}_{t+{\colorc e_n}-1})dt\bigg\}}
{\colora for $0\leq j\leq L_n-1$, where $\mathcal{X}_t^{n,\theta}$, $\mathcal{Y}_t^{n,\theta}$ {\colorc are defined} in Section~\ref{degenerate-lamn-section},
  \EQQ{F'_{n,\theta}=\big\{\tilde{Q}_1\Delta_k{\colorc \mathcal{X}^{n,\theta}},{\colorc \Delta_k^2\mathcal{Y}^{n,\theta}}\big\}_{k=1}^{{\colora n-{\colorc e_n}L_n}},}
and 
  \EQQ{\tilde{F}'_{n,\theta}=\bigg\{\tilde{Q}_1\tilde{a}({\colorc x_0},\theta)\Delta_k\mathcal{W},\tilde{Q}_2\tilde{a}({\colorc x_0},\theta)\int^1_0(\mathcal{W}_{t+k-1}-\mathcal{W}_{(t+k-2)\vee 0})dt\bigg\}_{k=1}^{n-{\colorc e_n}L_n}.}
Here, we denote 
  $\Delta_l \mathcal{V} =\mathcal{V}_l-\mathcal{V}_{l-1}$, {\colorc $\Delta_1^2\mathcal{V}=\Delta_1\mathcal{V}$, and}
  $\Delta_{l'}^2 \mathcal{V} =\Delta_{l'}\mathcal{V}-\Delta_{l'-1}\mathcal{V}$
for $l\geq 1$, $l'\geq 2$, and a stochastic process $\mathcal{V}={\colorc (\mathcal{V}_t)_{t\geq 0}}$.
}

\begin{discuss}
{\colorr 
  \EQQ{\Delta_k^2{\colorc \mathcal{Y}^{n,\theta}}=n^{3/2}\bigg\{\bigg(\check{Y}_{t_{j,k}}-\check{Y}_{t_{j,k-1}}-\frac{B\tilde{Y}_{t_{j,0}}}{n}\bigg)
   -\bigg(\check{Y}_{t_{j,k-1}}-\check{Y}_{t_{j,k-2}}-\frac{B\tilde{Y}_{t_{j,0}}}{n}\bigg)\bigg\}
   =n^{3/2}\Delta_{n,j,k}^2\check{Y}
  }
  \EQQ{\tilde{Q}_3\mathcal{X}_{{\colorc e_n}}-{\colorc \tilde{Q}_3B^+}(\mathcal{Y}_{{\colorc e_n}}-\mathcal{Y}_{{\colorc e_n}-1})
   \sim \tilde{Q}_3(\sqrt{n}(\tilde{Y}_{t_{j+1,0}}-\tilde{Y}_{t_{j,0}})
   -{\colorc \tilde{Q}_3B^+}(n^{3/2}(\Delta_{n,j,{\colorc e_n}}\check{Y}-B\tilde{Y}_{t_j,0}/n))
   =\sqrt{n}(\tilde{Q}_3\tilde{Y}_{t_{j+1,0}}-n{\colorc \tilde{Q}_3B^+}\Delta_{n,j,{\colorc e_n}}\check{Y}).
  }
}
\end{discuss}

{\colora Moreover, we define} 
  \EQNN{\tilde{X}_j^{n,\theta}&=\bigg(\Big(\sqrt{n}\tilde{Q}_1\tilde{a}_j\Delta_{n,j,k}W,
   n^{3/2}\tilde{Q}_2\tilde{a}_j\int_{t_{j,k-1}}^{t_{j,k}}(W_t-W_{(t-{\colorc 1/n})\vee t_{j,0}})dt\Big)_{k=1}^{{\colorc e_n}}, \\
   &\qquad \quad n^{3/2}\tilde{Q}_3\tilde{a}_j\int_{t_{j,{\colorc e_n}-1}}^{t_{j,{\colorc e_n}}}(W_{t_{j,{\colorc e_n}}}-W_t)dt\bigg)
  }
{\colora for $0\leq j\leq L_n-1$, and $\tilde{X}_{L_n}^{n,\theta}$ similarly, where $\tilde{a}_j=\tilde{a}(\tilde{Y}_{t_{j,0}},\theta)$.}
\begin{discuss}
{\colorr $X_{L_n}^{n,\theta}$の作り方は$1\leq k\leq n-{\colorc e_n}L_n-1$までは同じで$k\geq n-{\colorc e_n}L_n$の時は$0$.}
\end{discuss}

\subsection{Condition~(L) for the augmented model}\label{aug-L-subsection}

{\colore First, we} show Condition~(L) for the augmented model under a stronger condition.

\begin{description}
\item[Assumption (C2$''$).] (C2$'$) is satisfied, {\colora $\sup_{x,\theta}\lVert (\tilde{a}\tilde{a}^\top)^{-1}(x,\theta)\rVert_{{\rm op}}<\infty$}, 
and $\partial_x^i\partial_\theta^k{\colora \tilde{a}}(x,\theta)$ and $\partial_x^i\partial_y^j\partial_\theta^k{\colora \tilde{b}}(x,y,\theta)$ are bounded
{\colora for $i,j\in \mathbb{Z}_+$ and $0\leq k\leq 3$}.
{\colorc \item[Assumption (C5$'$).] (C5) holds, $g$ is bounded, and the convergence (\ref{C5-eq}) holds uniformly in $x$.}
\end{description}
{\colora Let $P_{\theta,n}^{{\rm aug}}$ be the probability measure induced by the augmented model.}
\begin{lemma}\label{partial-obs-L-lemma}
Assume (C2$''$), (C4), and (C5$'$). Then Condition~(L) is satisfied for $\{P_{\theta,n}^{{\rm aug}}\}_{\theta,n}$
{\colora by setting suitable $({\colorc e_n})_{n\in\mathbb{N}}$}.
\end{lemma}

\begin{proof}
{\colora We check (B1)--(B5) in Theorem~\ref{Malliavin-LAMN-thm}.
First, {\colore the} boundedness of $\tilde{a}_n$ and $\tilde{b}_n$ yields}
  \EQQ{{\colorc \sup_{\theta,t\in [0,{\colorc e_n}]}E[|\mathcal{X}^{n,\theta}_t-\mathcal{X}^{n,\theta}_{(t-1)\vee 0}|^q]<\infty, 
   \quad \sup_{k,\theta}E[|\Delta_k^2\mathcal{Y}^{n,\theta}|^q]<\infty.}}
\begin{discuss}
{\colorr グロンウォールの不等式より$E[|D_t\mathcal{Z}_r|^p]\leq Ce^{{\colorc e_n}}$. 
  \EQQ{D_t(\mathcal{X}_r-\mathcal{X}_{r-1}=\tilde{a}_t1_{[r-1,r]}(t)+\int^r_tD_t\mathcal{Z}_s\partial_z\tilde{b}_{n,s}ds
   +\int^r_tD_t\mathcal{Z}_s\partial_z\tilde{a}_{n,s}d\mathcal{W}_s
}
で$\partial_z\tilde{b},\partial_z\tilde{a}\leq C/\sqrt{n}$なので，グロンウォールより
  \EQQ{E[|D_t(\mathcal{X}_r-\mathcal{X}_{r-1}|^p]
   \leq \left\{\begin{array}{cc} \frac{C}{\sqrt{n}}\exp(2{\colorc e_n}/\sqrt{n})=o({\colorc e_n^{-1}}) & {\rm if} \ t< r-1 \\
   C\exp({\colorc e_n}/\sqrt{n})\leq C & {\rm if} \ t\geq r-1 \end{array} \right.
  }
よって$\sup_r\lVert \mathcal{X}_r-\mathcal{X}_{r-1}\rVert_{l,p}\leq C_p$.
$D_t\mathcal{Z}_r=\mathcal{U}_r\mathcal{U}_{r-l}^{-1}(D_t\mathcal{Z}_{r-l})^\top$を使った評価も可．
同様に$\sup_r\lVert \mathcal{X}_r-\mathcal{X}_{r-1}\rVert_{l,p}\leq C_{l,p}$ for $l\in\mathbb{N}$.
ゆえに
  \EQQ{\sup_k\lVert \Delta_k^2\mathcal{Y}\rVert_{l,p}\leq C_p, 
   \quad \lVert \tilde{Q}_3\mathcal{X}_{{\colorc e_n}}-{\colorc \tilde{Q}_3B^+}(\mathcal{Y}_{{\colorc e_n}}-\mathcal{Y}_{{\colorc e_n}-1})\rVert_{l,p}\leq C_p.
  }
$\lVert\partial_\theta^iF\rVert_{l,p}$の評価も同様．$D\partial_\theta F$は$D\mathcal{Z}$項などは$\int^k_{k-1}$をとっても問題ない．
$D\partial \mathcal{Z}$項は$D\mathcal{Z}$と同様グロンウォールを二回使う．

$D\mathcal{Z}$に対して使うGronwall's inequalityは積分範囲が$1$か$2$なのでOK.
\EQQ{\partial_\theta{\colorc \mathcal{Y}^{n,\theta}_t}=\int^t_0B\partial_\theta\mathcal{X}_sds, 
\quad \partial_\theta \mathcal{X}_t=O({\colorc e_n})+O(n^{-1/2})\int^t_0\partial_\theta \mathcal{Z}_sd\mathcal{W}_s}
より$E[|\partial_\theta \mathcal{X}_t|^q]^{1/q}\leq {\colorc e_n}\leq C_p{\colorc e_n}$.}
\end{discuss}
{\colora Then by a similar argument to Lemma~\ref{degenerate-L-lemma}, we have
  {\colorc \EQQ{\lVert \partial_\theta^{l'}{\colore [F_{n,\theta}]_k}\rVert_{l,p}\vee \lVert \partial_\theta^{l'}{\colore [F'_{n,\theta}]_k}\rVert_{l,p}<\infty}} 
for any $l\in \mathbb{N}$, {\colorc $l'\in\mathbb{N}$,} $p\geq 1$, and $1\leq k\leq {\colorc e_n}$.
\begin{discuss2}
{\colorg ここはリバイズで丁寧にやるか}
\end{discuss2}

{\colorc Let $\tilde{C}_{i,\theta}(x)=B\partial_{\theta_i}\tilde{a}\tilde{a}^+(x,\theta)B^\top(BB^\top)^{-1}$.}
(\ref{B3-check2}), (C2$''$), and (C4) yield}
  \EQN{\label{Q-commute-eq} \tilde{Q}_1\partial_\theta\tilde{a}&=R_1{\colorc \mathcal{Q}}\partial_\theta \tilde{a}\tilde{a}^+\tilde{a}=R_1\partial_\theta\tilde{a}\tilde{a}^+R_1^{-1}\tilde{Q}_1\tilde{a}, \\
   \tilde{Q}_2\partial_{\theta_i}\tilde{a}&={\colorc \tilde{C}_{i,\theta}B\tilde{a}=\tilde{C}_{i,\theta}\tilde{Q}_2\tilde{a}}, \\
   \tilde{Q}_3\partial_\theta \tilde{a}&=R_3(I_\kappa-{\colorc \mathcal{Q}})\partial_\theta \tilde{a}\tilde{a}^+\tilde{a}
   =R_3\partial_\theta \tilde{a}\tilde{a}^+(I_\kappa-{\colorc \mathcal{Q}})\tilde{a}
   =R_3\partial_\theta \tilde{a}\tilde{a}^+R_3^{-1}\tilde{Q}_3\tilde{a},}
{\colora and consequently we obtain $\partial_{\theta_i}\tilde{F}_{n,\theta}=B_{i,\theta}\tilde{F}_{n,\theta}$ for }
  \EQQ{B_{i,\theta}{\colorc (x_0)}={\rm diag}((R_1\partial_{\theta_i} \tilde{a}\tilde{a}^+R_1^{-1},{\colorc \tilde{C}_{i,\theta}})_{k=1}^{{\colorc e_n}},R_3\partial_{\theta_i}\tilde{a}\tilde{a}^+R_3^{-1}){\colora ({\colorc x_0},\theta)}.}

{\colora Moreover, similarly to the argument in Section~\ref{degenerate-lamn-section}, we have
  $\lVert \partial_\theta^l {\colorc [F_{n,\theta}]_k}-\partial_\theta^l{\colorc [\tilde{F}_{n,\theta}]_k}\rVert_{3-l,p}\leq C_p{\colorc e_n}/\sqrt{n}$
{\colorc for $1\leq k\leq q$}.
\begin{discuss}
{\colorr $B'_{i,\theta}={\rm diag}(()_{k=1}^{n-{\colorc e_n}L_n})$に対し，$\partial_\theta\tilde{F}'=B'\tilde{F}'$.}
\end{discuss}
For $2\leq k\leq {\colorc e_n}$, we have}
  \EQQ{{\colorc ([F_{n,\theta}-\tilde{F}_{n,\theta}]_{(k-1)q+l})_{l=1}^{q_1}}=\tilde{Q}_1\bigg(\int_{k-1}^k\tilde{b}_{n,s}ds+\int^k_{k-1}(\tilde{a}_{n,s}-\tilde{a}({\colorc x_0},\theta))d\mathcal{W}_s\bigg),}
  \EQNN{{\colorc ([F_{n,\theta}-\tilde{F}_{n,\theta}]_{(k-1)q+q_1+l})_{l=1}^{q_2}}
   &=\tilde{Q}_2\int_{k-1}^k({\colorc \mathcal{X}^{n,\theta}_t}-{\colorc \mathcal{X}^{n,\theta}_{t-1}}-{\colorc \tilde{a}(x_0,\theta)}(\mathcal{W}_t-\mathcal{W}_{t-1}))dt,}
{\colorc and}
  \EQNN{&{\colorc ([F_{n,\theta}-\tilde{F}_{n,\theta}]_{{\colorc e_n}q+l})_{l=1}^{\kappa-q_1}} \\
   &\quad =\tilde{Q}_3\bigg({\colorc \mathcal{X}^{n,\theta}_{{\colorc e_n}}}-\int_{{\colorc e_n}-1}^{{\colorc e_n}}{\colorc \mathcal{X}^{n,\theta}_t}dt-{\colorc \tilde{a}(x_0,\theta)}\int_{{\colorc e_n}-1}^{{\colorc e_n}}(\mathcal{W}_{{\colorc e_n}}-\mathcal{W}_t)dt\bigg) \\
   &\quad =\tilde{Q}_3\int_{{\colorc e_n}-1}^{{\colorc e_n}}({\colorc \mathcal{X}^{n,\theta}_{{\colorc e_n}}}-{\colorc \tilde{a}(x_0,\theta)}\mathcal{W}_{{\colorc e_n}}-{\colorc \mathcal{X}^{n,\theta}_t}+{\colorc \tilde{a}(x_0,\theta)}\mathcal{W}_t)dt.}
\begin{discuss}
{\colorr $\tilde{a}_{n,s}-\tilde{a}_0$に微積分学の基本定理を使えばいい．
$D$がついたら$\partial_z\tilde{b}$, $\partial_x\tilde{a}$等がつくからOK. $\partial_\theta$もOK.
$\mathcal{X}_t-\tilde{a}(x_0,\theta)\mathcal{W}$もやはり似たような$\tilde{a}_{n,s}-\tilde{a}_0$を含む形でOK.}
\end{discuss}

{\colora Then, by a similar argument to the proof of Lemma~\ref{degenerate-L-lemma}, we have}
  $\lVert \partial_\theta^l [F_{n,\theta}]_k-\partial_\theta^l[\tilde{F}_{n,\theta}]_k\rVert_{3-l,p}\leq C_p{\colorc e_n}/\sqrt{n}$ 
for {\colore $1\leq k\leq e_nq+\kappa-q_1$}, {\colora and a similar estimate for 
  $\partial_\theta^l [F'_{n,\theta}]_k-\partial_\theta^l[\tilde{F}'_{n,\theta}]_k$.} 
{\colora Hence we obtain (B3) with $\rho_n={\colorc e_n}/\sqrt{n}$.

Let $\mathcal{K}_n{\colorc (\theta)}$ be the Malliavin matrix of
  $\{\{\tilde{Q}_1{\colorc \mathcal{X}^{n,\theta}_j},{\colorc \mathcal{Y}^{n,\theta}_j}\}_{j=1}^{{\colorc e_n}},\tilde{Q}_3{\colorc \mathcal{X}^{n,\theta}_{{\colorc e_n}}}\}$.}
{\colore By} Proposition~\ref{block-nondeg-lemma}, (0.8.5.3) in Horn and Johnson~\cite{hor-joh13}, {\colore and} 
the fact that {\colora $\lVert D{\colorc \mathcal{X}^{n,\theta}_j}\rVert_{1,p}$ and $\lVert D{\colorc \mathcal{Y}^{n,\theta}_j}\rVert_{1,p}$ are bounded},
we have $E[\det \mathcal{K}_n|^{-p}]\leq \tilde{c}(p,{\colorc e_n})$, where $\tilde{c}(p,{\colorc e_n})$ is a positive constant 
depending {\colore on} $p$ and ${\colorc e_n}$.
\begin{discuss}
{\colorr $\mathcal{K}_n$は$\gamma_{{\colorc e_n}}$ in Appendix Cの部分行列なので$(\det \mathcal{K}_n)^{-1}=(\det\gamma_{{\colorc e_n}})^{-1}\det G$
($G$は吉田教科書のブロック行列公式の$G$).}
\end{discuss}

{\colore Because} there exists an invertible matrix $M_n$ {\colorc depending on $e_n$} such that 
  ${\colora \tilde{K}_n(\theta)=\langle D\tilde{F}_{n,\theta},D\tilde{F}_{n,\theta}\rangle}=M_n\mathcal{K}_nM_n^\top$,
there exists a constant $c(p,{\colorc e_n})$ such that $\sup_\theta E[|\det {\colora \tilde{K}_n}(\theta)|^{-p}]\leq c(p,{\colore e_n})$.
{\colora {\colore Because} we also have a similar estimate for $\langle D\tilde{F}'_{n,\theta},D\tilde{F}'_{n,\theta}\rangle$,
\begin{discuss}
{\colorr ${\colorc \mathcal{Q}}$を射影する段階で$k$の数を削れば同様．}
\end{discuss}
we have (B2) and (\ref{rho-condition}) by letting ${\colorc e_n}$ {\colore diverge} to infinity sufficiently slowly.
Similarly, we have (N2) {\colorc (the upper bound of $\sup_\theta E_j[|\det K_j^{-1}(\theta)|^p]$ can depend on $n$ in (N2))}.}
\begin{discuss}
{\colorr 任意の$p$に対して$c(p,e_n)\leq n^\delta$となるようなコントロールは無理だがB2は$p=8$だけ．N2はupper boundが$n$に依存していい．}
\end{discuss}

{\colora Now we only need to show (B4).}
We have $\tilde{K}_n(\theta)=\psi_{{\colorc e_n}}^{2,2}({\colora \tilde{a}\tilde{a}^\top(z_0,\theta)})$, 
and $B_{i,\theta}\tilde{K}_n(\theta)+\tilde{K}_n(\theta)B_{i,\theta}^\top=\psi_{{\colorc e_n}}^{2,2}(\partial_{\theta_i}({\colora \tilde{a}\tilde{a}^\top({\colorc x_0},\theta)}))$ 
{\colora by (\ref{Q-commute-eq}).}

{\colore We define $G_j^n=(G_{j,i}^n)_{i=1}^d$ in Section~\ref{LAMN-Malliavin-subsection} by}
  \EQN{\label{G-def} G_{j,i}^n&=\frac{1}{2}(\tilde{X}_j^{n,\theta_0})^\top(B_{i,\theta_0}^\top\tilde{K}_n^{-1}{\colora (\theta_0)}
   +\tilde{K}_n^{-1}{\colora (\theta_0)}{\colorc B_{i,\theta_0}}){\colora \big|_{{\colorc x_0=\tilde{Y}_{t_{j,0}}}}}\tilde{X}_j^{n,\theta_0} \\
   &\quad -{\rm tr}(B_{i,\theta_0}){\colora \big|_{{\colorc x_0=\tilde{Y}_{t_{j,0}}}}}.}
{\colora Then, $\gamma_j(\bar{x}_{j-1})$ in Section~\ref{LAMN-Malliavin-subsection} is calculated as 
  \EQ{[\gamma_j({\colorc x_0})]_{ii'}=\frac{1}{2}{\rm tr}(\psi_{{\colorc e_n}}^{2,2}(\tilde{a}\tilde{a}^\top)^{-1}
   \psi_{{\colorc e_n}}^{2,2}(\partial_{\theta_i}(\tilde{a}\tilde{a}^\top))\psi_{{\colorc e_n}}^{2,2}(\tilde{a}\tilde{a}^\top)^{-1}
   \psi_{{\colorc e_n}}^{2,2}(\partial_{\theta_{i'}}(\tilde{a}\tilde{a}^\top)))({\colorc x_0},\theta_0)}
for $0\leq j\leq L_n-1$, and $\gamma_{L_n}({\colorc x_0})$ is similarly calculated with the estimate
  \EQ{\label{gamma-Ln-est} \frac{1}{n}|[\gamma_{L_n}(\tilde{Y}_{t_{L_n,0}})]_{ii'}|\leq \frac{{\colorc C}\alpha_n^2{\colorc e_n^2\cdot e_n^2}}{n}\to 0}
as $n\to \infty$ {\colorc by Lemma~\ref{tildeK-est-lemma}.}

Therefore, {\colorc (C5$'$) implies that there exists $n_0\in\mathbb{N}$ such that 
  \EQQ{\sup_{n\geq n_0}\bigg(\frac{1}{n}\sum_{j=0}^{L_n}E[|\gamma_j(\tilde{Y}_{t_{j,0}})|]\bigg)<\infty,}
and}
  \EQ{\frac{1}{n}\sum_{j=0}^{{\colorc L_n}}[{\colorc \gamma_j}(\tilde{Y}_{t_{j,0}})]_{ii'}
   -\frac{1}{{\colorc 2}}\sum_{j=0}^{L_n-1}\int_{t_{j,0}}^{t_{j+1,0}}[g(\tilde{Y}_t)]_{ii'}dt\overset{P}\to 0.}
Lemma~\ref{B4-suff-lemma}, Theorem~3.2 in Jacod~\cite{jac97}, and the inequality 
  \EQQ{\lVert\psi_{e_n}^{2,2}(\partial_{\theta_i}(\tilde{a}\tilde{a}^\top))(x_0,\theta)\rVert_{{\rm op}}<C{\colore e_n}}
yield (B4) with $\mathcal{G}_{\colorc j}=\sigma(W_s;s\leq {\colorc t_{j,0}})$.
\begin{discuss}
{\colorr $\mathcal{G}_t$はusual conditionを満たさないが，Jacod~\cite{jac97}ではfiltrationのright-continuityしか仮定しておらず，$W_t$のfiltrationなのでそれは言える．}
\end{discuss}

}
\end{proof}
}

\begin{discuss}
{\colorr $\Delta_2^2{\colorc \mathcal{Y}^{n,\theta}}=\int^1_0(\mathcal{W}_{t+1}-\mathcal{W}_t)dt$,
$E[{\colorc \mathcal{X}^{n,\theta}_1}{\colorc \mathcal{Y}^{n,\theta}_1}]\approx \int^1_0tdt=1/2$, 
\EQQ{E[({\colorc \mathcal{Y}^{n,\theta}_1})^2]=\int^1_0\int^1_0t\wedge sdtds=\int^1_0((1-s)s+s^2/2)ds=1/2-1/6=1/3.}
$E[{\colorc \mathcal{X}^{n,\theta}_1}\Delta_2^2{\colorc \mathcal{Y}^{n,\theta}}]=\int^1_0(1-t)dt=1/2$,
\EQQ{E[{\colorc \mathcal{Y}^{n,\theta}_1}\Delta_2^2{\colorc \mathcal{Y}^{n,\theta}}]=E[\int^1_0W_tdt\int^1_0(W_{s+1}-W_s)ds]=\int^1_0\int^1_0(t-s)\vee 0dtds
=\int^1_0\int^{1-s}_0tdtds=\int^1_0\frac{(1-s)^2}{2}ds=1/6.}
\EQNN{E[(\Delta_2^2{\colorc \mathcal{Y}^{n,\theta}})^2]&=\int^1_0\int^1_0E[(W_{t+1}-W_t)(W_{s+1}-W_s)]dsdt
=2\int^1_0\int^t_0(s+1-t)dsdt=2\int^1_0(t(1-t)+t^2/2)dt \\
&=2(1/2-1/6)=2/3.}
}
\end{discuss}

\subsection{{\colore Approximation of the log-likelihood ratio}}\label{ori-L-subsection}

We show that the log-likelihood ratio $\log({\colora dP_{\theta_0+h/\sqrt{n},n}^{{\rm aug}}/dP_{\theta_0}^{{\rm aug}}})$
can be approximated by a random variable {\colore that} is observable in the original model.
{\colora Let $X'_j=([X_j^{n,\theta_0}]_k)_{k=1}^{q{\colorc e_n}}$ (removing the last {\colorc element} of $X_j^{n,\theta_0}$),
  $\dot{Y}_0=\tilde{z}_{{\rm ini}}$, $\dot{Y}_j={\colorc \mathcal{Q}}\tilde{Y}_{t_{j,0}}+n({\colorc \tilde{Q}_3B^+}(\check{Y}_{t_{j,0}}-\check{Y}_{t_{j,-1}}))$ for $j\geq 1$.
Let {\colorc $\mathfrak{U}_j=(\mathfrak{U}_{i,j})_{i=1}^d$}, where 
  \EQNN{{\colorc \mathfrak{U}_{i,j}}&=-{\colorc \frac{1}{2}}\Big\{X'^\top_j\partial_{\theta_i}(\psi_{{\colorc e_n}}^{1,1}(\tilde{a}\tilde{a}^\top)^{-1})(\dot{Y}_j,{\colorc \theta_0})X'_j \\
   &\qquad\qquad +{\rm tr}(\partial_{\theta_i}(\psi_{{\colorc e_n}}^{1,1}(\tilde{a}\tilde{a}^\top)^{-1})
   \psi_{{\colorc e_n}}^{1,1}(\tilde{a}\tilde{a}^\top))(\dot{Y}_j,{\colorc \theta_0})\Big\},
  }
{\colorc 
  \EQQ{\mathfrak{V}_j=\bigg(\frac{1}{2}{\rm tr}(\psi_{e_n}^{1,1}(\tilde{a}\tilde{a}^\top)^{-1}\psi_{e_n}^{1,1}(\partial_{\theta_i}(\tilde{a}\tilde{a}^\top))
   \psi_{e_n}^{1,1}(\tilde{a}\tilde{a}^\top)^{-1}\psi_{e_n}^{1,1}(\partial_{\theta_{i'}}(\tilde{a}\tilde{a}^\top)))(\dot{Y}_j,\theta_0)\bigg)_{1\leq i,i'\leq d}
  }
for $0\leq j\leq L_n-1$.}
Then $\mathfrak{U}_j$ and $\mathfrak{V}_j$ are functionals of the original observations.
}

\begin{proposition}\label{original-LAMN-prop}
Assume (C2$''$), (C4), and (C5$'$). Then
\begin{equation*}
{\colora \log\frac{dP_{\theta_0+h/\sqrt{n}}^{{\rm aug}}}{dP_{\theta_0}^{{\rm aug}}}
-{\colorc \frac{h}{\sqrt{n}}\cdot \sum_{j=0}^{L_n-1}\mathfrak{U}_j+\frac{h^\top}{2n}\sum_{j=0}^{L_n-1}\mathfrak{V}_j h}\overset{P}\to 0}
\end{equation*}
{\colorc as $n\to\infty$} for any ${\colora h}\in\mathbb{R}^d$.
\end{proposition}

\begin{proof}
{\colore Because} the augmented model satisfies Condition~(L) {\colorc and 
\EQQ{G_{L_n}^n/\sqrt{n}\overset{P}\to 0}
by Lemma~\ref{B4-suff-lemma} and (\ref{gamma-Ln-est})}, it is sufficient to show that 
\begin{equation*}
{\colora {\colorc \frac{1}{\sqrt{n}}}\sum_{j=0}^{L_n-1}{\colorc (\mathfrak{U}_j-G_j^n)}\overset{P}\to 0
\quad {\colorc {\rm and} \quad \frac{h^\top}{n}\sum_{j=0}^{L_n-1}(\mathfrak{V}_j-\gamma_j(\tilde{Y}_{t_{j,0}}))h\overset{P}\to 0,}}
\end{equation*}
{\colora where $G_j^n=(G_{j,i}^n)_{i=1}^d$.

Let 
  ${\colore \mathfrak{J}_i}(x)=\partial_{\theta_i}(\psi_{{\colorc e_n}}^{1,1}(\tilde{a}\tilde{a}^\top)^{-1})(x,{\colorc \theta_0})$
and
  $\tilde{X}'_j=([\tilde{X}_j^{n,\theta_0}]_k)_{k=1}^{q{\colorc e_n}}$. Let
  {\colorc \EQN{\tilde{\mathfrak{U}}_{i,j}&=-\frac{1}{2}\left\{\tilde{X}'^\top_j{\colore \mathfrak{J}_i}(\tilde{Y}_{t_{j,0}})\tilde{X}'_j
   +{\rm tr}({\colore \mathfrak{J}_i}\psi_{{\colorc e_n}}^{1,1}(\tilde{a}\tilde{a}^\top))(\tilde{Y}_{t_{j,0}},{\colorc \theta_0})\right\}, \\
   \tilde{\mathfrak{V}}_j
   &=\bigg(\frac{1}{2}{\rm tr}({\colore \mathfrak{J}_i}\psi_{e_n}^{1,1}(\tilde{a}\tilde{a}^\top)
   {\colore \mathfrak{J}_{i'}}\psi_{e_n}^{1,1}(\tilde{a}\tilde{a}^\top))(\tilde{Y}_{t_{j,0}},\theta_0)\bigg)_{1\leq i,i'\leq d}.
  }}

{\colore First, we} show that $\sum_j({\colorc \mathfrak{U}_{i,j}-\tilde{\mathfrak{U}}_{i,j}}){\colorc /\sqrt{n}}\overset{P}\to 0$.

Let $\tilde{b}_j=\tilde{b}(\tilde{Y}_{t_{j,0}},\check{Y}_{t_{j,0}},\theta_0)$.
{\colore Because}
  \EQNN{&\Delta_{n,j,k}\tilde{Y}-\tilde{a}_j\Delta_{n,j,k}W \\
   &\quad =\int_{t_{j,k-1}}^{t_{j,k}}(\tilde{a}(\tilde{Y}_t,\theta_0)-\tilde{a}_j)dW_t+\tilde{b}_j\Delta_{n,j,k}t+O_p\bigg(\frac{{\colorc e_n}}{n^{3/2}}\bigg) \\
   &\quad =\int_{t_{j,k-1}}^{t_{j,k}}\sum_{l=1}^\kappa[\tilde{a}_j(W_t-W_{t_{j,0}})]_l\partial_{x_l}\tilde{a}_jdW_t+\frac{\tilde{b}_j}{n}+O_p\bigg(\frac{{\colorc e_n^2}}{n^{3/2}}\bigg)
  }
\begin{discuss}
{\colorr $\tilde{b}(\tilde{Y}_t,\check{Y}_t,\theta_0)-\tilde{b}_j$は$\partial_x\tilde{b}$部分は$\tilde{Y}_t$の係数が有界なのでよい．
${\colore \nabla_2}\check{b}$部分は$B\int^t_{t_{j,0}}\tilde{Y}_sds$なのでOK.}
\end{discuss}
{\colorc for $k\geq 1$}, and
  \EQN{&\Delta_{n,j,k}^2\check{Y}-B\tilde{a}_j\int_{t_{j,k-1}}^{t_{j,k}}(W_t-W_{{\colorc t-1/n}})dt \\
   &\quad =B\int_{t_{j,k-1}}^{t_{j,k}}\int_{t-{\colorc 1/n}}^t(\tilde{a}(\tilde{Y}_s,\theta_0)-\tilde{a}_j)dW_sdt
   +\frac{B\tilde{b}_j}{{\colorc n^2}}+O_p\bigg(\frac{{\colorc e_n}}{n^{5/2}}\bigg) \\
   &\quad =B\int_{t_{j,k-1}}^{t_{j,k}}\int^t_{t-{\colorc 1/n}}\sum_{l=1}^\kappa [\tilde{a}_j(W_s-W_{t_{j,0}})]_l\partial_{x_l}\tilde{a}_jdW_sdt
   +\frac{B\tilde{b}_j}{{\colorc n^2}} +O_p\bigg(\frac{{\colorc e_n^2}}{n^{5/2}}\bigg)
  }
for {\colorc $k\geq 2$},
together with a similar estimate for 
  \EQQ{\Delta_{n,j,1}\check{Y}-B\check{Y}_{t_{j,0}}/n-B\tilde{a}_j\int_{t_{j,0}}^{t_{j,1}}(W_t-W_{t_{j,0}})dt,}
we have
\EQ{\label{X-tildeX-eq} X_j'-\tilde{X}'_j={\colorc X''_j}+O_p(n^{-1}{\colorc e_n^2}),}
where ${\colorc X''_j}=O_p(n^{-1/2}{\colorc e_n})$ {\colorc and $E[\tilde{X}'_j(X''_j)^\top|\mathcal{G}_{t_{j,0}}]=0$.}

Moreover, {\colore because} $\sup_x\lVert (\psi_{{\colorc e_n}}^{1,1})^{-1}(\tilde{a}\tilde{a}^\top(x,\theta_0))\rVert_{{\rm op}}\leq C\alpha_n{\colorc e_n}$ {\colorc by Lemma~\ref{tildeK-est-lemma},} 
and $\dot{Y}_j-\tilde{Y}_{t_{j,0}}$ is equal to an $n^{-1/2}$-order martingale difference plus an $n^{-1}$-order term, we have
  \EQN{\label{tr-est} &\frac{1}{\sqrt{n}}\sum_{j=1}^{L_n-1}\left\{{\rm tr}({\colore \mathfrak{J}_i}\psi_{{\colorc e_n}}^{1,1}(\tilde{a}\tilde{a}^\top))(\dot{Y}_j,{\colorc \theta_0})
   -{\rm tr}({\colore \mathfrak{J}_i}\psi_{{\colorc e_n}}^{1,1}(\tilde{a}\tilde{a}^\top))(\tilde{Y}_{t_{j,0}},{\colorc \theta_0})\right\} \\
   &=\frac{1}{\sqrt{n}}\sum_{j=1}^{L_n-1}\int_0^1\partial_x{\rm tr}({\colore \mathfrak{J}_i}\psi_{{\colorc e_n}}^{1,1}(\tilde{a}\tilde{a}^\top))
   (u\dot{Y}_j+(1-u)\tilde{Y}_{t_{j,0}},{\colorc \theta_0})du\cdot (\dot{Y}_j-\tilde{Y}_{t_{j,0}}) \\
   &=\frac{1}{\sqrt{n}}\sum_{j=1}^{L_n-1}\partial_x{\rm tr}({\colore \mathfrak{J}_i}\psi_{{\colorc e_n}}^{1,1}(\tilde{a}\tilde{a}^\top))(\tilde{Y}_{t_{j-1,0}})
   (\dot{Y}_j-\tilde{Y}_{t_{j,0}})+O_p(n^{-1/2}\cdot n\cdot n^{-1}\alpha_n^4{\colorc e_n^7}) \\
   &\overset{P}\to 0 
  }
as $n\to \infty$ {\colorc by rearranging $e_n$ if necessary.
Similarly, we have
  \EQ{\label{mcv-diff} \frac{h^\top }{n}\sum_{j=0}^{L_n-1}(\mathfrak{V}_j-\tilde{\mathfrak{V}}_j)h\overset{P}\to 0.}
}

(\ref{X-tildeX-eq}), (\ref{tr-est}), and a similar estimate yield
  \EQN{\label{mcu-diff} &{\colorc \frac{1}{\sqrt{n}}\sum_j(\mathfrak{U}_{i,j}-\tilde{\mathfrak{U}}_{i,j})} \\
   &\quad =-\frac{1}{2\sqrt{n}}\sum_j\big\{2\tilde{X}'^\top_j{\colore \mathfrak{J}_i}(\dot{Y}_j)(X'_j-\tilde{X}'_j)+(X'_j-\tilde{X}'_j)^\top{\colore \mathfrak{J}_i}(\dot{Y}_j)(X'_j-\tilde{X}'_j) \\
   &\quad \quad \quad \quad \quad \quad \quad +\tilde{X}'^\top_j({\colore \mathfrak{J}_i}(\dot{Y}_j)-{\colore \mathfrak{J}_i}(\tilde{Y}_{t_{j,0}}))\tilde{X}'_j\big\}+o_p(1) \\
   &\quad =-\frac{1}{\sqrt{n}}\sum_j\tilde{X}'^\top_j{\colore \mathfrak{J}_i}(\dot{Y}_j){\colorc X''_j} \\
   &\qquad -\frac{1}{2\sqrt{n}}\sum_j{\rm tr}(({\colore \mathfrak{J}_i}(\dot{Y}_j)-{\colore \mathfrak{J}_i}(\tilde{Y}_{t_{j,0}}))\psi_{{\colorc e_n}}^{1,1}(\tilde{a}\tilde{a}^\top(\tilde{Y}_{t_{j,0}},{\colorc \theta_0}))) \\
   &\quad \overset{P}\to 0.
  }
\begin{discuss}
{\colorr $\lVert (\psi_{{\colorc e_n}}^{1,1})^{-1}\rVert_{{\rm op}}\leq \alpha_n$ ($\tilde{Y}_{t_{j,0}}$ではなく
$\dot{Y}_j$だが評価はProposition~\ref{one-interval-nondeg-lemma}の評価を満たしているので同様．)}
\end{discuss}
{\colorc Moreover, (C5$'$) yields
  \EQ{\label{tildeMcv-conv} \frac{h^\top}{n}\sum_{j=0}^{L_n-1}(\tilde{\mathfrak{V}}_j-\gamma_j(\tilde{Y}_{t_{j,0}}))h\overset{P}\to 0.}
} 

{\colorc Thanks to (\ref{mcv-diff})--(\ref{tildeMcv-conv}),} it is sufficient to show 
  {\colorc \EQ{\frac{1}{\sqrt{n}}\sum_{j=0}^{L_n-1}\{\tilde{\mathfrak{U}}_j-G_j^n\}\overset{P}\to 0.}}
We can easily check
  $E[{\colorc \tilde{\mathfrak{U}}_{i,j}-G_{j,i}^n}|\mathcal{G}_{t_{j,0}}]=0$
and (C5$'$) yields
  {\colorc \EQN{&\frac{1}{n}\sum_jE[|\tilde{\mathfrak{U}}_j-G_j^n|^2|\mathcal{G}_{t_{j,0}}] \\
   &\quad =\frac{1}{n}\sum_jE[|\tilde{\mathfrak{U}}_j|^2-2\tilde{\mathfrak{U}}_j\cdot G_j^n+|G_j^n|^2|\mathcal{G}_{t_{j,0}}] \\
   &\quad =\frac{1}{n}\sum_j{\rm tr}(\mathcal{T}_{1,1,{\colorc e_n}}(\tilde{Y}_{t_{j,0}})-2\mathcal{T}_{1,2,{\colorc e_n}}(\tilde{Y}_{t_{j,0}})
   +\mathcal{T}_{2,2,{\colorc e_n}}(\tilde{Y}_{t_{j,0}}))\overset{P}\to 0.
  }}
} 

Then Lemma~9 in Genon-Catalot and Jacod \cite{gen-jac93} yields the conclusion.

\end{proof}
\begin{discuss}
{\colorr Gloter and GobetのProposition 3の最後の方でやっているブロック行列の端の部分の評価が${\colorc e_n}$がどんなにゆっくり無限大に発散しても成り立つことが，${\colorc e_n}$を遅く発散できることの本質．}
\end{discuss}

\subsection{{\colore Proof of Theorem~\ref{partial-LAMN-thm}}}

{\colore In light of Proposition~\ref{original-LAMN-prop},}
a similar argument to Proposition~4 in~\cite{glo-gob08} yields
\begin{equation*}
{\colora \log\frac{dP_{\theta_0+h/\sqrt{n},n}}{dP_{\theta_0,n}}
-h\cdot \sum_{j=0}^{L_n-1}{\colorc \mathfrak{U}_j}+\frac{1}{2}h^\top\sum_{j=0}^{L_n-1}{\colorc \mathfrak{V}_j}h\overset{P}\to 0.}
\end{equation*}
Therefore, we obtain Condition~(L) of the original model under {\colora (C2$''$), (C4), and (C5$'$)}.\\

{\colora Let $\phi_q^1$ and $\phi_q^2$ be the same as the ones in the proof of Theorem~\ref{degenerate-LAMN-thm}.
%We can further assume that $\phi_q^2(Q_2y)=Q_2\phi_q^2(y)$.
Let $P_{\theta,q,n}$ be the probability measure generated by replacing the coefficients $\tilde{a}(x,\theta)$ and $\tilde{b}(z,\theta)$
by $\tilde{a}_q(x,\theta)=\tilde{a}(\phi_q^1(x),\theta)$ and $\tilde{b}_q(z,\theta)=\tilde{b}(\phi_q^1(x),\phi_q^2(y),\theta)$ ($z=(x,y)$), respectively.
Then (C2$''$), (C4), and (C5$'$) are satisfied for $\{P_{\theta,q,n}\}_{\theta,n}$.
Therefore, similarly to the proof of Theorem~\ref{degenerate-LAMN-thm}, we have the conclusion.
}

{\colorc 
\section{Proofs of the results in Section~\ref{Mcal-section}}\label{Mcal-proof-section}

\subsection{Proof of Lemma~\ref{B4-suff-lemma}}

(\ref{gamma-j-eq}) and (\ref{Gj-def-eq}) {\colore yield} Point~1. Moreover, we have
  \EQNN{&\epsilon_n^4\sum_{j=1}^{m_n}E[|G_j^n|^4|\mathcal{G}_{j-1}] \\
   &\quad \leq C\epsilon_n^4\sum_{j=1}^{m_n}E_j\bigg[\bigg|\tilde{F}_{n,\theta,j}^\top 
   \frac{B_{j,i,\theta_0}^\top\tilde{K}_j^{-1}(\theta_0)+\tilde{K}_j^{-1}(\theta_0)B_{j,i,\theta_0}}{2}\tilde{F}_{n,\theta,j}\bigg|^4\bigg]\bigg|_{\tilde{x}_{j-1}=\tilde{X}_{j-1}^{n,\theta_0}} \\
   &\quad \leq C\epsilon_n^4m_n\alpha_n^8\bar{k}_n^8\sup_{j,i,\theta,\tilde{x}_{j-1}}
   \lVert \tilde{K}_j(\theta_0)B_{j,i,\theta_0}^\top+B_{j,i,\theta_0}\tilde{K}_j(\theta_0)\rVert_{{\rm op}}^4\to 0 
  }
as $n\to\infty$.

\qed

\subsection{Proof of Proposition~\ref{delLogP-eq}}

Theorem~2.1.2 in~\cite{nua06} shows that $F_{n,\theta,j}$ admits a density $p_{j,\bar{x}_{j-1}}(x_j,\theta)$.

For any $g\in C^1_b(\mathbb{R}^{k_j-k_{j-1}})$ and $h\in C^\infty_o(\Theta)$, we have
  \EQNN{&-\int \partial_\theta h(\theta)E_j[g(F_{n,\theta,j})]d\theta \\
   &\quad =\int h(\theta)\partial_\theta E_j[g(F_{n,\theta,j})]d\theta \\
   &\quad =\int h(\theta)E_j\bigg[\sum_k\frac{\partial g}{\partial y_k}\partial_\theta {\colore [F_{n,\theta,j}]_k}\bigg]d\theta \\
   &\quad =\int h(\theta)E_j\bigg[\sum_{k,k'}\langle D_jg(F_{n,\theta,j}),D_j{\colore [F_{n,\theta,j}]_{k'}}\rangle_H[K_j^{-1}(\theta)]_{kk'}\partial_\theta {\colore [F_{n,\theta,j}]_k}\bigg]d\theta \\
   &\quad =\int h(\theta)\int g(x_j)E_j[\delta_j(L^\theta(\partial_\theta F_{n,\theta,j}))
    |F_{n,\theta,j}=x_j]p_{j,\bar{x}_{j-1}}(x_j,\theta)dx_jd\theta. 
  }
\begin{discuss}
{\colorr $L^\theta(\partial_\theta F_{n,\theta,j})\in \mathbb{D}^{1,2}$が必要．最初の等式は部分積分， 二つ目の等式は$\partial_\theta F(\theta)=\partial_\theta F(\theta_0)+\partial_\theta^2 F(\theta_0)(\theta-\theta_0)$
をつかって$\sup_\theta \lVert \partial_\theta F\rVert$を局所的に評価できることからLebesgueの微分定理を使う．
フレシェ微分ができるなら合成関数の微分もできるだろう．
$L^\theta(\partial_\theta F)\in \mathbb{D}^{1,p}$が必要だから$\partial_\theta F\in \mathbb{D}^{1,p}$が必要．}
\end{discuss}
Therefore, we obtain
  \EQN{\label{d-theta-eq}
   &-\int \partial_\theta h(\theta)p_{j,\bar{x}_{j-1}}(x_j,\theta)d\theta \\
   &\quad =\int h(\theta)E_j[\delta_j(L^\theta(\partial_\theta {\colore F_{n,\theta,j}}))|F_{n,\theta,j}=x_j]p_{j,\bar{x}_{j-1}}(x_j,\theta)d\theta
  }
almost everywhere in $x_j\in \mathbb{R}^{k_j-k_{j-1}}$ for any $h\in C^\infty_o(\Theta)$. 
{\colore Because} $C^\infty_o(\Theta)$ is separable with respect to {\colore the} Sobolev norm $\lVert \cdot \rVert_W^{1,2}$, 
we have (\ref{d-theta-eq}) for any $h\in C^\infty_o(\Theta)$ almost everywhere in $x_j$.
Similarly, we can obtain an equation for $\int \partial_\theta^l h(\theta)p_{j,\bar{x}_{j-1}}(x_j,\theta)d\theta$ for $l=2,3$, then Theorem~5.3 in Shigekawa~\cite{shi04} yields
$p_{j,\bar{x}_{j-1}}(x_j,\cdot ) \in C^2(\Theta)$ almost everywhere in $x_j\in \mathbb{R}^{k_j-k_{j-1}}$.
Then a similar argument with $h=\delta_\theta$ (Dirac delta) yields (\ref{log-density-p-eq1}).

Similarly, we obtain
\begin{eqnarray}
&&\int g(x_j)\partial_\theta^2 p_{j,\bar{x}_{j-1}}(x_j,\theta)dx_j \nonumber \\
&&\quad =E_j\bigg[\sum_k\frac{\partial g}{\partial x_k}\partial_\theta^2{\colore [F_{n,\theta,j}]_k}
+\sum_{k,l}\frac{\partial^2g}{\partial x_k\partial x_l}\partial_\theta {\colore [F_{n,\theta,j}]_k}\partial_\theta {\colore [F_{n,\theta,j}]_l}\bigg] \nonumber \\
&&\quad =E_j\big[g(F_{n,\theta,j})\delta_j(L^\theta(\partial_\theta^2 F_{n,\theta,j})) \big]
+E_j\bigg[\sum_k\frac{\partial g}{\partial x_k}(F_{n,\theta,j})
{\colorc [\mathfrak{A}_j]_k}\bigg] \nonumber \\
&&\quad =E_j\big[g(F_{n,\theta,j})\delta_j(L^\theta(\partial_\theta^2 F_{n,\theta,j})) \big]
+E_j[g(F_{n,\theta,j})\delta_j(L^\theta({\colorc \mathfrak{A}_j}))] \nonumber
\end{eqnarray}
for any $g\in C^2_b(\mathbb{R}^{k_j-k_{j-1}})$, which implies (\ref{log-density-p-eq2}).

\qed

\subsection{Proof of Proposition~\ref{delp-p-est-prop}}

The first inequality is obtained {\colore because}
  \EQNN{&E[|\delta_j(L^\theta(\partial_{\theta_i}F_{n,\theta,j}))|^4]^{1/4} \\
   &\quad \leq {\colord C}{\colore \bigg\lVert \sum_{k,l}\partial_{\theta_i}[F_{n,\theta,j}]_k
   [K_j^{-1}(\theta)]_{k,l}D_j[F_{n,\theta,j}]_l\bigg\rVert_{1,4}} \\
   &\quad \leq C\sum_{k,l}\lVert \partial_{\theta_i}{\colore [F_{n,\theta,j}]_k}\rVert_{1,16}
   \lVert[K_j^{-1}(\theta)]_{k,l}\rVert_{1,8}\lVert {\colore [F_{n,\theta,j}]_l}\rVert_{2,16}\leq C\alpha_n\bar{k}_n^2.
  }
The estimate for $\delta_j(L^\theta(\partial_{\theta_i}\partial_{\theta_l}F_{n,\theta,j}))$ is similarly obtained. Moreover, we have
\begin{eqnarray}
&&\lVert \delta_j(L^\theta((\delta_j(L^\theta(\partial_{\theta_i}F_{n,\theta,j}\partial_{\theta_l}{\colore [F_{n,\theta,j}]_k})))_k))\rVert_{0,2} \nonumber \\
&&\quad \leq {\colord C}{\colore \bigg\lVert \sum_{k,l}\delta_j(L^\theta(\partial_{\theta_i}F_{n,\theta,j}\partial_{\theta_l}{\colore [F_{n,\theta,j}]_k})) [K_j^{-1}(\theta)]_{k,l}D_j[F_{n,\theta,j}]_l\bigg\rVert_{1,2}} \nonumber \\
&&\quad \leq {\colord C\alpha_n\bar{k}_n^2}\sum_k\lVert \delta_j(L^\theta(\partial_{\theta_i}F_{n,\theta,j}\partial_{\theta_l}{\colore [F_{n,\theta,j}]_k}))\rVert_{1,4} \nonumber \\
&&\quad \leq {\colore C\alpha_n\bar{k}_n^2\sum_k\bigg\lVert
 \partial_{\theta_l}{\colore [F_{n,\theta,j}]_k}\sum_{l,l'}\partial_{\theta_i}[F_{n,\theta,j}]_l
 [K_j^{-1}(\theta)]_{l,l'}D_j[F_{n,\theta,j}]_{l'}\bigg\rVert_{2,4}} \nonumber \\
&&\quad \leq {\colord C\alpha_n^2\bar{k}_n^4}. \nonumber 
\end{eqnarray}
\qed

\subsection{Proof of Lemma~\ref{tildeK-est-lemma}}

We denote $F_i=[F_{n,\theta,j}]_i$ and $\tilde{F}_i=[\tilde{F}_{n,\theta,j}]_i$. {\colore Because}
\begin{equation*}
\lVert [K_j-\tilde{K}_j]_{il}\rVert_{2,p}\leq \lVert \langle D_j(F_i-\tilde{F}_i),D_jF_l\rangle \rVert_{2,p}+\lVert \langle D_j\tilde{F}_i,D_j(F_l-\tilde{F}_l)\rangle \rVert_{2,p}
\end{equation*}
and
  \EQNN{\lVert \langle D_j(F_i-\tilde{F}_i),D_jF_l\rangle \rVert_{2,p}&\leq \lVert \lVert D_j(F_i-\tilde{F}_i)\rVert_H \lVert D_jF_l\rVert_H \rVert_{2,p} \\
   &\leq \lVert F_i-\tilde{F}_i\rVert_{3,2p}\lVert F_l\rVert_{3,2p}\leq C_p\rho_n, 
  }
we have 
\begin{equation*}
\sup_{i,l,j,\bar{x}_{j-1},\theta}\lVert [K_j(\theta)-\tilde{K}_j(\theta)]_{il}\rVert_{2,p} \leq C_p\rho_n.
\end{equation*}
\begin{discuss}
{\colorr $D(H)^{2,p}$のノルムは$E[\lVert u\rVert_H^p]$で定義される．ちなみに$H\otimes H$の正規直交基底は$(e_m\otimes e_n)_{m,n}$でノルムもそこから定義される}
\end{discuss}

Moreover, we have
 $\tilde{K}_j=K_j+(\tilde{K}_j-K_j)=K_j(I+K_j^{-1}(\tilde{K}_j-K_j))$,
 $E[\lVert K_j^{-1}\rVert_{{\rm op}}^p]^{1/p}\leq \alpha_n{\colorb \bar{k}_n}$,
and 
\begin{equation*}
E[\lVert K_j^{-1}(\tilde{K}_j-K_j)\rVert_{{\rm op}}]\leq {\colord C\alpha_n\bar{k}_n\cdot \rho_n\bar{k}_n}\to 0
\end{equation*}
as $n\to \infty$, by (B2) and {\colore the fact} that $\alpha_n\rho_n\bar{k}_n^2\to 0$.
Then, for sufficiently large $n$, we have $\lVert K_j^{-1}(\tilde{K}_j-K_j)\rVert_{{\rm op}}<1/2$ with positive probability,
and therefore $\tilde{K}_j^{-1}$ exists and 
\begin{equation*}
\lVert \tilde{K}_j^{-1}\rVert_{{\rm op}}\leq \lVert K_j^{-1}\rVert_{{\rm
 op}}(1-\lVert K_j^{-1}(\tilde{K}_j-K_j)\rVert_{{\rm op}})^{-1}\leq
 C\alpha_n{\colorb \bar{k}_n}
\end{equation*}
with positive probability.
{\colore Because} $\tilde{K}_j$ is deterministic, we obtain $|[\tilde{K}_j^{-1}]_{il}|\leq C\alpha_n{\colorb\bar{k}_n}$.
\qed
}

\subsection{{\colorc Proof of Lemma~\ref{root-p-smooth-lemma}}}
{\colord

\begin{lemma}\label{abs-conti-lemma}
Let $f:[0,1]\to \mathbb{R}$ be a continuous function and $E\subset \mathbb{R}$ be a finite set. 
Assume that the derivative $\dot{f}(t)$ exists almost everywhere in $t\in f^{-1}(E^c)$ and $\int_{f^{-1}(E^c)}|\dot{f}(t)|dt<\infty$.
Then, $f$ is absolutely continuous on $[0,1]$.
\end{lemma}
\begin{proof}
We may assume that $E$ is not empty. We denote $E=\{a_1,\cdots, a_k\}$ for some $k\in \mathbb{N}$ and $a_1<\cdots <a_k$.
It is sufficient to show that 
\begin{equation}\label{f-integral-eq}
|f(t)-f(s)|\leq \int_{f^{-1}(E^c)\cap [s,t]}|\dot{f}(u)|du
\end{equation}
for any $0\leq s<t\leq 1$.

Fix $s,t\in [0,1]$ satisfying $s<t$. First, we assume that $f(t),f(s)\not\in E$. We only show the case {\colore where} there exist $k_1, k_2$ such that $1\leq k_1\leq k_2\leq k$, $f(s)<a_{k_1}$, and $a_{k_2}<f(t)$.
Other cases are proved in a similar way.
\begin{discuss}
{\colorr 後は逆向きの可能性と同じブロックに入る場合}
\end{discuss}

Let $K=k_2-k_1+1$ and $E_\epsilon:=\{y\in\mathbb{R}|\min_{x\in E}|x-y|\leq \epsilon\}$. We set a positive number $\epsilon$ 
so that $\epsilon<\min_{2\leq l\leq k}|a_l-a_{l-1}|/2$ and $f(t),f(s)\not\in E_\epsilon$.
We inductively define
  \EQNN{s_0&=\sup\{u\in[s,t]|f(u)=f(s)\}, \\
   s_j&=\sup\{u\in(t_{j-1},t]|f(u)=a_{k_1+j-1}+\epsilon\} \quad (j=1,\cdots, K), \\
   t_j&=\inf\{u\in (s_j,t]|f(u)=a_{k_1+j}-\epsilon\} \quad (j=0,\cdots, K-1), \\
   t_K&=\inf\{u\in (s_K,t]|f(u)=f(t)\}.
  }
Then, we obtain $f^{-1}(E)\cap(\cup_{j=0}^K(s_j,t_j))=\emptyset$ and 
\begin{eqnarray}
|f(t)-f(s)|&\leq&\sum_{j=0}^K|f(t_j)-f(s_j)|+\sum_{j=1}^K|f(s_j)-f(t_{j-1})| \nonumber \\
&\leq&\sum_{j=0}^K\int^{t_j}_{s_j}|\dot{f}(u)|du+2K\epsilon\leq \int_{f^{-1}(E^c)\cap [s,t]}|\dot{f}(u)|du+2K\epsilon. \nonumber
\end{eqnarray}
By letting $\epsilon\to 0$, we obtain (\ref{f-integral-eq}).

In the case {\colore where} $f(s)=a_{k_1}$ for some $k_1\in \{1,\cdots, k\}$, 
then by setting $s_1,t_1,\cdots, s_K,t_K$ similarly, we have
\begin{equation*}
|f(t)-f(s)|\leq \sum_{j=1}^K|f(t_j)-f(s_j)|+\epsilon
\end{equation*}
for sufficiently small $\epsilon$, and consequently we have (\ref{f-integral-eq}).
We can similarly show (\ref{f-integral-eq}) in the case that $f(t)=a_{k_2}$ for some $k_2\in \{1,\cdots, k\}$.
\begin{discuss}
{\colorr $f$がcontiでないと$s_j,t_j$がwell-definedでないのと,$(s_j,t_j)\cap f^{-1}(E)=\emptyset$とならない．}
\end{discuss}
\end{proof}

{\colorc
\noindent 
{\bf Proof of Lemma~\ref{root-p-smooth-lemma}.}
Let $p,q>1$ with $1/p+1/q=1$ and $1-q/2>0$. We abbreviate $p_{j,t}(x_j)=p_{j,\bar{x}_{j-1}}(x_j,\theta_{th})$.
Then, H${\rm \ddot{o}}$lder's inequality yields
  \EQNN{&\int \int^1_0 \sqrt{p_{j,t}} (x_j)|\mathcal{E}_j^1(x_j,\theta_{th})|dtdx_j \\
   &\quad =\int^1_0E_j\bigg[\frac{1}{\sqrt{p_{j,t}}}|E_j[\delta_j(L^{\theta_{th}}(\partial_{\theta_i}F_{n,\theta_{th},j}))|F_{n,\theta_{th},j}]|1_{\{p_{j,t}\neq 0\}}\bigg]dt \\
   &\quad \leq  \sup_{t\in [0,1]}\bigg\{E_j[|\delta_j(L^{\theta_{th}}(\partial_{\theta_i}F_{n,\theta_{th},j}))|^p]^{1/p}
   \bigg(\int p_{j,t}^{1-q/2}(x_j)dx_j\bigg)^{1/q}\bigg\}.
  }
Proposition~2.1.5 in Nualart~\cite{nua06} and its proof yield $\sup_t\int p_{j,t}^{1-q/2}(x_j)dx_j<\infty$ under (N2).
Together with (B1) and (B2), we obtain 
  \EQQ{\int \int^1_0 \sqrt{p_{j,t}} (x_j)|\mathcal{E}_j^1(x_j,\theta_{th})|dtdx_j<\infty,} 
which implies 
\begin{equation}\label{root-p-derivative-est}
\int^1_0 \frac{|\partial_t p_{j,t}|}{\sqrt{p_{j,t}}}(x_j)1_{\{p_{j,t}\neq 0\}}dt
=\int^1_0 \sqrt{p_{j,t}} (x_j)|\mathcal{E}_j^1(x_j,\theta_{th})|dt<\infty
\end{equation}
almost everywhere in $x_j$.

The function $\sqrt{p_{j,t}}$ has derivative $\partial_t\sqrt{p_{j,t}}=\partial_t p_{j,t}/(2\sqrt{p_{j,t}})$ if $\sqrt{p_{j,t}}\neq 0$ by Proposition~\ref{delLogP-eq}.
Therefore, Lemma~\ref{abs-conti-lemma} and (\ref{root-p-derivative-est}) yield the conclusion.

\qed
}
}

{\colorp
\section{Nondegeneracy of the Malliavin matrix for degenerate diffusion processes}

\subsection{{\colorn The Malliavin matrix on a section}}

Let $m_1,m_2,r,{\colorn L}\in\mathbb{N}$.
Let $(\mathfrak{X},\mathfrak{F},\mu)$ be the canonical probability space associated with {\colore an} $r$-dimensional Wiener process $\mathcal{W}=(\mathcal{W}_t)_{t\in[0,{\colorn L}]}$.
Let $\mathcal{X}_t$ and $\mathcal{Y}_t$ be $m_1$- and $m_2$-dimensional diffusion processes, respectively, 
on a Wiener space satisfying $(\mathcal{X}_0,\mathcal{Y}_0)=(x_0,y_0)$ and	
\begin{eqnarray}
d\mathcal{X}_t&=&\tilde{B}(t,\mathcal{X}_t,\mathcal{Y}_t)dt+\tilde{A}(t,\mathcal{X}_t,\mathcal{Y}_t)d\mathcal{W}_t, \nonumber \\
d\mathcal{Y}_t&=&\check{B}(t,\mathcal{X}_t,\mathcal{Y}_t)dt, \nonumber
\end{eqnarray}
where $x_0\in\mathbb{R}^{m_1}$, $y_0\in\mathbb{R}^{m_2}$, {\colore and} $\tilde{B}(t,x,y)$, $\tilde{A}(t,x,y)$, and $\check{B}(t,x,y)$ {\colore are} $\mathbb{R}^{m_1}$-, $\mathbb{R}^{m_1}\otimes \mathbb{R}^r$- and $\mathbb{R}^{m_2}$-valued functions, respectively.
We assume that the derivatives ${\colore \partial_{(x,y)}}\tilde{B}$, $\partial_t^i{\colore \partial_{(x,y)}^j}\check{B}$, and ${\colore \partial_{(x,y)}^l}\tilde{A}$ 
exist and are continuous with respect to $(t,x,y)$ for $i\in \{0,1\}$, $j\in \{1,2,3\}$, and $l\in\{0,1\}$.
We denote $z=(x,y)$ and by $E_\mu$ the expectation with respect to $\mu$. 
Let $\mathfrak{H}=L^2([0,{\colorn L}];\mathbb{R}^r)$ and $D$ be the {\colore Malliavin--Shigekawa} derivative operator associated with $\mathcal{W}$. 
Let $\gamma_{\mathcal{X},\mathcal{Y}}$ be the Malliavin matrix of $(\mathcal{X}_1,\mathcal{Y}_1)$, that is,
\begin{equation*}
\gamma_{\mathcal{X},\mathcal{Y}}=\left(
\begin{array}{cc}
\langle D\mathcal{X}_1,D\mathcal{X}_1\rangle_{\mathfrak{H}} & \langle D\mathcal{X}_1,D\mathcal{Y}_1\rangle_{\mathfrak{H}} \\
\langle D\mathcal{Y}_1,D\mathcal{X}_1\rangle_{\mathfrak{H}} & \langle D\mathcal{Y}_1,D\mathcal{Y}_1\rangle_{\mathfrak{H}}
\end{array}
\right).
\end{equation*}

{\colorc For a multi-index $(i_1,\cdots, i_l)$, we denote $|A_{i_1,\cdots,i_l}|^2=\sum_{i_1,\cdots, i_l}A_{i_1,\cdots,i_l}^2$.}
\begin{proposition}\label{one-interval-nondeg-lemma}
Assume that there exist constants $M_1$ and $M_2$ such that
\begin{equation}\label{nondeg-prop-cond1}
\sup_{t,x,y}\bigg({\colorc |{\colore \partial_{(x,y)}^l}\tilde{B}(t,x,y)| \vee |\partial_t^i{\colore \partial_{(x,y)}^j}\check{B}(t,x,y)| 
\vee |{\colore \partial_{(x,y)}^l}\tilde{A}(t,x,y)|}\bigg)\leq M_1
\end{equation}
for $i\in \{0,1\}$, $j\in \{1,2,3\}$, and $l\in \{0,1\}$, and
\begin{equation}\label{nondeg-prop-cond2}
\sup_{t,x,y}(\lVert ((\partial_x\check{B})^\top\partial_x \check{B})^{-1}(t,x,y)\rVert_{{\rm op}} \vee \lVert (\tilde{A}\tilde{A}^\top)^{-1}(t,x,y)\rVert_{{\rm op}})\leq M_2.
\end{equation}
Then $\gamma_{\mathcal{X},\mathcal{Y}}$ is positive definite almost surely, and for any $p\geq 1$, 
there exists a constant $C_p$ depending only on {\colorc $x_0$, $y_0$}, $p$, $m_1$, $m_2$, $M_1$, and $M_2$
such that $E_\mu[|\det \gamma_{\mathcal{X},\mathcal{Y}}|^{-p}]\leq C_p$.
{\colorc If further 
\EQ{ \label{nondeg-prop-further} \partial_y\check{B}\equiv 0\quad {\rm or} \quad |\check{B}|\leq M_1,} 
{\colore then $C_p$ depends on neither} $x_0$ nor $y_0$.}
\end{proposition} 

To prove Proposition~\ref{one-interval-nondeg-lemma}, {\colore first we show the} nondegeneracy of 
the Malliavin matrix $\gamma_\mathcal{X}=\langle D\mathcal{X}_1,D\mathcal{X}_1\rangle_{\mathfrak{H}}$ for $\mathcal{X}_1$.
Let 
\begin{equation*}
B(t,x,y)=\left(
\begin{array}{c}
\tilde{B}(t,x,y) \\
\check{B}(t,x,y)
\end{array}
\right), \quad 
A(t,x,y)=\left(
\begin{array}{c}
\tilde{A}(t,x,y)  \\
O_{m_2,r}
\end{array}
\right),
\end{equation*}
$B_t=B(t,\mathcal{X}_t,\mathcal{Y}_t)$, and $A_t=A(t,\mathcal{X}_t,\mathcal{Y}_t)$.
We define an $(m_1+m_2)\times (m_1+m_2)$ matrix-valued process $(\mathcal{U}_t)_{t\in [0,{\colorn L}]}$ 
by a stochastic integral equation
\begin{equation*}
[\mathcal{U}_t]_{ij}=\delta_{ij}+\sum_{k=1}^{m_1+m_2}\int^t_0[{\colore \nabla}B_s]_{ki}[\mathcal{U}_s]_{kj}ds + \sum_{k,l=1}^{m_1+m_2}\int^t_0[{\colore [\nabla]_k}A_s]_{il}[\mathcal{U}_s]_{kj}d\mathcal{W}_s^l,
\end{equation*}
{\colore where $\nabla=\partial_{(x,y)}$.}
Then by the argument in Section~2.3.1 of Nualart~\cite{nua06}, 
$\mathcal{U}_t$ is invertible and we have
\begin{eqnarray}
[\mathcal{U}_t^{-1}]_{ij}&=&\delta_{ij}-\sum_{k=1}^{m_1+m_2}\int^t_0[\mathcal{U}_s^{-1}]_{ik}\bigg(
[{\colore \nabla}B_s]_{jk}-\sum_{l,\alpha=1}^{m_1+m_2}[{\colore [\nabla]_\alpha}A_s]_{kl}[{\colore [\nabla]_j}A_s]_{\alpha l}\bigg)ds \nonumber \\
&&- \sum_{k,l=1}^{m_1+m_2}\int^t_0[\mathcal{U}^{-1}_s]_{ik}[{\colore [\nabla]_j}A_s]_{kl}d\mathcal{W}_s^l. \nonumber
\end{eqnarray}
Moreover, we obtain 
\begin{equation}\label{DZ-eq}
{\colorn (D_r\mathcal{Z}_t)^\top=\mathcal{U}_t\mathcal{U}^{-1}_rA_r1_{\{r\leq t\}}, }
\end{equation}
where $\mathcal{Z}_t=(\mathcal{X}_t^\top,\mathcal{Y}_t^\top)^\top$.

\begin{lemma}\label{gamma_x_inv-lemma}
Under the assumptions of Proposition~\ref{one-interval-nondeg-lemma}, $\gamma_{\mathcal{X}}$ is an invertible matrix almost surely and 
for all $p\geq 1$, there exists a positive constant $C'_p$ depending only on $p$, $m_1$, $M_1$, and $M_2$ such that
\begin{equation*}
E_\mu[|\det(\gamma_{\mathcal{X}})|^{-p}]\leq C'_p.
\end{equation*}
\end{lemma}
\begin{proof}
{\colorc Let $\mathcal{I}=(I_{m_1} \ O_{m_1,m_2})$.} 
Let $\tau=1-\sup\{t\in [0,1];(\lVert\mathcal{U}_1\rVert_{{\rm op}} +
 \lVert\mathcal{U}_1^{-1}\rVert_{{\rm
 op}})^3\lVert\mathcal{U}_t-\mathcal{U}_1\rVert_{{\rm
 op}}>(48\sqrt{3}M_1^2M_2)^{-1} \wedge (1/6)\} \vee 0$.
\begin{discuss}
{\colorr 参考：$t=1$の時$\lVert \mathcal{U}_t-\mathcal{U}_1\rVert_{{\rm op}} =0$で，$\tilde{A},\tilde{B},\check{B}$は有界なので$\tau>0$. }
\end{discuss}
Then, {\colore because} $\tilde{A}\tilde{A}^{\top}\geq (1/M_2) I_{m_1}$, we have
  \EQNN{\gamma_{\mathcal{X}}&=\int^1_0(D_t\mathcal{X}_1)^\top D_t\mathcal{X}_1dt \\
   &=\int^1_0{\colorc \mathcal{I}}\mathcal{U}_1\mathcal{U}_t^{-1}A_tA_t^\top(\mathcal{U}_t^{\top})^{-1}\mathcal{U}_1^{\top}{\colorc \mathcal{I}^\top}dt \\
   &\geq \frac{1}{M_2} \int^1_{1-\tau}{\colorc \mathcal{I}}\mathcal{U}_1\mathcal{U}_t^{-1}
   {\colorc \mathcal{I}^\top\mathcal{I}}(\mathcal{U}_t^{\top})^{-1}\mathcal{U}_1^{\top}{\colorc \mathcal{I}^\top}dt. 
  }
For $x\in\mathbb{R}^{m_1}$ and $t\in [1-\tau,1]$, simple calculations show that
  \EQNN{&x^{\top}{\colorc \mathcal{I}}\mathcal{U}_1\mathcal{U}_t^{-1}{\colorc \mathcal{I}^\top\mathcal{I}}
   (\mathcal{U}_t^{\top})^{-1}\mathcal{U}_1^{\top}{\colorc \mathcal{I}^\top}x \\
   &\quad \geq |x|^2 - |x|^2\bigg\lVert{\colorc \mathcal{I}}\mathcal{U}_1(\mathcal{U}_t^{-1}-\mathcal{U}_1^{-1})
   {\colorc \mathcal{I}^\top\mathcal{I}}(\mathcal{U}_t^{\top})^{-1}\mathcal{U}_1^{\top}{\colorc \mathcal{I}^\top}\bigg\rVert_{{\rm op}} \\
   &\qquad - |x|^2\bigg\lVert{\colorc \mathcal{I}}((\mathcal{U}_t^{\top})^{-1}-(\mathcal{U}_1^{\top})^{-1})\mathcal{U}_1^{\top}{\colorc \mathcal{I}^\top}\bigg\rVert_{{\rm op}} \\
   &\quad \geq |x|^2 - 4|x|^2\lVert \mathcal{U}_t-\mathcal{U}_1\rVert_{{\rm op}} \lVert \mathcal{U}_1^{-1}\rVert^2_{{\rm op}}\lVert \mathcal{U}_1\rVert_{{\rm op}}
   -2|x|^2\lVert \mathcal{U}_t-\mathcal{U}_1\rVert_{{\rm op}} \lVert \mathcal{U}_1^{-1}\rVert_{{\rm op}}\\
   &\quad \geq |x|^2/3. 
  }
Here we used Proposition~2.7 in Chapter~II of Conway~\cite{con85}, the equation $\mathcal{U}_1(\mathcal{U}_t^{-1}-\mathcal{U}_1^{-1})=(\mathcal{U}_1-\mathcal{U}_t)\mathcal{U}_t^{-1}$, and {\colore the fact} that
\begin{equation}\label{Ut-norm}
\lVert \mathcal{U}_t^{-1}\rVert_{{\rm op}} \leq \lVert\mathcal{U}_1^{-1}\rVert_{{\rm op}} (1-\lVert \mathcal{U}_1^{-1}\rVert_{{\rm op}} \lVert \mathcal{U}_t-\mathcal{U}_1\rVert_{{\rm op}} )^{-1} \leq 2\lVert\mathcal{U}_1^{-1}\rVert_{{\rm op}}.
\end{equation}
Hence, we have $\gamma_{\mathcal{X}} \geq \frac{\tau}{3M_2}I_{m_1}$.

Moreover, for any $q>0$, there exists a constant $C''_p$ depending only on $p$, $m_1$, $M_1$, and $M_2$ such that 
  \EQNN{&\mu[\tau<1/t] \\
   &\quad \leq  \mu\left[(\lVert\mathcal{U}_1\rVert_{{\rm op}} 
   + \lVert\mathcal{U}_1^{-1}\rVert_{{\rm op}})^3\sup_{0\leq s\leq 1/t}\lVert\mathcal{U}_{1-s}-\mathcal{U}_1\rVert_{{\rm op}}
   \geq \frac{1}{48\sqrt{3}M_1^2M_2} \wedge \frac{1}{6}\right] \\
   &\quad \leq C''_q(1/t)^{2q}  
  }
for any $t>1$.
Together with the equation $E_\mu[\tau^{-q}]=\int^{\infty}_0\mu[\tau<(1/t)^{1/q}]dt$, we obtain the conclusion.

\end{proof}

\noindent
{\bf Proof of Proposition~\ref{one-interval-nondeg-lemma}.}
{\colore First, we} have
\begin{eqnarray}
\langle D\mathcal{Y}_1,D\mathcal{Y}_1\rangle_{\mathfrak{H}}=\int^1_0(D_t\mathcal{Y}_1)^\top D_t\mathcal{Y}_1dt, 
\quad \langle D\mathcal{Y}_1,D\mathcal{X}_1\rangle_{\mathfrak{H}}=\int^1_0(D_t\mathcal{Y}_1)^\top D_t\mathcal{X}_1dt. \nonumber 
\end{eqnarray}

The determinant formula for a partitioned matrix (see (0.8.5.3) in Horn and Johnson~\cite{hor-joh13}) yields
$\det(\gamma_{\mathcal{X},\mathcal{Y}})=\det(\gamma_{\mathcal{X}})\det F$, where
\begin{eqnarray}
F&=&\langle D\mathcal{Y}_1,D\mathcal{Y}_1\rangle_{\mathfrak{H}}-\langle D\mathcal{Y}_1,D\mathcal{X}_1\rangle_{\mathfrak{H}}\gamma^{-1}_{\mathcal{X}}\langle D\mathcal{X}_1,D\mathcal{Y}_1\rangle_{\mathfrak{H}}. \nonumber
\end{eqnarray}
Therefore, thanks to Lemma~\ref{gamma_x_inv-lemma}, it is sufficient to show that for any $p\geq 1$,
there exists a positive constant $C'''_p$ depending only on $p$, $m_1$, $m_2$, $M_1$, and $M_2$ such that $E_\mu[|\det F|^{-p}]\leq C'''_p$.

Let $M=\gamma^{-1}_{\mathcal{X}}\int^1_0(D_s\mathcal{X}_1)^{\top}D_s\mathcal{Y}_1ds$. Then, $F$ can be rewritten as
\EQ{\label{F-eq} F=\int^1_0\bigg(D_t\mathcal{Y}_1-D_t\mathcal{X}_1M\bigg)^{\top}\bigg(D_t\mathcal{Y}_1-D_t\mathcal{X}_1M\bigg)dt.}
We also have
  \EQN{\label{DY-eq}
   D_t\mathcal{Y}_1&=\int^1_tD_t(\check{B}(s,\mathcal{Z}_s))ds
   =\int^1_t(D_t\mathcal{X}_s\partial_x\check{B}_s+D_t\mathcal{Y}_s\partial_y\check{B}_s)ds \\	
   &=A_t^\top(\mathcal{U}_t^{-1})^\top\int^1_t\mathcal{U}_s^\top{\colorc \mathcal{I}^\top}\partial_x\check{B}_sds
   +\int^1_tD_t\mathcal{Y}_s\partial_y\check{B}_sds \\
   &=:Z_{t,1}+Z_{t,2}, 
  }
where $\check{B}_s=\check{B}(s,\mathcal{Z}_s)$.
\begin{discuss}
{\colorr $D_t\mathcal{X}_1=A_t^\top(\mathcal{U}_t^{-1})^\top \mathcal{U}_1^\top (I_{m_1} \ O)^\top$.}
\end{discuss}
Let $Z_{t,0}=(1-t)A_t^\top(\mathcal{U}_t^{-1})^\top \mathcal{U}_1^\top(I_{m_1} \ O_{m_1,m_2})^\top \partial_x\check{B}_1$ and 
  \EQNN{\tau'&=\tau \wedge \bigg(1-\sup\Big\{t\in(0,1);\sup_{s\in [t,1]}\lVert D_t\check{B}_s\rVert_{{\rm op}}> (12\sqrt{6}M_1M_2(1-t))^{-1}\Big\}\bigg) \\
   &\quad \wedge \bigg(1-\sup\Big\{t\in(0,1);\sup_{s\in [t,1]}\lVert \partial_x\check{B}_s-\partial_x\check{B}_1\rVert_{{\rm op}}> (32\sqrt{3}M_1M_2)^{-1}\Big\}\bigg).
  }
Then, {\colore because} $\check{B}$ is linear growth, for any $q>0$, there exists a constant $\tilde{C}_q$ such that
  \EQN{\label{tau'-est} \mu(\tau'\leq 1/t)&\leq \mu(\tau\leq 1/t)+\mu\bigg(\sup_{s\in [0,1/t]}\sup_{u\in[s,1]}\lVert D_{1-s}\check{B}_{1-u}\rVert_{{\rm op}}>\frac{t}{12\sqrt{6}M_1M_2}\bigg) \\
   &\quad +\mu\bigg(\sup_{s\in[0,1/t]}\lVert \partial_x \check{B}_{1-s}-\partial_x\check{B}_1\rVert_{{\rm op}}>\frac{1}{32\sqrt{3}M_1M_2}\bigg) \\
   &\leq \tilde{C}_q(1/t)^{2q} 
  }
for any $t\geq 1$, which implies that $E_\mu[\tau'^{-q}]$ is finite.
\begin{discuss}
{\colorr 伊藤の公式を使うときに$\check{B}$が生で出てくるからlinear growthとグロンウォールで評価する．$D_t\check{B}_s={\colore \partial_{(x,y)}}\check{B}_sA_t^\top (\mathcal{U}_t^{-1})^\top \mathcal{U}_s^\top$より，
$\sup_{s\in [0,1/t]}\sup_{u\in[s,1]}\lVert D_{1-s}\check{B}_{1-u}\rVert_{{\rm op}}$がモーメント評価できる．}
\end{discuss}

Let $\bar{\mathcal{U}}_t={\colorc \mathcal{I}}\mathcal{U}_1\mathcal{U}_t^{-1}{\colorc \mathcal{I}^\top}$ and
\begin{equation*}
F_0=\int^1_{1-\tau'}(Z_{t,0}-D_t\mathcal{X}_1M)^\top (Z_{t,0}-D_t\mathcal{X}_1M)dt.
\end{equation*}
By {\colore the} matrix inequality
\begin{equation}\label{mat-eq1}
-A^\top A-B^\top B\leq A^\top B+B^\top A\leq A^\top A +B^\top B
\end{equation}
\begin{discuss}
{\colorr $A^\top A +B^\top B-A^\top B-B^\top A=(A-B)^\top(A-B)\geq 0$など}
\end{discuss}
for matrices $A$ and $B$ of the same size, we have
\begin{equation}\label{mat-eq2}
(C-M)^{\top}(C-M)+(D-M)^{\top}(D-M)\geq (C-D)^{\top}(C-D)/2
\end{equation}
for matrices $C$ and $D$.
\begin{discuss}
{\colorr $(A+B)^\top (A+B)\leq 2A^\top A+2B^\top B$より$A=C-M$, $B=M-D$とすればよい}
\end{discuss}

Together with the inequalities $\tilde{A}_t\tilde{A}_t^{\top} \geq (1/M_2) I_{m_1}$ and $(\partial_x\check{B}_1)^\top \partial_x\check{B}_1\geq (1/M_2)I_{m_2}$, we obtain
\EQNN{
F_0&=\int^1_{1-\tau'}{\colorc \mathfrak{B}(t)^\top} \bar{\mathcal{U}}_t\tilde{A}_t\tilde{A}_t^\top \bar{\mathcal{U}}_t^\top{\colorc \mathfrak{B}(t)}dt \\
&\geq \frac{1}{M_2}\int^1_{1-\tau'}{\colorc \mathfrak{B}(t)^\top} \bar{\mathcal{U}}_t\bar{\mathcal{U}}_t^\top{\colorc \mathfrak{B}(t)}dt \\
&\geq \frac{1}{M_2}\inf_{t\in[1-\tau',1]}(\lVert (\bar{\mathcal{U}}_t\bar{\mathcal{U}}_t^\top)^{-1}\rVert_{{\rm op}}^{-1})
\int^1_{1-\tau'}{\colorc \mathfrak{B}(t)^\top} {\colorc \mathfrak{B}(t)}dt \\
&=\frac{1}{M_2}\inf_{t\in[1-\tau',1],|x|=1}|x^\top\bar{\mathcal{U}}_t\bar{\mathcal{U}}_t^\top x| \\
&\quad\times \int^1_{1-\tau'/2}\bigg\{{\colorc \mathfrak{B}(t)^\top}{\colorc \mathfrak{B}(t)}
+{\colorc \mathfrak{B}(t-\tau'/2)^\top}{\colorc \mathfrak{B}(t-\tau'/2)}\bigg\}dt \\
&\geq \frac{1}{2M_2}\cdot \frac{4}{9}
\int^1_{1-\tau'/2}\frac{\tau'^2}{4}(\partial_x\check{B}_1)^\top\partial_x\check{B}_1dt\geq \frac{\tau'^3}{36M_2^2}I_{m_2},
}
{\colorc where $\mathfrak{B}(t)=(1-t)\partial_x\check{B}_1-M$.}
Here we used {\colore the fact} that
  \EQN{|\bar{\mathcal{U}}_t^\top x|^2
   &=|x+{\colorc \mathcal{I}}(\mathcal{U}_t^{-1})^\top(\mathcal{U}_1-\mathcal{U}_t)^\top{\colorc \mathcal{I}^\top}x|^2 \\
   &\geq\left(|x|-\lVert \mathcal{U}_t^{-1}\rVert_{{\rm op}}\lVert \mathcal{U}_1-\mathcal{U}_t\rVert_{{\rm op}}|x|\right)^2
   \geq (1-2/6)^2=4/9 
  }
for $t\in [1-\tau',1]$ and $|x|=1$.

Let $F'$ be similarly defined to $F$ by changing the interval of integration to $[1-\tau',1)$.
{\colore Because} (\ref{mat-eq1}) yields
  \EQN{\label{mat-eq3} C^{\top}C-D^{\top}D &=(C-D)^{\top}(C-D)+D^{\top}(C-D)+(C-D)^\top D \\
   &\geq (C-D)^\top (C-D)-2(C-D)^{\top}(C-D)-D^{\top}D/2  \\
   &=-(C-D)^{\top}(C-D)-D^{\top}D/2 
  }
for matrices $C$ and $D$ of the same size, together with (\ref{F-eq}), (\ref{DY-eq}), and (\ref{mat-eq1}), we obtain
  \EQNN{F'-F_0&\geq -\frac{1}{2}F_0-\int^1_{1-\tau'}(Z_{t,1}+Z_{t,2}-Z_{t,0})^\top(Z_{t,1}+Z_{t,2}-Z_{t,0})dt \\
   &\geq -\frac{1}{2}F_0 -2\int^1_{1-\tau'}(Z_{t,1}-Z_{t,0})^\top (Z_{t,1}-Z_{t,0})dt-2\int^1_{1-\tau'}Z_{t,2}^\top Z_{t,2}dt.
  }
{\colore Because} 
\begin{equation*}
Z_{t,1}-Z_{t,0}
=A_t^\top(\mathcal{U}_t^{-1})^\top\int^1_t\Big\{(\mathcal{U}_s-\mathcal{U}_1)^\top{\colorc \mathcal{I}^\top}\partial_x\check{B}_s
+\mathcal{U}_1^\top{\colorc \mathcal{I}^\top}(\partial_x\check{B}_s-\partial_x\check{B}_1)\Big\}ds,
\end{equation*}
we have
  \EQNN{&F'-F_0/2 \\
   &\quad \geq  -4M_1^4\tau'^3\sup_{t\in[1-\tau',1]}(\lVert \mathcal{U}_t^{-1}\rVert^2_{{\rm op}})
   \sup_{t\in[1-\tau',1]}(\lVert \mathcal{U}_t-\mathcal{U}_1\rVert^2_{{\rm op}})I_{m_2} \\
   &\quad \quad -4M_1^2\tau'^3\sup_{t\in[1-\tau',1]}\lVert \mathcal{U}_1\mathcal{U}_t^{-1}\rVert_{{\rm op}}^2
   \sup_{t\in[1-\tau',1]}\lVert \partial_x\check{B}_t-\partial_x\check{B}_1\rVert_{{\rm op}}^2I_{m_2} \\
   &\quad \quad -2M_1^2\tau'^3\sup_{1-\tau'\leq t\leq s\leq 1}\lVert D_t\mathcal{Y}_s\rVert_{{\rm op}}^2I_{m_2} \\
   &\quad \geq  -16M_1^4\tau'^3\lVert \mathcal{U}_1^{-1}\rVert_{{\rm op}}^2\sup_{t\in [1-\tau',1]}\lVert \mathcal{U}_t-\mathcal{U}_1\rVert_{{\rm op}}^2I_{m_2} \\
   &\quad \quad -\frac{64}{9}M_1^2\tau'^3\sup_{t\in [1-\tau',1]}\lVert \partial_x \check{B}_t-\partial_x\check{B}_1\rVert_{{\rm op}}^2 I_{m_2} \\
   &\qquad -2M_1^2\tau'^3\sup_{1-\tau'\leq t\leq s\leq 1}\{(1-t)^2\lVert D_t\check{B}_s\rVert_{{\rm op}}^2\}I_{m_2} \\
   &\quad \geq  -\frac{\tau'^3}{432M_2^2}I_{m_2}-\frac{\tau'^3}{432M_2^2}I_{m_2}-\frac{\tau'^3}{432M_2^2}I_{m_2}. 
  }
Here we used (\ref{Ut-norm}) and {\colore the fact} that 
$\lVert \mathcal{U}_1\mathcal{U}_t^{-1}\rVert_{{\rm op}}\leq 1+\lVert \mathcal{U}_1-\mathcal{U}_t\rVert_{{\rm op}}\lVert \mathcal{U}_t^{-1}\rVert_{{\rm op}}\leq 4/3$
for $t\in[1-\tau',1]$.
Therefore, we conclude that
\begin{equation*}
\det F\geq \det F'\geq \det\bigg(\frac{1}{2}F_0-\frac{\tau'^3}{144M_2^2}I_m\bigg) \geq \bigg(\frac{\tau'^3}{144M_2^2}\bigg)^{m_2}.
\end{equation*}

{\colorc If further (\ref{nondeg-prop-further}) is satisfied, {\colore then} the upper bound in (\ref{tau'-est}) {\colore depends on neither} $x_0$ nor $y_0$
{\colore because} 
  \EQQ{\partial_x\check{B}_t-\partial_x\check{B}_s=\int^t_s\sum_i\partial_x\partial_{y_i}\check{B}_u[\check{B}_u]_idu
   +({\rm terms \ with \ bounded \ moments})}
by It${\rm \hat{o}}$'s formula.
}

\qed
}

{\colorn
\subsection{The Malliavin matrix of block observations}

Let $\gamma_l$ be the Malliavin matrix of $((\mathcal{X}_j,\mathcal{Y}_j))_{j=1}^l$.
\begin{proposition}\label{block-nondeg-lemma}
Assume the conditions of Proposition~\ref{one-interval-nondeg-lemma}.
Let $C_p$ be the one in Proposition~\ref{one-interval-nondeg-lemma}.
Then, $\gamma_L$ is positive definite almost surely, and $E_\mu[|\det \gamma_L|^{-p}]\leq C_{pL}$
for any $p\geq 1$. 
\end{proposition} 

\begin{proof}
We may assume that $L\geq 2$. Let $2\leq l\leq L$. {\colore Because} (\ref{DZ-eq}) implies 
\begin{equation*}
(D_t\mathcal{Z}_l)^\top =\mathcal{U}_l\mathcal{U}_t^{-1}A_t=\mathcal{U}_l{\colorc \mathcal{U}_{l-1}^{-1}}(D_t\mathcal{Z}_{l-1})^\top
\end{equation*}
for $t\leq l-1$, we have
\begin{equation*}
\langle D\mathcal{Z}_j,D\mathcal{Z}_l\rangle_{\mathfrak{H}}
=\langle D\mathcal{Z}_j,D\mathcal{Z}_{l-1}\rangle_{\mathfrak{H}}(\mathcal{U}_{l-1}^{-1})^\top\mathcal{U}_l^\top
\end{equation*}
for $j\leq l-1$, and
  \EQNN{\langle D\mathcal{Z}_l,D\mathcal{Z}_l\rangle_{\mathfrak{H}}
   &=\int^l_0(D_t\mathcal{Z}_l)^\top D_t\mathcal{Z}_ldt \\
   &=\langle D\mathcal{Z}_l,D\mathcal{Z}_{l-1}\rangle_{\mathfrak{H}}(\mathcal{U}_{l-1}^{-1})^\top\mathcal{U}_l^\top
   +\int^l_{l-1}(D_t\mathcal{Z}_l)^\top D_t\mathcal{Z}_ldt. 
  }

Then, by setting $\tilde{\gamma}_l=\int^l_{l-1}(D_t\mathcal{Z}_l)^\top D_t\mathcal{Z}_ldt$, we have $\det \gamma_l=\det \gamma_{l-1} \det \tilde{\gamma}_l$
{\colore because} 
\begin{equation*}
\left(\begin{array}{c}
\langle D\mathcal{Z}_1,D\mathcal{Z}_{l-1}\rangle_{\mathfrak{H}}(\mathcal{U}_{l-1}^{-1})^\top\mathcal{U}_l^\top \\
\cdots \\
\langle D\mathcal{Z}_l,D\mathcal{Z}_{l-1}\rangle_{\mathfrak{H}}(\mathcal{U}_{l-1}^{-1})^\top\mathcal{U}_l^\top
\end{array}
\right)
\end{equation*}
is a linear combination of $\{([\gamma_l]_{ij})_{1\leq i\leq (m_1+m_2)l}\}_{(m_1+m_2)(l-2)<(m_1+m_2)(l-1)}$.
Therefore, Proposition~\ref{one-interval-nondeg-lemma} implies that 
\begin{equation*}
E_\mu[|\det \gamma_L|^{-p}]=E_\mu\bigg[\bigg|\prod_{l=1}^L\det \tilde{\gamma}_l\bigg|^{-p}\bigg]
\leq \prod_{l=1}^LE_\mu[|\det \tilde{\gamma}_l|^{-pL}]\leq C_{pL}.
\end{equation*}
\end{proof}
}

{\colora
\section{An auxiliary lemma related to partitioned matrices}

\begin{lemma}\label{block-mat-lem}
Let $A_1$, $A_2$, $B$, and $C$ be matrices of suitable size so that 
\EQQ{\MAT{cc}{A_i & B \\ B^\top & C}}
is a partitioned matrix for $i=1,2$. 
Assume that $A_1$ and $C-B^\top A_1^{-1}B$ are invertible.
Then {\colore we have}
  \EQN{\label{block-mat-eq} &A_1^{-1}(A_2 \ B)\MAT{cc}{A_1 & B \\ B^\top & C}^{-1}\MAT{c}{A_2 \\ B^\top} \\
   &\quad =(A_1^{-1}A_2)^2+A_1^{-1}(A_2A_1^{-1}-I)B(C-B^\top A_1^{-1}B)^{-1}B^\top (A_1^{-1}A_2-I).}
In particular, the right-hand side of (\ref{block-mat-eq}) is equal to the unit matrix if $A_1=A_2$.
\end{lemma}
\begin{proof}
A simple calculation yields the conclusion by using (0.8.5.6) in Horn and Johnson~\cite{hor-joh13}.
\end{proof}
}

\end{document}